\documentclass[letterpaper,oneside,10pt]{article}
\usepackage[USenglish]{babel} 
\usepackage[ansinew]{inputenc}
\usepackage{verbatim}
\usepackage{amsmath, amsthm, amssymb, amsfonts}
\usepackage[USenglish]{babel} 
\usepackage{verbatim}
\usepackage[pdftex]{graphicx}
\usepackage{color}
\usepackage{float}
\usepackage{latexsym}
\usepackage{amsmath}
\usepackage{amsfonts}
\usepackage{enumerate}
\usepackage{url}

\renewcommand{\Re}{\ensuremath{\mathbb{R}}}
\begin{document}
\newcommand{\sg}{\mathbf{s}_g}
\newcommand{\bg}{\mathbf{b}_g}
\newcommand{\sa}{\mathbf{s}_a}
\newcommand{\ba}{\mathbf{b}_a}
\newcommand{\sgh}{\hat{\mathbf{s}}_g}
\newcommand{\bgh}{\hat{\mathbf{b}}_g}
\newcommand{\sah}{\hat{\mathbf{s}}_a}
\newcommand{\bah}{\hat{\mathbf{b}}_a}
\newcommand{\mw}{\left\|\omega\right\|}
\newcommand{\dd}{\partial}
\newtheorem{ass}{Assumption}
\newtheorem{thm}{Theorem}
\newtheorem{rmk}{Remark}
\newtheorem{lem}{Lemma}
\newtheorem{deff}{Definition}

\pagestyle{plain}

\title{\textbf{Model Reference Adaptive Control of \\ Systems with Gain Scheduled Reference Models}}

\author{Mehrdad Pakmehr \footnote{Postdoctoral fellow at the School of Aerospace Engineering, Georgia Institute of Technology, Atlanta, GA 30332, {\small mehrdad.pakmehr@gatech.edu.}},
Tansel Yucelen \footnote{Assistant Professor at the Mechanical and Aerospace Engineering Department, Missouri University of Science and Technology, Rolla, Mo 65409, {\small yucelen@mst.edu.}}      }

\date{\null}

\maketitle
\pagestyle{plain} 

\section*{Abstract}

Firstly, a new state feedback model reference adaptive control approach is developed for uncertain systems with gain scheduled reference models in a multi-input multi-output (MIMO) setting. Specifically, adaptive state feedback for output tracking control problem of MIMO nonlinear systems is studied and gain scheduled reference model system is used for generating desired state trajectories. Using convex optimization tools, a common Lyapunov matrix is computed for multiple linearizations near equilibrium and non-equilibrium points of the nonlinear closed loop gain scheduled reference system. This approach guarantees stability of the closed-loop gain scheduled system. Adaptive state feedback control scheme is then developed, and its stability is proven. The resulting closed-loop system is shown to have bounded solutions with bounded tracking error, with the proposed stable gain scheduled reference model. Secondly, the developed  control approach is improved for systems with constraints on the control inputs. The resulting closed-loop system is shown to have bounded solutions with bounded tracking error. Sufficient conditions for ultimate boundedness of the closed-loop system are derived. A semi-global stability result is proved with respect to the level of saturation for open-loop unstable plants while the stability result is shown to be global for open-loop stable plants. Thirdly, a decentralized adaptive state feedback control architecture is developed and its stability is proved. Specifically, the resulting closed-loop system is shown to have bounded solutions with bounded tracking error for all the subsystems with the proposed stable gain scheduled reference model. Simulation results are presented for each control architecture.

\section{Mathematical Preliminaries}

\subsection{Projection Operator}
The definitions and lemmas presented here are mainly adopted from \cite{caltechnote-lavretsky-2010, projection-lavretsky-2012, adaptive-pomet-1992}.

\begin{deff} \label{gsa-def2}
 Consider a convex compact set with a smooth boundary
\begin{equation}\label{eqn_appA1}
\displaystyle \Omega_c=\{ \theta \in R^n | f(\theta) \leq c \}, ~ 0 \leq c \leq 1,
\end{equation}
where $f: \Re^n \rightarrow \Re $ is a smooth convex function defined as
\begin{equation}\label{eqn_appA2}
\displaystyle f(\theta)=\frac{\theta^\mathsf{T} \theta - \theta^2_{max}}{\epsilon_{\theta} \theta^2_{max}},
\end{equation}
where $\theta_{max}$ is the norm bound imposed on the parameter vector $\theta$, and $\epsilon_{\theta}$ denotes the convergence tolerance of our choice. Let the true value of the parameter $\theta$, denoted by $\theta^*$, belong to $\Omega_0$, i.e. $\theta^* \in \Omega_0$, the projection operator for two vectors $\theta, y \in \Re^n$ is defined as
\begin{equation} \label{eqn_appA3}
\begin{array}{l}
\displaystyle \mathrm{Proj}(\theta, y) =  \left\{
   \begin{array}{ll}
       y-\frac{\triangledown f}{||\triangledown f||} \langle \frac{\triangledown f ^\mathsf{T} }{||\triangledown f||}, y \rangle f(\theta), & ~\mbox{if} ~f(\theta) > 0 \wedge \triangledown f ^\mathsf{T}y > 0,\\
       y,         & ~\mbox{otherwise},
   \end{array}   \right.
\end{array}
\end{equation}
where $\triangledown f(\theta)= \left( \frac{\partial f(\theta)}{\partial \theta_1},...,  \frac{\partial f(\theta)}{\partial \theta_n}  \right) \in \Re^n$ is the gradient vector of $f$ evaluated at $\theta$ and it is computed as
\begin{equation}\label{eqn_appA30}
\displaystyle \triangledown f(\theta)=\frac{2 \theta^\mathsf{T}}{\epsilon_{\theta} \theta^2_{max}},
\end{equation}
\end{deff}
Figure \ref{Proj_op} illustrates the projection operator.
\begin{figure}[!ht]
\centering
\includegraphics[width=0.8\textwidth]{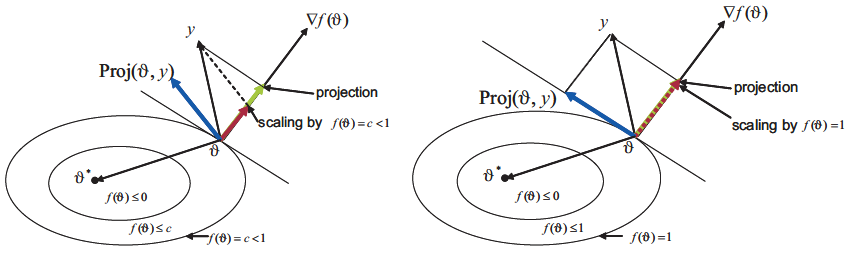}
\caption{Illustration of the projection operator \cite{AdaptiveCourse-hovakimyan-2009}.}\label{Proj_op}
\end{figure}

\begin{lem} \label{gsa-lem1}
One important property of the projection operator follows. Given $\theta^* \in \Omega_0$,
\begin{equation}\label{eqn_appA4}
\displaystyle  (\theta-\theta^*)^\mathsf{T} (\mathrm{Proj}(\theta, y)-y) \leq 0.
\end{equation}
\end{lem}

\begin{proof}
Note that $ (\theta-\theta^*)^\mathsf{T} (\mathrm{Proj}(\theta, y)-y) = (\theta^* - \theta)^\mathsf{T} (y-\mathrm{Proj}(\theta, y)) $. For $f(\theta)>0$ and $\triangledown f ^\mathsf{T} y > 0$, the left-hand side of inequality (\ref{eqn_appA4}) is
\begin{equation}\label{eqn_appA5}
\displaystyle  (\theta^* - \theta)^\mathsf{T} \left( y- \left( y- \frac{\triangledown f(\theta)(\triangledown f(\theta))^\mathsf{T} }{||\triangledown f(\theta)||^2} \right) \right),
\end{equation}
Since $\theta^* \in \Omega_0$ and due to the convexity of $f(\theta)$, we have $(\theta^* - \theta)^\mathsf{T} \triangledown f(\theta) \leq 0$. Hence
\begin{equation}\label{eqn_appA6}
\displaystyle \frac{(\theta^* - \theta)^\mathsf{T} \triangledown f(\theta) (\triangledown f(\theta))^\mathsf{T} y}{||\triangledown f(\theta)||^2} \leq 0,
\end{equation}
otherwise $\mathrm{Proj}(\theta, y)=y$.  
\end{proof}

\begin{deff} \label{gsa-def3}
The general form of the projection operator is the $n \times m$ matrix extension of the vector definition (\ref{gsa-def2}).
\begin{equation}\label{eqn_appA7}
\displaystyle \mathrm{Proj}(\Theta, Y) = [ \mathrm{Proj}(\theta_1, y_1), ..., \mathrm{Proj}(\theta_m, y_m) ],
\end{equation}
where $\Theta=[\theta_1 ... \theta_m ] \in \mathbb{R}^{n \times m}, Y=[y_1 ... y_m] \in \mathbb{R}^{n \times m}$, and $F=[f_1(\theta_1) ... f_m(\theta_m)]^\mathsf{T} \in \mathbb{R}^{m \times 1}$, then using definition (\ref{gsa-def2}) we have
\begin{equation} \label{eqn_appA8}
\begin{array}{l}
\displaystyle \mathrm{Proj}(\theta_j, y_j) =  \left\{
   \begin{array}{ll}
       y_j-\frac{\triangledown f_j}{||\triangledown f_j||} \langle \frac{\triangledown f_j ^\mathsf{T} }{||\triangledown f_j||}, y_j \rangle f_j(\theta_j), & ~\mathrm{if} ~f_j(\theta_j) > 0 \wedge \triangledown f_j ^\mathsf{T} y_j > 0,\\
       y_j,  & ~\mathrm{otherwise},
   \end{array}   \right.
\end{array}
\end{equation}
for $j=1 \text{ to } m$.
\end{deff}

\begin{lem} \label{gsa-lem2}
Let $F=[ f_1(\theta_1) ... f_m(\theta_m) ]^\mathsf{T} \in \mathbb{R}^{m \times 1}$ be a convex vector function and $\Theta=[ \theta_1 ... \theta_m ], \Theta^*=[ \theta^*_1 ... \theta^*_m ], Y=[ y_1 ... y_m ]$, where $\Theta, \Theta^*,Y \in \mathbb{R}^{n \times m}$ then,
\begin{equation}\label{eqn_appA9}
\begin{array}{l}
\displaystyle \mathrm{trace} \left\{ (\Theta-\Theta^*)^\mathsf{T} (\mathrm{Proj}(\Theta, Y)-Y) \right\} \leq 0.
\end{array}
\end{equation}
\end{lem}

\begin{proof}
Using (\ref{eqn_appA4})
\begin{equation}\label{eqn_appA10}
\begin{array}{l}
\displaystyle  \mathrm{trace} \left\{ (\Theta-\Theta^*)^\mathsf{T} (\mathrm{Proj}(\Theta, Y)-Y) \right\} 
   = \sum_{j=1}^m (\theta_j-\theta_j^*)^\mathsf{T} (\mathrm{Proj}(\theta_j, y_j)-y_j) \leq 0.
\end{array}
\end{equation}     
\end{proof}

\begin{lem} \label{gsa-lem3}
If an initial value problem, such as adaptive control algorithm with adaptive law and initial conditions, is defined by
\begin{enumerate}
  \item $\dot{\theta}=\mathrm{Proj}(\theta, y)$;
  \item $\theta(t=0)=\theta_0 \in \Omega_1$;
  \item $f(\theta): \mathbb{R}^m \rightarrow \mathbb{R}$ is convex.
\end{enumerate}
Then $\theta(t) \in \Omega_1 \forall t \geq 0$.
\end{lem}

\begin{proof}
Taking the derivative of the convex function
\begin{equation}\label{eqn_appA11}
\displaystyle \dot{f}(\theta)=(\triangledown f(\theta))^\mathsf{T} \dot{\theta}=(\triangledown f(\theta))^\mathsf{T} \mathrm{Proj}(\theta,y).
\end{equation}
Substituting (\ref{eqn_appA11}) into (\ref{eqn_appA3}) leads to
\begin{equation}\label{eqn_appA12}
\begin{array}{l}
\displaystyle    \dot{f}(\theta)= (\triangledown f(\theta))^\mathsf{T} \mathrm{Proj}(\theta,y)
    = \left\{
   \begin{array}{ll}
\displaystyle       (\triangledown f(\theta))^\mathsf{T} y (1- f(\theta)), & ~~\mathrm{if} ~f(\theta) > 0 \wedge \triangledown f ^\mathsf{T}y > 0,\\
\displaystyle       (\triangledown f(\theta))^\mathsf{T} y,                & ~~\mathrm{otherwise},
   \end{array}   \right.
\end{array}
\end{equation}
therefore
\begin{equation}\label{eqn_appA13}
\begin{array}{l}
\displaystyle    \left\{
   \begin{array}{lll}
       \dot{f}(\theta) > 0,  & ~~\mathrm{if} ~ 0<f(\theta)<1 \wedge \triangledown f ^\mathsf{T}y > 0,\\
       \dot{f}(\theta) = 0,  & ~~\mathrm{if} ~ f(\theta)=1 \wedge  \triangledown f ^\mathsf{T}y > 0,\\
       \dot{f}(\theta) < 0,  & ~~\mathrm{if} ~ f(\theta) \leq 0 \vee \triangledown f ^\mathsf{T}y \leq 0.
   \end{array}   \right.
\end{array}
\end{equation}
Thus $f(\theta_0) \leq 1 \Rightarrow f(\theta(t)) \leq 1$ for all $t \geq 0$, hence $\theta(t) \in \Omega_1$ for all $t \geq 0$.  
\end{proof}

\begin{deff} \label{gsa-def10}
\cite{robustbook-ionnou-1996, projection-lavretsky-2012} A variant of the projection algorithm, $\Gamma$-projection, updates the parameter along a symmetric positive definite gain $\Gamma$ as defined below
\begin{equation} \label{eqn_appA14}
\begin{array}{l}
\displaystyle \mathrm{Proj}_{\Gamma}(\theta, y) =  \left\{
   \begin{array}{ll}
       \Gamma y- \Gamma \frac{\triangledown f(\theta) (\triangledown f(\theta))^\mathsf{T}}{(\triangledown f(\theta))^\mathsf{T} \Gamma \triangledown f(\theta)} \Gamma y f(\theta), & ~\mathrm{if} ~f(\theta)>0 \wedge \triangledown f^\mathsf{T}\Gamma y>0,\\
       \Gamma y,  & ~\mathrm{otherwise.}
   \end{array}   \right.
\end{array}
\end{equation}
\end{deff}

\begin{lem} \label{gsa-lem10}
Given $\theta^* \in \Omega_0$, then
\begin{equation}\label{eqn_appA15}
\displaystyle  (\theta-\theta^*)^\mathsf{T} (\Gamma^{-1} \mathrm{Proj}_{\Gamma}(\theta, y)-y) \leq 0.
\end{equation}
\end{lem}

\begin{proof}
Note that $ (\theta-\theta^*)^\mathsf{T} (\Gamma^{-1} \mathrm{Proj}_{\Gamma}(\theta, y)-y) = (\theta^* - \theta)^\mathsf{T} (y- \Gamma^{-1} \mathrm{Proj}_{\Gamma}(\theta, y)) $. For $f(\theta)>0$ and $\triangledown f^\mathsf{T} \Gamma y > 0$, the left-hand side of inequality (\ref{eqn_appA14}) is
\begin{equation}\label{eqn_appA16}
\displaystyle  (\theta^* - \theta)^\mathsf{T} \left( y- \Gamma^{-1} \left( \Gamma y- \Gamma \frac{\triangledown f(\theta)(\triangledown f(\theta))^\mathsf{T} }{(\triangledown f(\theta))^\mathsf{T} \Gamma \triangledown f(\theta)} \Gamma y f(\theta)  \right) \right).
\end{equation}
Since $\theta^* \in \Omega_0$ and due to the convexity of $f(\theta)$, we have $(\theta^* - \theta)^\mathsf{T} \triangledown f(\theta) \leq 0$. Hence
\begin{equation}\label{eqn_appA17}
\displaystyle \frac{(\theta^* - \theta)^\mathsf{T} \triangledown f(\theta) (\triangledown f(\theta))^\mathsf{T} \Gamma y}{(\triangledown f(\theta))^\mathsf{T} \Gamma \triangledown f(\theta)} f(\theta) \leq 0,
\end{equation}
otherwise $\mathrm{Proj}(\theta, y)=\Gamma y$.  
\end{proof}

\begin{lem} \label{gsa-lem11}
Let $\mathrm{Proj}_{\Gamma}(\Theta, Y)$ be defined similar to Definition \ref{gsa-def3}, $F=[ f_1(\theta_1) ... f_m(\theta_m) ]^\mathsf{T} \in \mathbb{R}^{m \times 1}$ be a convex vector function and $\Theta=[ \theta_1 ... \theta_m ], \Theta^*=[ \theta^*_1 ... \theta^*_m ], Y=[ y_1 ... y_m ]$, where $\Theta, \Theta^*,Y \in \mathbb{R}^{n \times m}$ then,
\begin{equation}\label{eqn_appA18}
\begin{array}{l}
\displaystyle \mathrm{trace} \left\{ (\Theta-\Theta^*)^\mathsf{T} (\Gamma^{-1} \mathrm{Proj}_{\Gamma}(\Theta, Y)-Y) \right\} \leq 0.
\end{array}
\end{equation}
\end{lem}

\begin{proof}
The proof is similar to the proof of Lemma \ref{gsa-lem2}.    
\end{proof}

\subsection{Rectangular Saturation Function}

The definitions in this section are adopted from \cite{PhDThesis-jang-2009, MsThesis-schwager-2005}.

The constraints on the control inputs will be defined as a rectangular saturation function of $v$. The saturation function is given by $R_s(v)$, where the elements of $R_s$ are defined by
\begin{equation}\label{eqn_ags72}
\begin{array}{c}
\displaystyle R_{S_i}=\mathrm{sat}(v_i)=\left\{
       \begin{array}{ll}
           v_i,                 & \mathrm{if} ~|v_i| \leq v_{i, \max}, ~ i=1,...,m, \\
           v_{i, \max} \mathrm{sgn}(v_i), & \mathrm{if} ~|v_i| > v_{i, \max}.
       \end{array}
    \right.
\end{array}
\end{equation}
This saturation function can be expressed as the sum of a direction preserving component and an error component, so that
\begin{equation}\label{eqn_ags73}
\begin{array}{c}
\displaystyle R_{S}=\mathrm{sat}(v)=\left\{
       \begin{array}{ll}
           v,                       & \mathrm{if} ~||v|| \leq h(v), \\
           \bar{v}=v_{d}+\tilde{v}, & \mathrm{if} ~||v|| > h(v),
       \end{array}
    \right.
\end{array}
\end{equation}
where $v_d=\hat{e}h(v)$. $\hat{e}=v/||v||$ is the unit vector in the direction of $v$, and $h(v)$ returns the magnitude of the projection of $v$ onto the hyper-rectangle. In this formulation $v_d$ is in the same direction as $v$ and $\tilde{v}$ is an error vector. Figure \ref{Rectang_sat_fcn} illustrates the nature of $R_s$ for the case where $m=2$. It can be shown that $\tilde{v}$ is a bounded vector.
\begin{figure}[!ht]
\centering
\includegraphics[width=0.3\textwidth]{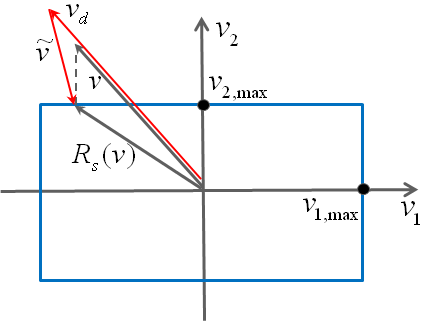}
\caption{The control input $\bar{v}$, saturated by rectangular saturation can be decomposed into $v_{d}$ and $\tilde{v}$ \cite{PhDThesis-jang-2009}. }\label{Rectang_sat_fcn}
\end{figure}

\begin{deff} \label{gsa-def4}
The function $R_s(.)$, is a multi-dimensional rectangular saturation function defined by
\begin{equation}\label{eqn_ags74}
 \begin{array}{c}
\displaystyle R_s(v) =  \left[
   \begin{array}{l}
        v_{1, \max} \mathrm{sat}\left(\frac{v_1}{v_{1, \max}} \right)\\
        .\\
        .\\
        v_{m, \max} \mathrm{sat}\left(\frac{v_m}{v_{m, \max}} \right)
   \end{array}   \right],
\end{array}
\end{equation}
where $\mathrm{sat}(.)$ for all $x \in \Re$ is given by
\begin{equation}\label{eqn_ags75}
\begin{array}{c}
\displaystyle \mathrm{sat}(x)=\left\{
       \begin{array}{ll}
           x,       & \mathrm{if} ~|x| \leq 1, \\
           \mathrm{sgn}(x),  & \mathrm{if} ~|x|  >   1.
       \end{array}
    \right.
\end{array}
\end{equation}
\end{deff}
Despite the advantages of $R_s(v)$, the direction of $R_s(v)$ is not necessarily consistent with that of $v$, which causes additional complexities in the stability analysis.

\section{Linear Parameter Dependent Reference Model Design}
Consider the nonlinear dynamical system
\begin{equation}\label{eqn_gs1}
\begin{array}{c}
\dot{x}^p(t)= f^p(x^p(t),u(t)),\\[5pt]
y(t)=g^p(x^p(t),u(t)),
\end{array}
\end{equation}
where $x^p(t) \in \Re^n$ is the state vector, $u(t) \in \Re^m$ is the control input vector, $y(t) \in \Re^m$ is the output vector, $f^p(.)$ is an $n$-dimensional differentiable nonlinear vector function which represents the plant dynamics, and $g^p(.)$ is an $m$-dimensional differentiable nonlinear vector function which generates the plant outputs. We intend to design a feedback control such that $y(t)$ properly tracks a reference signal $r(t)$ as $t$ goes to infinity, where $r(t) \in D_r \subset \Re^m$, and $D_r$ is a compact set. For each $r \in D_r$, there is a unique pair $(x^p_e, u_e)$ that depends continuously on $r$ and satisfies the equations
\begin{equation}\label{eqn_gs2}
\begin{array}{c}
0= f^p(x^p_e,u_e),\\[3pt]
r=g^p(x^p_e,u_e),
\end{array}
\end{equation}
where $x^p_e$ is the desired equilibrium point and $u_e$ is the steady-state control that is needed to maintain equilibrium at $x^p_e$.  It is often useful to parameterize the family of system equilibria as follows:
\begin{deff} \label{def1}
The functions $x^p_e(\alpha(t)), u_e(\alpha(t))$, and $r_e(\alpha(t))$ define an equilibrium family for the plant (\ref{eqn_gs1}) on the set $\Omega$ if
\begin{equation}\label{eqn_gs222}
\begin{array}{l}
f^p(x^p_e(\alpha(t)),u_e(\alpha(t)))=0, \\[5pt]
g^p(x^p_e(\alpha(t)),u_e(\alpha(t)))=r_e(\alpha(t)),  ~\alpha \in \Omega.
\end{array}
\end{equation}
\end{deff}
\medskip
The family of plant linear models, for all $\alpha \in \Omega$ can be written as
\begin{equation}\label{eqn_gs5}
\begin{array}{l}
\delta \dot{x}^p(t) = A^p(\alpha(t)) \delta x^p(t) + B^p(\alpha(t)) \delta u(t),\\[5pt]
\delta y(t) = C^p(\alpha(t)) \delta x^p(t) + D^p(\alpha(t)) \delta u(t),
\end{array}
\end{equation}
where $\delta x^p(t) = x^p(t)-x^p_e(\alpha(t))$, $\delta y(t) = y(t)-y_e(\alpha(t))$, and $\delta u(t) = u(t)-u_e(\alpha(t))$.
$A^p(\alpha(t)),~B^p(\alpha(t)),~C^p(\alpha(t))$, and $D^p(\alpha(t))$ are the parameterized plant linearization family matrices and $x^p_e(\alpha(t)),~u_e(\alpha(t))$, and $y_e(\alpha(t))$ are the parameterized steady-state variables for the states, inputs and outputs of the plant, which form the equilibrium manifold of plant (\ref{eqn_gs1}). The subscript "$e$" stands for "steady-state" throughout this manuscript. The parameter $\alpha(t)$ is called the scheduling variable and should be measurable in real time. $\alpha(t)$ is a function of endogenous variables (i.e., depending on the plant states). Here, we defined the scheduling parameter to be the Euclidean norm of the output vector ($\alpha(t)=||y(t)||$). In order to make the design process easier \cite{GSstability-pakmehr-2013, GainSchedStabConf-pakmehr-2013}, we control the system via filtered inputs, rather than the input themselves, so there is no need for equilibrium control value other than zero (i.e. $x^c_e (\alpha(t))=0, v_e (\alpha(t))=0, \forall \alpha$). The filter is defined as $\frac{u(s)}{v(s)}=\frac{\eta_c}{s+\eta_c}$, where $\eta_c > 0$. The plant (\ref{eqn_gs1}) with the filtered inputs, and its general controller can be written as
\begin{equation} \label{eqn_gs111}
\begin{array}{l}
  \underbrace{\left[
       \begin{array}{c}
           \dot{x}^p(t) \\
           \dot{u}(t) \\
           \dot{x}^c(t)
       \end{array}
    \right]}_{\dot{x}} =
     \underbrace{\left[
       \begin{array}{c}
           f^p(x^p(t),u(t)) \\
           -\eta_c u(t)  \\
           f^c(x^c(t),g^p(x^p(t),u(t)) ,r(t))
       \end{array}
    \right]}_{f(x(t),r(t))}
     +  \underbrace{\left[
       \begin{array}{c}
           0 \\
           \eta_c I \\
           0
       \end{array}
    \right]}_{B} v(t), \\[5pt]
v(t)=\underbrace{g^c(x^c(t),g^p(x^p(t),u(t)), r(t))}_{g(x(t),r(t))},
\end{array}
\end{equation}
and the closed-loop nonlinear system is
\begin{equation}\label{eqn_gs112}
\begin{array}{l}
\dot{x}(t)= F(x(t),r(t)),
\end{array}
\end{equation}
where $x(t) \in D_x \subset \Re^{n+2m}$, and $r(t) \in D_r \subset \Re^m$. The augmented linear family of systems for the augmented plant for all $\alpha \in \Omega$ is
\begin{equation}\label{eqn_gs62}
\begin{array}{l}
  \underbrace{\left[
       \begin{array}{c}
           \delta \dot{x}^p(t) \\
           \delta \dot{u}(t)
       \end{array}
    \right]}_{ \delta \dot{x}_{\mathrm{aug}}(t)} =
     \underbrace{\left[
       \begin{array}{cc}
           A^p(\alpha(t)) &  B^p(\alpha(t))\\
              0      &  -\eta_c I
       \end{array}
    \right]}_{A_{\mathrm{aug}}(\alpha(t))}
    \underbrace{\left[
       \begin{array}{c}
           \delta x^p(t) \\
           \delta u(t)
           \end{array}
    \right]}_{\delta x_{\mathrm{aug}}(t)} 
    + \underbrace{\left[
       \begin{array}{c}
           0 \\
           \eta_c I
           \end{array}
    \right]}_{B_{\mathrm{aug}}} \delta v(t) ,\\[3pt]
\delta y = \underbrace{[C^p(\alpha(t)), D^p(\alpha(t))]}_{C_{\mathrm{aug}}(\alpha(t))}
   \underbrace{\left[
       \begin{array}{c}
           \delta x^p(t) \\
           \delta u(t)
           \end{array}
    \right]}_{\delta x_{\mathrm{aug}}(t)},
\end{array}
\end{equation}
and the controller is defined to be
\begin{equation} \label{eqn_gs65}
\begin{array}{l}
  \left[
       \begin{array}{c}
           \dot{x}^c(t) \\
            v(t)
       \end{array}
    \right] =
     \left[
       \begin{array}{ccc}
           -\epsilon_c I & I &  -I\\
           K^{\mathsf{T}}_i(\alpha(t)) & 0 & 0
       \end{array}
    \right] ~
     \left[
       \begin{array}{c}
            x^c(t) \\
            \delta y(t) \\
            \delta r(t)
       \end{array}
    \right].
\end{array}
\end{equation}
For the case where we have $\delta y(t) = \delta x^p(t)$, (i.e. $C^p(\alpha(t))=I, D^p(\alpha(t))=0$) the linearized closed-loop system (\ref{eqn_gs62}) with controller (\ref{eqn_gs65}) becomes
\begin{equation} \label{eqn_gs22}
\begin{array}{l}
  \underbrace{\left[
       \begin{array}{c}
           \delta \dot{x}^p(t) \\
           \delta \dot{u}(t) \\
            \dot{x}^c(t)
       \end{array}
    \right]}_{\delta \dot{x}(t)} =
     \underbrace{\left[
       \begin{array}{ccc}
           A^p(\alpha(t)) &~ B^p(\alpha(t)) &~ 0 \\
           0 &~ -\eta_c I & ~ \eta_c K^{\mathsf{T}}_i(\alpha(t)) \\
           I      &~ 0  &~ -\epsilon_c I
       \end{array}
    \right]}_{A_{m}(\alpha(t))} 
     \underbrace{\left[
       \begin{array}{c}
           \delta x^p(t) \\
           \delta u(t) \\
            x^c(t)
       \end{array}
    \right]}_{\delta x(t)} +\underbrace{\left[
       \begin{array}{c}
            0 \\
            0 \\
           - I
       \end{array}
    \right]}_{B_{r}} \delta r(t), ~ \forall \alpha \in \Omega.
\end{array}
\end{equation}
where $\delta x^p(t)= x^p(t)-x^p_e(\alpha(t))$, $\delta u(t)=u(t)-u_e(\alpha(t))$, $\delta y(t)=y(t)-y_e(\alpha(t))$, and $\delta r(t)=r(t)-y_e(\alpha(t))$.
\begin{rmk}\label{rmk_21}
Using pre-designed linear controllers available for important operating points of the system, $K^{\mathsf{T}}_i(\alpha(t))$ can be obtained based on a stability preserving interpolation approach described in \cite{interp-stilwell-2000} with respect to the scheduling parameter $\alpha$ in a smooth, continuous way. An approach by which the interpolated controller stabilizes the linearized plant for all $\alpha \in \Omega$. Another approach is to compute $K^{\mathsf{T}}_i(\alpha(t))$ by polynomial approximation as a function of $\alpha$.
\end{rmk}
\subsection{Stability Analysis of the Reference Model}
\begin{ass} \label{gsa_ass1}
Matrix $A_{m}(\alpha(t))$ is bounded
\begin{equation}\label{eqn_gs51}
\begin{array}{l}
||A_{m}(\alpha(t))|| \leq k_A,  ~~ \alpha \in \Omega,
\end{array}
\end{equation}
where $k_A < \infty$ is a constant.
\end{ass}
\begin{rmk}\label{gsa_rmk_a1}
The feasibility of this assumption can be investigated by extensive numerical simulation studies of the physical system that is being investigated, using a high fidelity dynamic model. For systems such as gas turbine engines, it already has been investigated in \cite{PhDThesis-pakmehr-2013, GSstability-pakmehr-2013, GainSchedStabConf-pakmehr-2013}.
\end{rmk}
\begin{thm}\label{gsa_thm1}
\cite{PhDThesis-pakmehr-2013, GSstability-pakmehr-2013} Consider the closed-loop system (\ref{eqn_gs112}), and assume there is a family of equilibrium points $(x_e(t),r_e(t))$ such that $F(x_e(t),r_e(t))=0$. Define $A^{nl}_{m} = \frac{\partial F}{\partial x} \in \overline{S}, ~\forall x\in D_x $, where $\overline{S} := \{ A^{nl}_{m}, \forall x \in D_x \}$ is the set of linearizations of system (\ref{eqn_gs112}). Assume there exist symmetric positive definite matrices $P$ and $Q$, such that
\begin{equation}\label{eqn_gs54}
 P A^{nl}_{m}+A^{nl \mathsf{T}}_{m} P \leq -Q, ~~~ \forall A^{nl}_{cl} \in \overline{S},
\end{equation}
then the system (\ref{eqn_gs112}) is stable for all the trajectories defined by $A^{nl}_{cl} \in \overline{S}$. In other words, assuming the initial state is sufficiently close to some equilibrium, then the closed-loop system remains in a neighborhood of the equilibrium manifold for all $t \geq 0$.
\end{thm}
\begin{rmk}\label{gsa_rmk1}
In practice we cannot obtain $\overline{S}$, instead, we can linearize system (\ref{eqn_gs112}) for a large number of states $x_i$, $i=1, \ldots, L$, which we claim is sufficient to cover the set of actual operating conditions, to show the stability of the closed-loop system. Define $S:= \mathrm{Co}\{ A^{nl}_{m_1}, ..., A^{nl}_{m_L} \}$ as a matrix polytope described by its vertices, where $A^{nl}_{m_i} = \left. \frac{\partial F}{\partial x} \right|_{x=x_i} \in S$, for all $i \in \{ 1,2, ..., L \}$. Note that $A^{nl}_{m_i}$ can be obtained by linearizing the nonlinear system (\ref{eqn_gs112}) at non-equilibrium points (transient condition), and also at equilibrium points (steady state condition), which in this manuscript, are represented by $A_{m}(\alpha_i)$. Then using convex optimization tools \cite{YALMIP-lofberg-2004, sedumi-Sturm-2001}, for some matrix $Q=Q^\mathsf{T}$, we compute a single symmetric positive definite matrix $P$, such that
\begin{equation}\label{eqn_gs115}
 P A^{nl}_{m_i}+A^{nl \mathsf{T}}_{m_i} P \leq -Q, ~~ \forall i \in \{ 1,2, ..., L \}.
\end{equation}
With assumption \ref{gsa_ass1} satisfied, and the claim that $A_{m}(\alpha) \in S$, for all $\alpha \in \Omega$, then system (\ref{eqn_gs22}) is also stable for all the trajectories defined by $A_{m}(\alpha) \in S$. In the next section, we will show how to verify the above claim.
\end{rmk}
\begin{rmk}\label{rmk13}
For the purpose of stability analysis, there is a need for multiple linearizations of the closed-loop system to construct a feasible set $S$. The minimum number of required linearizations, $L$, depends on the physical system; and it changes for different dynamical systems. This knowledge usually can be obtained through an extensive numerical simulation study of the dynamical system, using a high fidelity model \cite{research-rugh-2000, feedbackControl-baumann-1986}.
\end{rmk}
\begin{lem}\label{gsa_lem1}
If matrices $P$ and $Q$ exist, such that LMI (\ref{eqn_gs115}) is satisfied, and $A_{m}(\alpha(t)) \in S$, for all $\alpha \in \Omega$, then system (\ref{eqn_gs22}) is stable.
\begin{equation}\label{eqn_gs118}
\displaystyle P A_{m}(\alpha(t)) + A^{\mathsf{T}}_{m}(\alpha(t)) P \leq -Q, ~~\forall \alpha \in \Omega.
\end{equation}
\end{lem}
\medskip
\begin{rmk}\label{gsa_rmk12}
The existence of a single matrix $P$ which guarantees the stability of a closed-loop system over some operating envelope has already been shown for dynamical systems such as gas turbine engines \cite{PhDThesis-pakmehr-2013, GSstability-pakmehr-2013, GainSchedStabConf-pakmehr-2013} and high performance aircraft \cite{PhdThesis-miotto-1997, FixedStrucA-miotto-1997, FixedStrucB-miotto-1997}. The numerical verification of the assumption in Lemma \ref{gsa_lem1}, that is the linearized plant lives in the convex hull of the linearization matrix samples, for gas turbine engines can be found in \cite{PhDThesis-pakmehr-2013, GSstability-pakmehr-2013}.
\end{rmk}

\section{Adaptive Control of Systems with Gain Scheduled Reference Models - Part I: Basic Approach}
\subsection{Introduction}

To facilitate the stability analysis of nonlinear systems with multiple equilibrium (operating) points (regions), an efficient technique is to approximate them by a linear time varying (LTV) system \cite{convergentSys-liu-1968, globalLin-liu-1969, lmi-boyd-1994}. In this section, we develop a gain scheduled reference model so that the nonlinear plant can track this reference model by using an adaptive controller. Gain scheduled reference model design and stability analysis is done using the method presented in \cite{gainsched-shamma-1988, analysis-shamma-1990, research-rugh-2000}. The scheduling variable in our reference model design process is an \textit{endogenous} parameter, which in the gas turbine engine case is a function of the gas turbine engine spool speeds.

Some of the works dedicated to the adaptive control of systems with multiple equilibrium points and with time varying reference systems are \cite{AdaptivePWstateJour-tao-2012, AdaptivePWgtmJour-tao-2012, AdaptiveGS-jang-2008, PhDThesis-jang-2009}. Adaptive control of piecewise linear systems has been developed in \cite{AdaptivePWstateJour-tao-2012, AdaptivePWgtmJour-tao-2012}. In this kind of adaptive control system, multiple linear time invariant (LTI) systems are used and transitions between these models are modeled as switches. These switchings introduce discontinuities and jumps in the control inputs. Adaptive control of time varying systems with gain scheduling is done in \cite{AdaptiveGS-jang-2008, PhDThesis-jang-2009}. The stability analysis for this system is shown using a time varying quadratic Lyapunov function, with some conditions on time varying Lyapunov matrix $P(t)$ and its rate $\dot{P}(t)$.

The contribution of this section is the development of a stable state feedback model reference adaptive control algorithm for systems with gain scheduled reference models in a MIMO setting; the approach is applicable to  systems, such as gas turbine engines, which the stability of their gain scheduled reference model is guaranteed by computing a single Lyapunov function \cite{PhDThesis-pakmehr-2013, GSstability-pakmehr-2013, GainSchedStabConf-pakmehr-2013}. A detailed stability analysis is done for the proposed state feedback adaptive control system with gain scheduled reference model, which can be used towards control software verification \cite{ContSoftware-feron-2010}.

The rest of this section is organized as follows. In subsection II, a model reference adaptive control with a gain scheduled reference model is designed with detailed stability proof. In subsection III, simulation results are shown for gain scheduled model reference adaptive control of a high fidelity physics-based nonlinear model of a JetCat SPT5 turboshaft engine. The simulations are done for two different cases including the nominal engine case, and a degraded engine due to aging. Simulation results show that the developed adaptive controller can be used effectively for the entire flight envelope of degraded turboshaft engines with guaranteed stability. Subsection IV, concludes this section.

\subsection{Model Reference Adaptive Control}

\subsubsection{Problem Formulation}
Consider linear parameter dependent dynamical system  for all $\alpha \in \Omega$
\begin{equation}\label{eqn_ags150}
\begin{array}{l}
  \underbrace{\left[
       \begin{array}{c}
           \delta \dot{x}^p(t) \\
           \delta \dot{u}(t)\\
           \dot{x}^c(t)
       \end{array}
    \right]}_{\delta \dot{x}(t)} =
     \underbrace{\left[
       \begin{array}{ccc}
           A^p(\alpha(t)) &  B^p(\alpha(t)) &  0\\
           0           &  -\eta_c I &  0\\
           C^p(\alpha(t)) &  D^p(\alpha(t)) & -\epsilon_c I
       \end{array}
    \right]}_{A(\alpha(t))}
    \underbrace{\left[
       \begin{array}{c}
           \delta x^p(t) \\
           \delta u(t) \\
                  x^c(t)
           \end{array}
    \right]}_{\delta x(t)}  + \underbrace{\left[
       \begin{array}{c}
           0 \\
           \eta_c I \\
           0
           \end{array}
    \right]}_{B} v(t) +
     \underbrace{\left[
       \begin{array}{c}
           0 \\
           0 \\
           -I
           \end{array}
    \right]}_{B_r} \delta r(t).
\end{array}
\end{equation}
For the case where $\delta y(t) = \delta x^p(t)$, i.e. $C^p(\alpha(t))=I, D^p(\alpha(t))=0$ we have
\begin{equation}\label{eqn_ags151}
A(\alpha)= \left[
       \begin{array}{ccc}
           A^p(\alpha(t)) &  B^p(\alpha(t))       &  0             \\
           0           &  -\eta_c I  &  0                   \\
           I           &  0                 & -\epsilon_c I
       \end{array}
            \right].
\end{equation}
For simplicity from now on we rename the variables $\delta x(t), \delta y(t)$ and $\delta r(t)$ as $\delta x(t) := x(t)$, $\delta y(t) := y(t)$ and $\delta r(t) := r(t)$. The plant (\ref{eqn_ags150}) can be written as
\begin{equation}\label{eqn_ags152}
\begin{array}{l}
\dot{x}(t) = A(\alpha(t)) x(t) + B v(t) + B_r r(t), ~~ \forall \alpha \in \Omega.
\end{array}
\end{equation}
The nominal control for this system is
\begin{equation}\label{eqn_ags51}
\begin{array}{l}
v_{nom}(t) = K^\mathsf{T}(\alpha(t)) x(t), ~~ \forall \alpha \in \Omega,
\end{array}
\end{equation}
where $K^\mathsf{T}(\alpha(t))=[0,~ 0,~ K_i^\mathsf{T}(\alpha(t))]$. The time-varying reference model is defined as
\begin{equation}\label{eqn_ags52}
\begin{array}{l}
\dot{x}_m(t) = A_m(\alpha(t)) x_m(t) + B_r r(t), ~~ \forall \alpha \in \Omega.
\end{array}
\end{equation}
In the previous section we showed the stability of this reference model. Note that $r(t) \in \Re^m$ is the command signal such that $||r(t)|| \leq r_{\max}$.

\begin{ass} \label{gsa-ass3}
There exists an ideal gain matrix $K^{*\mathsf{T}}(\alpha(t))=[0, 0, K_i^{*\mathsf{T}}(\alpha(t))] \in \Re^{m \times (n+2m)}$, that results in perfect matching between the reference model (\ref{eqn_ags52}) and the plant (\ref{eqn_ags152}) such that
\begin{equation}\label{eqn_ags153}
\begin{array}{l}
A_m(\alpha(t))=A(\alpha(t))+B K^{*\mathsf{T}}(\alpha(t)), ~~ \forall \alpha \in \Omega,
\end{array}
\end{equation}
where $A_m(\alpha(t)) \in \Re^{(n+2m) \times (n+2m)}$ and it is a Hurwitz matrix for all $\alpha \in \Omega$.
\end{ass}
\begin{rmk}\label{gsa_rmk_a3}
The feasibility of this assumption has already been verified in \cite{PhDThesis-pakmehr-2013, GSstability-pakmehr-2013, GainSchedStabConf-pakmehr-2013}, for gas turbine engine applications which we consider as the main application of this work.  For other systems, modeling and numerical studies are needed for such verification.  It is important to note that Assumption \ref{gsa-ass3} implies the existence of $K^*(\alpha(t))$ for all $\alpha \in \Omega$ such that (\ref{eqn_ags153}) holds, which is common in the adaptive control literature (see, for example, \cite{robustbook-ionnou-1996}) and not restrictive if the the pair $(A,B)$ are known to be controllable for all $\alpha \in \Omega$.
\end{rmk}
\begin{ass} \label{gsa-ass2}
Let $K^*(\alpha(t)) \in \Theta_K$ for all $\alpha \in \Omega$, where $\Theta_K$ is a known convex compact set. Note that $K^*(\alpha(t))=[K_1^*(\alpha(t)), ..., K_m^*(\alpha(t))]$, and $\Theta_K=\left\{ \Theta_{K_j}, j=1,...,m | \Theta_{K_j}=\{ \theta_{k_{ij}}, i=1,..., (n+2m) \} \right\}$. We also assume that $K^*(t)$ is continuously differentiable, and the derivative is uniformly bounded, $||\dot{K}^*(\alpha(t))|| \leq d_k < \infty$, and $||\dot{K}_j^*(\alpha(t))|| \leq d_{k_j} < \infty$ for all $\alpha \in \Omega$.
\end{ass}
\medskip
\begin{rmk}\label{gsa_rmk_a2}
$\alpha(t)$ is defined to be $\alpha(t)=||y(t)||=||x^p(t)||$; since it is a function of endogenous variables (i.e., the plant states), its boundedness is guaranteed by boundedness of the plant states. As a result, its derivative ($\dot{\alpha}(t)=\frac{{x^{p}(t)}^\mathsf{T} \dot{x}^p(t)}{||x^p(t)||}$) is also bounded. More details can be found in \cite{gainsched-shamma-1988, research-rugh-2000, PhDThesis-pakmehr-2013, GSstability-pakmehr-2013, overview-shamma-2012}.
\end{rmk}
\medskip
\begin{rmk}\label{gsa_rmk_a21}
Compact set $\Theta_K$ can be obtained by extensive numerical simulation studies of the system that the controller is being designed for. Smoothness, continuity, and differentiability of $K(\alpha(t))$, and also uniform boundedness of $\dot{K}(\alpha(t))$, can be guaranteed, by using proper design and computation process for $K(\alpha(t))$ (see Remark \ref{rmk_21}).
\end{rmk}

\subsubsection{Adaptive Control}

In order to force the plant to follow a gain scheduled reference model, the following adaptive controller is designed
\begin{equation}\label{eqn_ags54}
\begin{array}{l}
v(t)=v_{ad}(t) = \hat{K}^\mathsf{T}(t) x(t), ~~ \forall \alpha \in \Omega.
\end{array}
\end{equation}
Combining equations (\ref{eqn_ags54}) and (\ref{eqn_ags152}), we obtain closed-loop system
\begin{equation}\label{eqn_ags56}
\begin{array}{l}
\dot{x}(t)= A_m(\alpha(t)) x(t) + B \tilde{K}^\mathsf{T}(t) x(t) + B_r r(t), ~~ \forall \alpha \in \Omega, \\[3pt]
x(0)=x_0,
\end{array}
\end{equation}
where $\tilde{K}(t)=\hat{K}(t)-K^*(t)$. It is clear from $K^*(t)$, that a time-varying uncertainty is considered here, that represents unknown system parameters.  Defining $e(t)=x(t)-x_m(t)$, the error dynamics are
\begin{equation}\label{eqn_ags57}
\begin{array}{l}
\dot{e}(t)= A_m(\alpha(t)) e(t) + B \tilde{K}^\mathsf{T}(t) x(t), ~~ e(0)=0, ~~ \forall \alpha \in \Omega.
\end{array}
\end{equation}

With the knowledge of lower and upper bounds of the parameters $K^*(t)$, the parameter projection adaptive law is
\begin{equation}\label{eqn_ags58}
\begin{array}{l}
\dot{\hat{K}}(t)= \text{Proj}_{\Gamma} \left( \hat{K}(t), -x(t)e^\mathsf{T}(t) P B \right),  ~~ \hat{K}(0)=\hat{K}_0,
\end{array}
\end{equation}
where $\Gamma=\Gamma^\mathsf{T} >0$, $P=P^\mathsf{T} >0$ is a solution of LMI (\ref{eqn_gs118}), and $\text{Proj(.,.)}$ is the projection operator defined in Definition \ref{gsa-def3}. Note that the error dynamics for the controller gain, $\dot{\tilde{K}}(t)$, is given by
\begin{equation}\label{eqn_ags581}
\begin{array}{l}
\dot{\tilde{K}}(t)=\dot{\hat{K}}(t)-\dot{K}^*(t),  ~~ \tilde{K}(0)=\tilde{K}_0,
\end{array}
\end{equation}
Figure~\ref{fig:AdapGainSched_Cont_Struc} shows schematic diagram of the MRAC system with gain scheduled reference model.
\begin{figure}[!ht]
\centering
\includegraphics[width=0.7\textwidth]{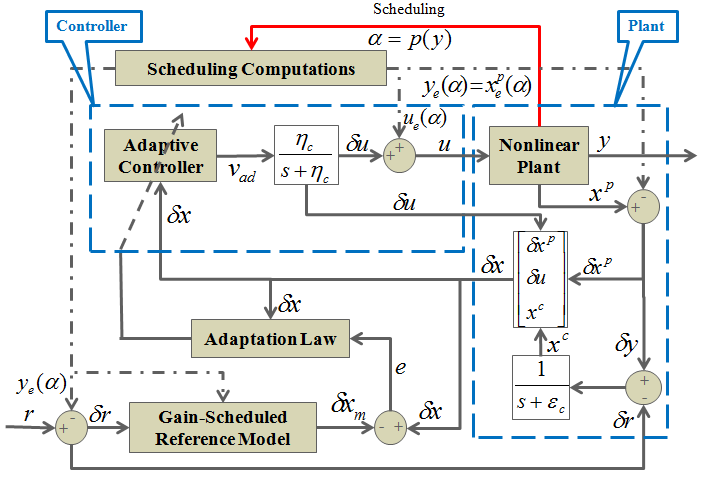}
\caption{Illustration of the model reference adaptive control with gain scheduled reference model }
\label{fig:AdapGainSched_Cont_Struc}
\end{figure}

\begin{thm} \label{gsa-thm5}
The error $e(t)$ in equation (\ref{eqn_ags57}) is bounded,
\begin{equation}\label{eqn_ags59}
\begin{array}{c}
||e(t)||_{L_{\infty}} \leq \sqrt{\frac{k_m ||\Gamma^{-1}||}{\lambda_{\min}(P)}},
\end{array}
\end{equation}
where
\begin{equation}\label{eqn_ags60}
\begin{array}{l}
 k_m :=4 \sum \limits_{j=1}^{m} \max \limits_{k^*_j(t) \in \Theta_{k_j}} ||k^*_j(t)||^2 
 + 4 \frac{\lambda_{\max}(P)}{\lambda_{\min}(Q)} ||\Gamma^{-1}||\sum \limits_{j=1}^{m} \max \limits_{k^*_j(t) \in \Theta_{k_j}} ||k^*_j(t)|| d_{k_j}.
\end{array}
\end{equation}
\end{thm}

\begin{proof}
A Lyapunov candidate function chosen as
\begin{equation}\label{eqn_ags61}
\begin{array}{l}
V (e(t),\tilde{K}(t))= e^\mathsf{T}(t) P e(t) + \text{trace} \left( \tilde{K}(t)^\mathsf{T} \Gamma^{-1} \tilde{K}(t) \right),
\end{array}
\end{equation}
where its time-derivative for all $\alpha \in \Omega$ is given by
\begin{equation}\label{eqn_ags62}
\begin{array}{l}
\dot{V}(.)= e^\mathsf{T}(t) \left( A_m^\mathsf{T}(\alpha(t)) P + P A_m(\alpha(t)) \right) e(t) 
 + 2e^\mathsf{T}(t) P B \tilde{K}^\mathsf{T}(t) x(t)
+ 2 \text{trace} \left( \tilde{K}^\mathsf{T}(t) \Gamma^{-1} \dot{\tilde{K}}(t) \right).
\end{array}
\end{equation}
Using Lemma \ref{gsa_lem1}, applying trace identity (valid for any two co-dimensional vectors a and b: $a^\mathsf{T}b=\text{trace} \left( b a^\mathsf{T} \right)$), and letting $Y_K(t)=-x(t) e^\mathsf{T}(t) P B$, and knowing $\dot{\tilde{K}}(t)=\dot{\hat{K}}(t)-\dot{K}^*(t)$ leads to
\begin{equation}\label{eqn_ags63}
\begin{array}{l}
\dot{V}(.)\leq -e^\mathsf{T}(t) Q e(t)- 2 \text{trace} \left( \tilde{K}^\mathsf{T}(t) \Gamma^{-1} \dot{K}^*(t) \right) \\[3pt]
~~~~~~~~~~ + 2 \text{trace} \left( \tilde{K}^\mathsf{T}(t) \left[ \Gamma^{-1} \text{Proj}_{\Gamma}(\hat{K}(t), Y_K(t))  - Y_K(t) \right] \right)
\end{array}
\end{equation}
Using Lemma \ref{gsa-lem11}
\begin{equation}\label{eqn_ags64}
\begin{array}{l}
\dot{V}(.) \leq -e^\mathsf{T}(t) Q e(t) - 2 \text{trace} \left( \tilde{K}^\mathsf{T}(t) \Gamma^{-1} \dot{K}^*(t) \right).
\end{array}
\end{equation}
In Essence, the projection operator ensures that the columns $\hat{K}_j$ of the adaptive parameter matrix $\hat{K}(t)$ do not exceed their pre-specified bounds $\hat{K}^{max}_j$, hence for all $t \geq 0,$
\begin{equation}\label{eqn_ags65}
\begin{array}{l}
 \sum \limits_{j=1}^{m} \max \limits_{t \geq 0} \left( \tilde{K}_j^\mathsf{T}(t) \Gamma^{-1} \tilde{K}_j(t) \right) \leq 4 ||\Gamma^{-1}|| \sum \limits_{j=1}^{m} \max \limits_{k^*_j \in \Theta_{k_j}} ||k^*_j(t)||^2.
\end{array}
\end{equation}
If at some $t_1 > 0$, one has $V(e(t_1),\tilde{K}(t_1)) > k_m ||\Gamma^{-1}||$, then it follows from (\ref{eqn_ags60}) and (\ref{eqn_ags61}) that
\begin{equation}\label{eqn_ags66}
\begin{array}{l}
 e^\mathsf{T}(t_1) P e(t_1) > 4 \frac{\lambda_{\max}(P)}{\lambda_{\min}(Q)} ||\Gamma^{-1}||\sum \limits_{j=1}^{m} \max \limits_{k^*_j \in \Theta_{k_j}} ||k^*_j(t)|| d_{k_j},
\end{array}
\end{equation}
and then
\begin{equation}\label{eqn_ags67}
\begin{array}{l}
 e^\mathsf{T}(t_1) (\lambda_{\min}(Q) I) e(t_1) \geq \frac{\lambda_{\min}(Q)}{\lambda_{\max}(P)} e^\mathsf{T}(t_1) P e(t_1) \\[5pt]
~~~~~~~~~~~~~~~~~~~~~~~~~~~~~ > 4 ||\Gamma^{-1}||\sum \limits_{j=1}^{m} \max \limits_{k^*_j \in \Theta_{k_j}} ||k^*_j(t)|| d_{k_j}.
\end{array}
\end{equation}
Notice that $|\text{trace} \left( \tilde{K}^\mathsf{T}(t) \Gamma^{-1} \dot{K}^*(t) \right)| \leq 2 ||\Gamma^{-1}||\sum \limits_{j=1}^{m} \max \limits_{k^*_j \in \Theta_i} ||k^*_j(t)|| d_{k_j}$, $\forall t \geq 0$, which along with the bound in (\ref{eqn_ags67}) leads to $\dot{V}(e(t_1),\tilde{K}(t_1)) < 0$. Since $x_m(0)=x(0)$, we can verify that $V(e(0),\tilde{K}(0)) \leq 4 ||\Gamma^{-1}|| \\ \sum \limits_{j=1}^{m} \max \limits_{k^*_j \in \Theta_{k_j}} ||k^*_j(t)||^2 < k_m ||\Gamma^{-1}||$. Therefore, $V(e(t),\tilde{K}(t)) \leq k_m ||\Gamma^{-1}||$, $\forall t \geq 0$. Since $\lambda_{\min}(P) ||e(t)||^2 \leq e^\mathsf{T}(t) P e(t) \leq V(e(t),\tilde{K}(t))$, then $ ||e(t)|| \leq \sqrt{\frac{k_m ||\Gamma^{-1}||}{\lambda_{\min}(P)}}$. The result in (\ref{eqn_ags59}) follows from the fact that this bound holds uniformly for all $t \geq 0$.
\end{proof}

\subsection{Turboshaft Engine Control Study}

We apply the developed adaptive controller to a high fidelity physics-based model of JetCat SPT5 turboshaft engine driving a variable pitch propeller developed in \cite{fitzgerald-model-2013, pakmehr-decentmodel-2011}. The effect of engine degradation due to aging is modeled in the nonlinear simulation by modifying the efficiencies and flow capacities of key engine components such as: High Pressure Compressor (HPC), High Pressure Turbine (HPT) and Low Pressure Turbine (LPT). These efficiency and flow capacity parameters are known as engine health parameters, and the values of these parameters used in this simulation corresponding to moderate degradation are $\eta_{hpc}=-1.470 \%$, $W_{c,hpc}=-2.455 \%$, $\eta_{hpt}=-1.315 \%$, $W_{c,hpt}=+0.880 \%$, $\eta_{lpt}=-0.269 \%$, and $W_{c,lpt}=+0.1294 \%$, where their nominal values are zero. For a standard day at sea level condition we chose three properly separated equilibrium points on the nominal plant equilibrium manifold for linearizing the plant model at those points. The linearization matrices for these equilibrium points and steady state values of the engine variables, scheduling parameter and control parameters are given as follows:
\begin{itemize}
\item Equilibrium Point 1 (Cruise):\\
	$u_{1e1}=0.4685,~ u_{2e1}=16 ~(\text{deg}),~ x_{1e1}=0.7264,~ x_{2e1}=0.5,~ T_{e1}=70.5125 ~(\text{N}),~ \alpha_1=0.8818,$ and
\begin{eqnarray} \label{eqn_agsf3}
\begin{array}{c}
  A^p_1=
   \left[
       \begin{array}{cc}
           -1.7 & 0.1 \\
            0.6 & -1.1
       \end{array}
    \right],~ B^p_1=
    \left[
       \begin{array}{cc}
           1.2 & 0 \\
           0.3 & -0.023
       \end{array}
       \right], \\[5pt]
        C^p_1=I, ~   
				K_{i1}= \left[
       \begin{array}{cc}
           -0.4 & -0.4 \\
           -0.4 & -0.4
       \end{array}
    \right].
\end{array}
\end{eqnarray}

\item Equilibrium Point 2:\\
	$u_{1e2}=0.3,~ u_{2e2}=16 ~(\text{deg}),~ x_{1e2}=0.5327,~ x_{2e2}=0.3678,~ T_{e2}=38.155 ~(\text{N}),~ \alpha_2=0.6473,$ and
\begin{eqnarray} \label{eqn_agsf4}
\begin{array}{c}
  A^p_2=
   \left[
       \begin{array}{cc}
           -0.85 & 0.032 \\
            0.32 & -0.64
       \end{array}
    \right],~ B^p_2=
    \left[
       \begin{array}{cc}
           1.0 & 0 \\
           0.17 & -0.011
       \end{array}
       \right], \\[5pt]
       C^p_2=I, ~
   K_{i2}=
   \left[
       \begin{array}{cc}
           -0.3 & -0.3 \\
           -0.3 & -0.3
       \end{array}
    \right].
\end{array}
\end{eqnarray}
\item Equilibrium Point 3 (Idle):\\
	$u_{1e3}=0.145,~ u_{2e3}=16 ~(\text{deg}),~ x_{1e3}=0.295,~ x_{2e3}=0.161,~ T_{e3}=7.317 ~(\text{N}),~ \alpha_3=0.3361,$ and the matrices are
\begin{eqnarray} \label{eqn_agsf5}
\begin{array}{c}
  A^p_3=
   \left[
       \begin{array}{cc}
           -0.38 & -0.0008 \\
            0.26  & -0.34
       \end{array}
    \right], ~ B^p_3=
    \left[
       \begin{array}{cc}
           0.7 & 0 \\
           0.1 & -0.0024
       \end{array}
       \right], \\[5pt]
       C^p_3= I, ~
   K_{i3}=
   \left[
       \begin{array}{cc}
           -0.2 & -0.2 \\
           -0.2 & -0.2
       \end{array}
    \right].
\end{array}
\end{eqnarray}
\end{itemize}

Other controller parameters are $\epsilon_c=1,~\eta_c=3$. To show the stability of the closed loop reference system, 40 different (30 equilibrium, and 10 non-equilibrium) linearizations have been used, to solve inequality (\ref{eqn_gs115}), in Matlab with the aid of YALMIP \cite{YALMIP-lofberg-2004} and SeDuMi \cite{sedumi-Sturm-2001} packages. The numerical value for the common matrix $P$ is
\begin{eqnarray*}
P = \left[
    \begin{array}{cccccc}
        0.491  &  0.079  &  0.102  & -0.004  & -0.072 & -0.039\\
        0.079  &  0.446  &  0.053  &  0.007  & -0.097 & -0.013\\
        0.102  &  0.053  &  0.181  & -0.041  & -0.028 & -0.022\\
       -0.004  &  0.007  & -0.041  &  0.130  &  0.023 &  0.013\\
       -0.072  & -0.097  & -0.028  &  0.023  &  0.321 &  0.045\\
       -0.039  & -0.013  & -0.022  &  0.013  &  0.045 &  0.332
    \end{array}
    \right],
\end{eqnarray*}
where its condition number is $\kappa(P)=6.6303$, and $Q=0.1\times I_6$. Figure \ref{comp1_map_stab_gsa} shows JetCat SPT5 turboshaft engine compressor map. In this map the approximate stall line and also the operating line for this simulation have been shown. The engine operates in a safe region with a big stall margin during its acceleration from idle to cruise condition. The points which are used for linearization and stability analysis of the closed-loop system also have been shown in this figure. Out of 40 points, 30 are related to the equilibrium linearizations which are situated on the steady-state operating line of the engine, and the other 10 points are related to the non-equilibrium linearizations which are situated near the steady-state operating line of the engine. The engine operating lines for the nominal and degraded engine models are shown in this figure.
\begin{figure}[!ht]
\centering
\includegraphics[width=0.7\textwidth]{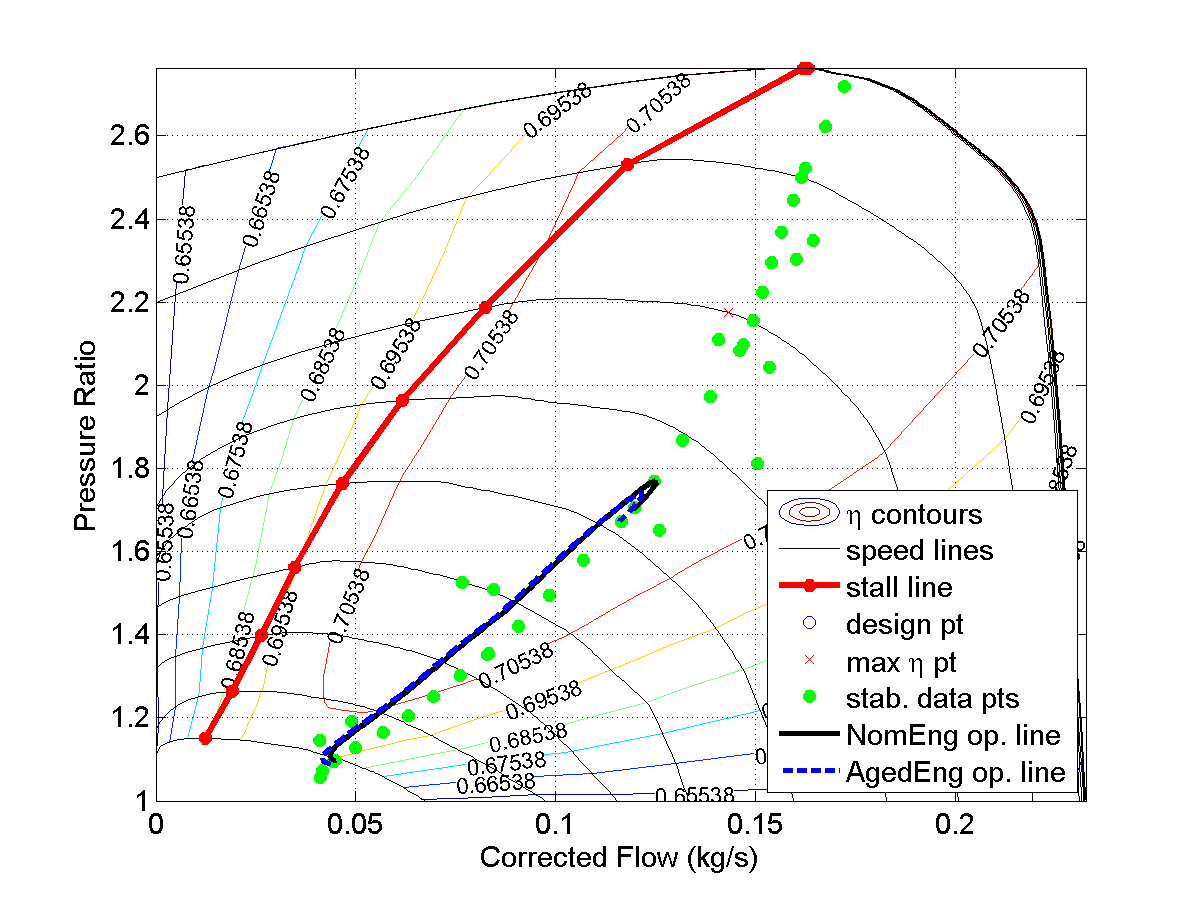}
\caption{JetCat SPT5 engine compressor map with data points used to compute $P$ and operating line for nominal engine, and deteriorated engine.}\label{comp1_map_stab_gsa}
\end{figure}

Based on the dynamical properties of the gas turbine engine system, the numerical values for the adaptive controller are set as follows
\begin{equation} \label{eqn_agsf6}
\begin{array}{c}
  \Gamma =\text{diag}([100,~100,~100,~100,~100,~100]), \\[5pt]
     K^*_i \in \Theta_{K_i}= \left[
       \begin{array}{cc}
          $[-2, ~ 0]$   &  $[-2, ~ 0]$ \\
          $[-2, ~ 0]$   &  $[-2, ~ 0]$
       \end{array}
    \right], \\[5pt]
    \hat{K}_i(0)= \left[
       \begin{array}{cc}
          -0.1950   &  -0.1950 \\
          -0.1970   &  -0.1970
       \end{array}
    \right].
\end{array}
\end{equation}
Simulation results are shown in Figures \ref{fig1_gsa_01} to \ref{fig1_gsa_15}. Two different simulations included in these pictures include the nominal engine (NomEng) and degraded/aged engine (AgedEng) cases.
\begin{figure}[!ht]
\centering
\begin{minipage}[b]{0.48\textwidth}
\centering
\includegraphics[width=0.99\textwidth]{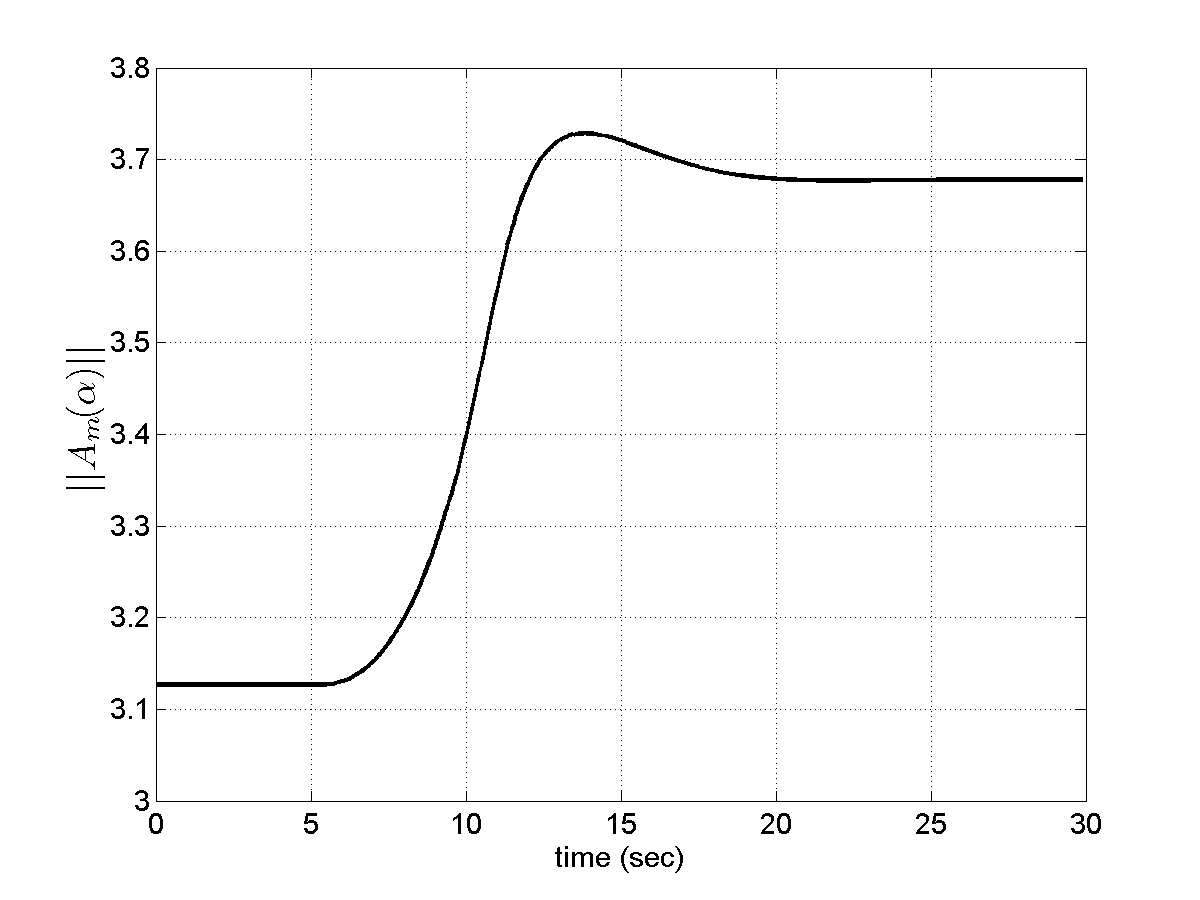}
\caption{Norm of reference model matrix ($||A_{m}(\alpha(t))||$) } \label{fig1_gsa_01}
\end{minipage}
\hfill
\begin{minipage}[b]{0.48\textwidth}
\centering
\includegraphics[width=0.99\textwidth]{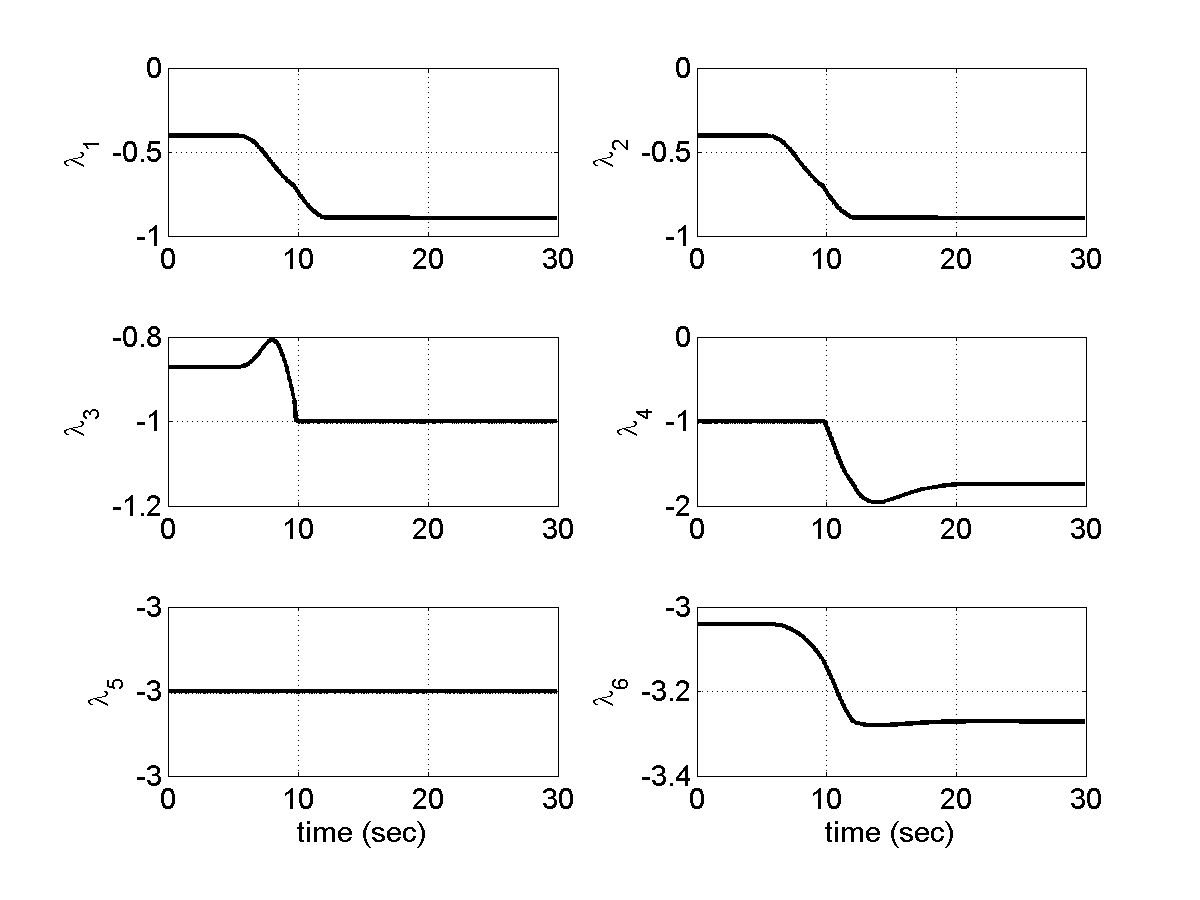}
\caption{Closed-loop system eigenvalues ($\lambda[A_{m}(\alpha(t))]$)} \label{fig1_gsa_02}
\end{minipage}
\end{figure}
Figure \ref{fig1_gsa_01} shows the history of the norm of desired reference system matrix $||A_{m}(\alpha(t))||$. As it can be seen the figure shows the boundedness of these two matrices, in accordance with Assumption \ref{gsa_ass1}, where $k_A=4.1023$. Figure \ref{fig1_gsa_02} shows the history of the desired reference system matrix eigenvalues $\lambda [A_{m}(\alpha(t))]$. As it is apparent, all the six eigenvalues remain negative with the time change of the scheduling parameter $\alpha(t)$. Figures \ref{fig1_gsa_07} and \ref{fig1_gsa_08} show the outputs (i.e., high and low pressure spool speeds) tracking their reference signals for three cases.

\begin{figure}[!ht]
\centering
\begin{minipage}[b]{0.48\textwidth}
\centering
\includegraphics[width=0.98\textwidth]{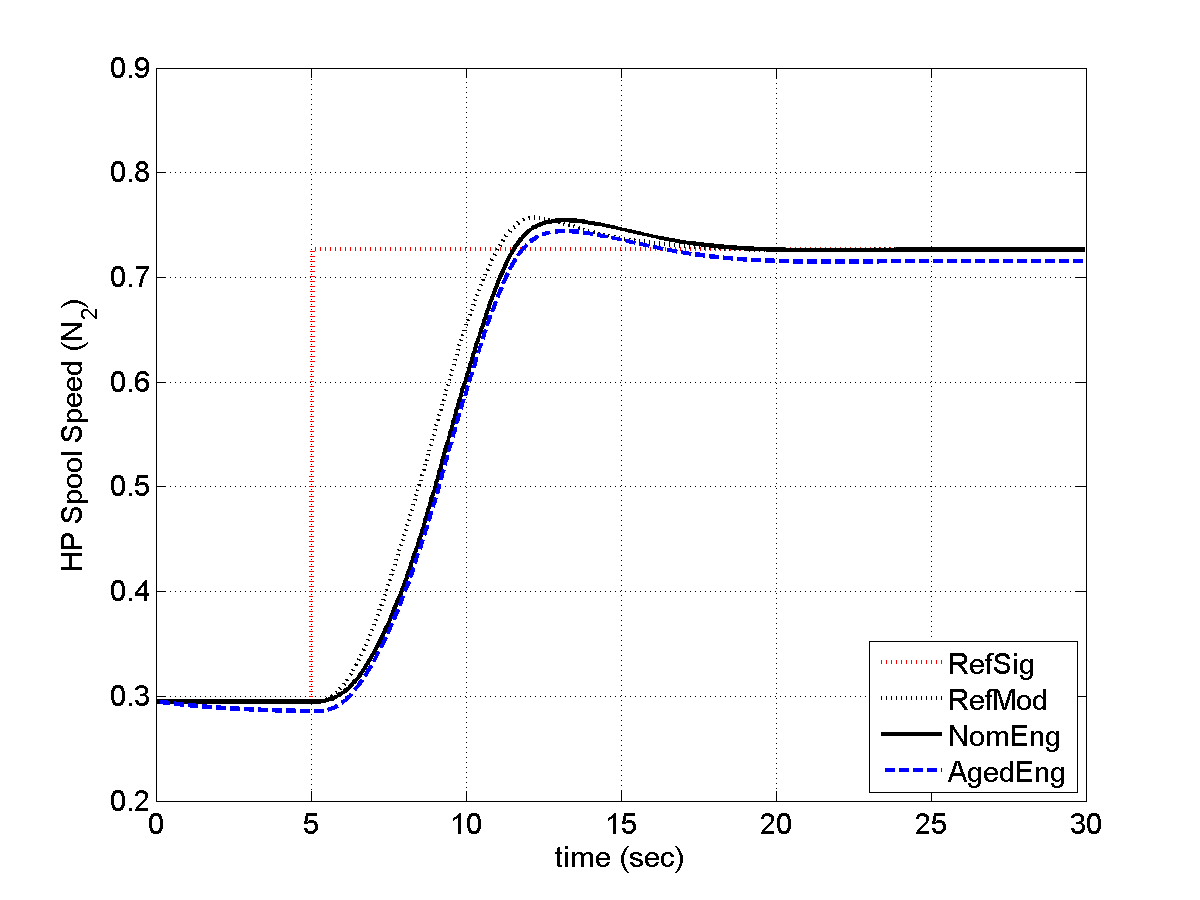}
\caption{High pressure spool speed and its reference signal}\label{fig1_gsa_07}
\end{minipage}
\hfill
\begin{minipage}[b]{0.48\textwidth}
\centering
\includegraphics[width=0.99\textwidth]{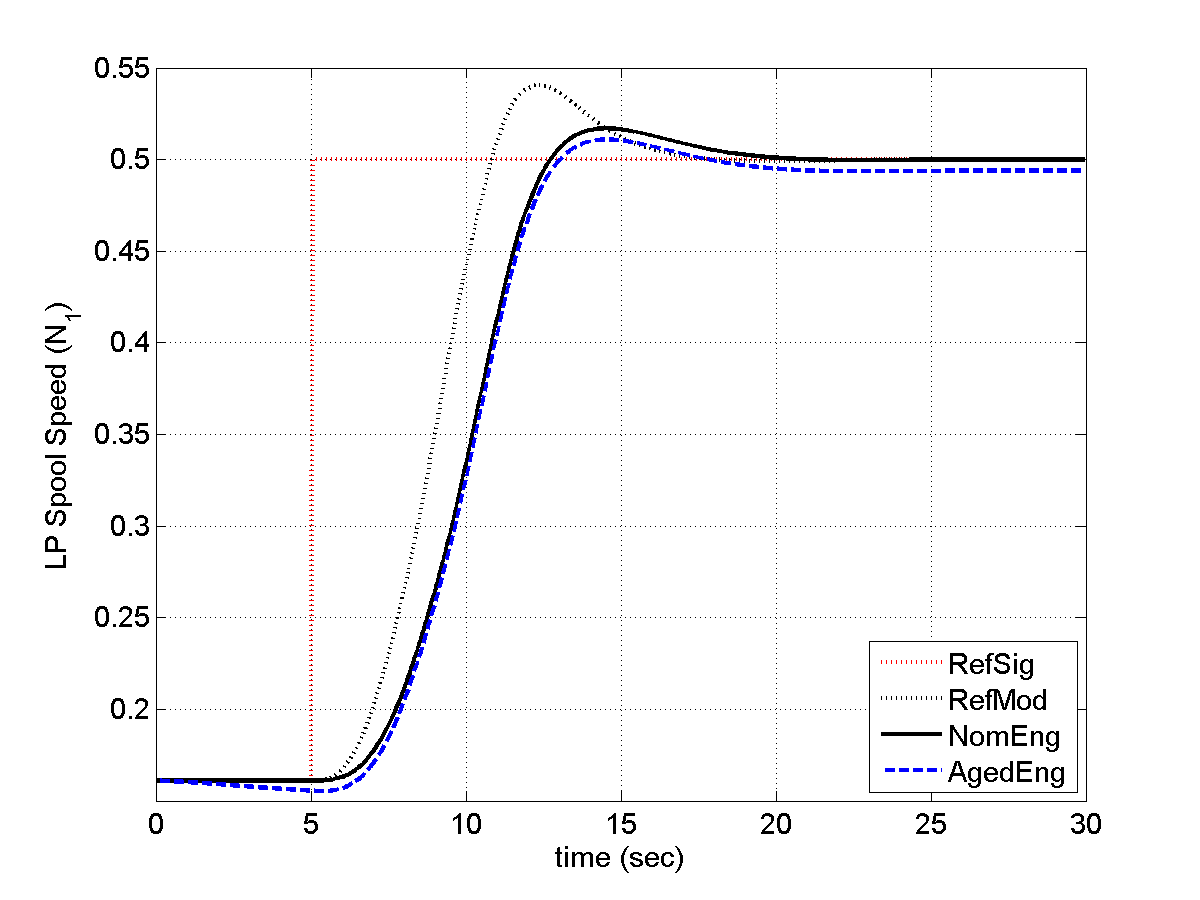}
\caption{Low pressure spool speed and its reference signal}\label{fig1_gsa_08}
\end{minipage}
\end{figure}
Figure \ref{fig1_gsa_09} shows the history of thrust and it is following its reference command from idle to cruise condition and then back to the idle for standard day, sea level condition. Figure \ref{fig1_gsa_10} shows the evolution of the infinity norm of the errors $||e(t)||$. The steady-state error in the Aged Engine (AgedEng) simulation case is because of the effect of the aging on the engine health parameters, and this causes a change in the equilibrium manifold of the aged engine in comparison to the nominal engine (NomEng). In other words, since we are using nominal engine equilibrium manifold to design a linear parameter dependant reference model, and the aged engine linear model has a different equilibrium manifold $x^p_{e, nom}(\alpha(t))\neq x^p_{e, aged}(\alpha(t))$, and $u_{e, nom}(\alpha(t))\neq u_{e, aged}(\alpha(t))$, then $\delta x^p_{aged}(t)= x^p_{aged}(t)-x^p_{e, aged}(\alpha(t)) \neq 0$, and $\delta u_{aged}(t)= u_{aged}(t)-u_{e, aged}(\alpha(t)) \neq 0$, and this means $||\delta x_{aged}(t)|| > \delta x_{\min} \neq 0$ for all $t > 0$, hence, there will be a steady-state error value greater than zero.

\begin{figure}[!ht]
\centering
\begin{minipage}[b]{0.48\textwidth}
\centering
\includegraphics[width=0.99\textwidth]{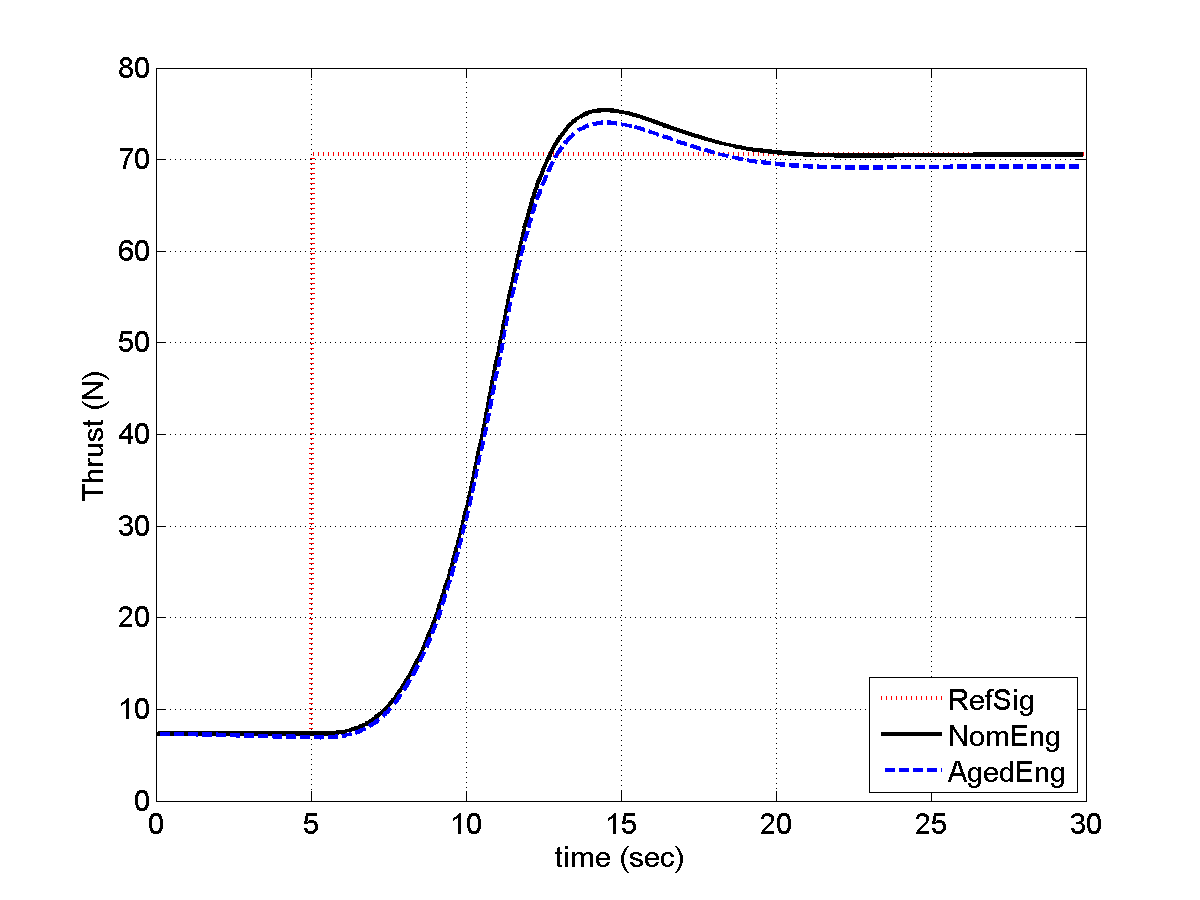}
\caption{Thrust and its reference signal}\label{fig1_gsa_09}
\end{minipage}
\hfill
\begin{minipage}[b]{0.48\textwidth}
\centering
\includegraphics[width=0.99\textwidth]{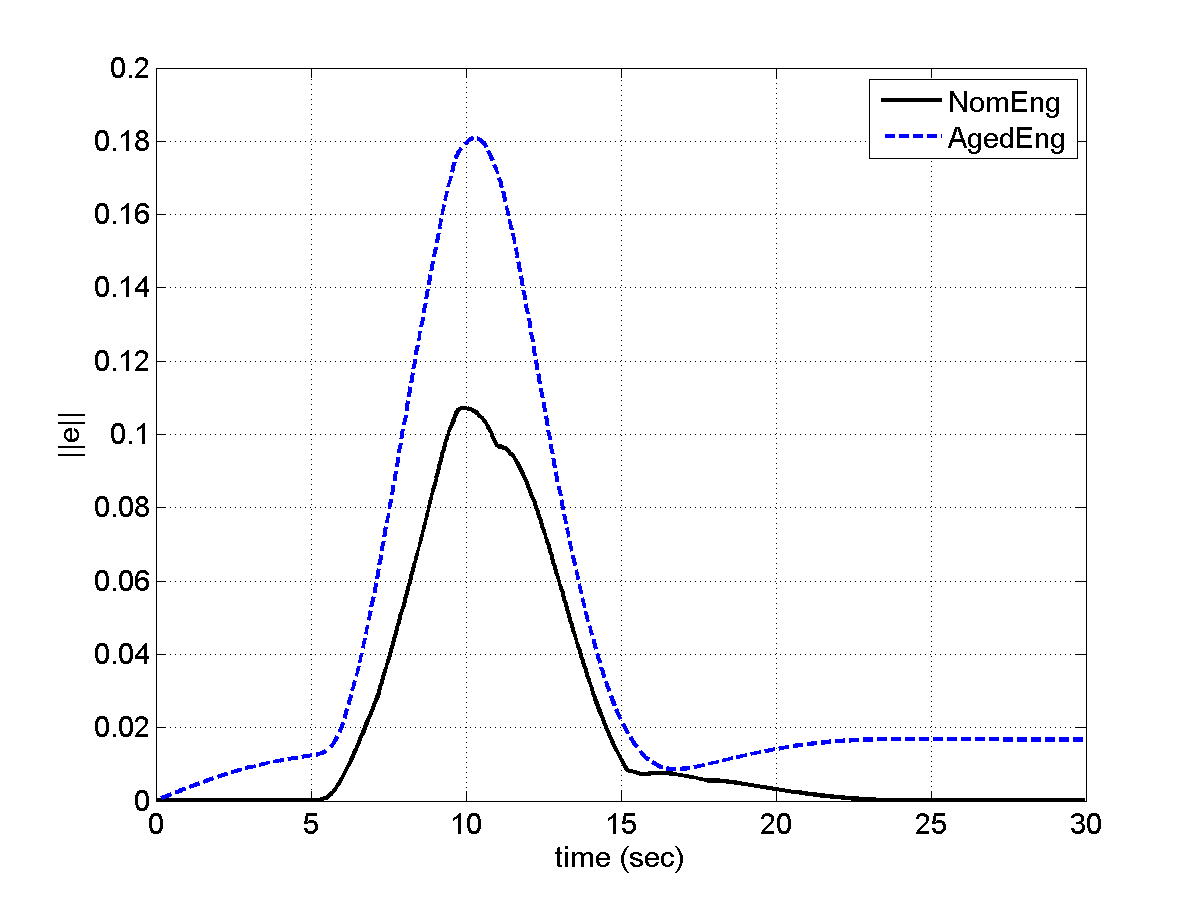}
\caption{Norm of the error signals $||e(t)||_{\infty}$ }\label{fig1_gsa_10}
\end{minipage}
\end{figure}

Figure \ref{fig1_gsa_11} shows the evolution of the control inputs $v(t)=[v_1(t), v_2(t)]^T$, which are inputs to the augmented system, each element is corresponding to one of the control inputs to the original system. Figure \ref{fig1_gsa_14} shows the histories of fuel flow and propeller pitch angle as the control inputs to the plant. Figures \ref{fig1_gsa_15} shows the evolution of the gain scheduling controller integral gain matrix ($K_i(\alpha(t))$) and also adaptive controller gain matrix ($\hat{K}_i(t)$).

\begin{figure}[!ht]
\centering
\begin{minipage}[b]{0.48\textwidth}
\centering
\includegraphics[width=0.99\textwidth]{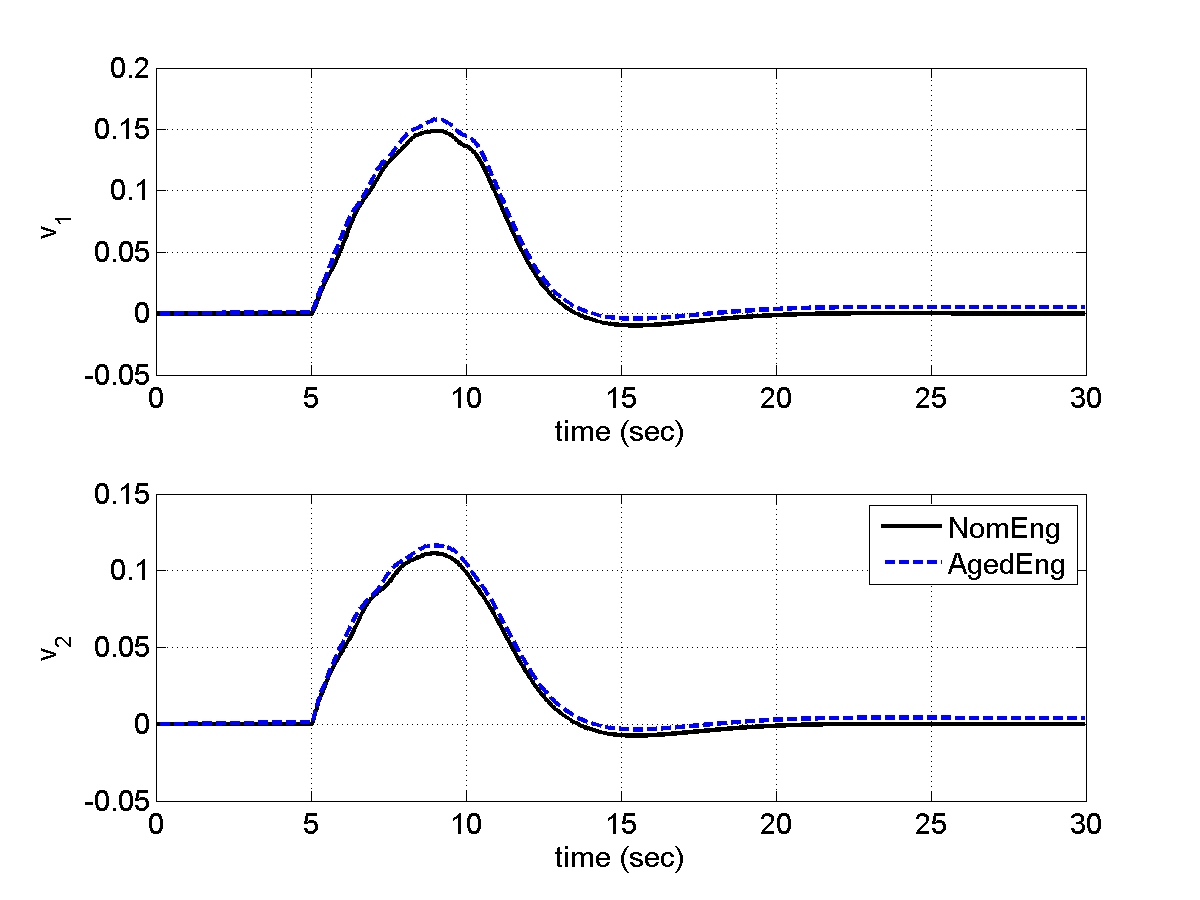}
\caption{Control inputs to the augmented system ($v(t)$)}\label{fig1_gsa_11}
\end{minipage}
\hfill
\begin{minipage}[b]{0.48\textwidth}
\centering
\includegraphics[width=0.99\textwidth]{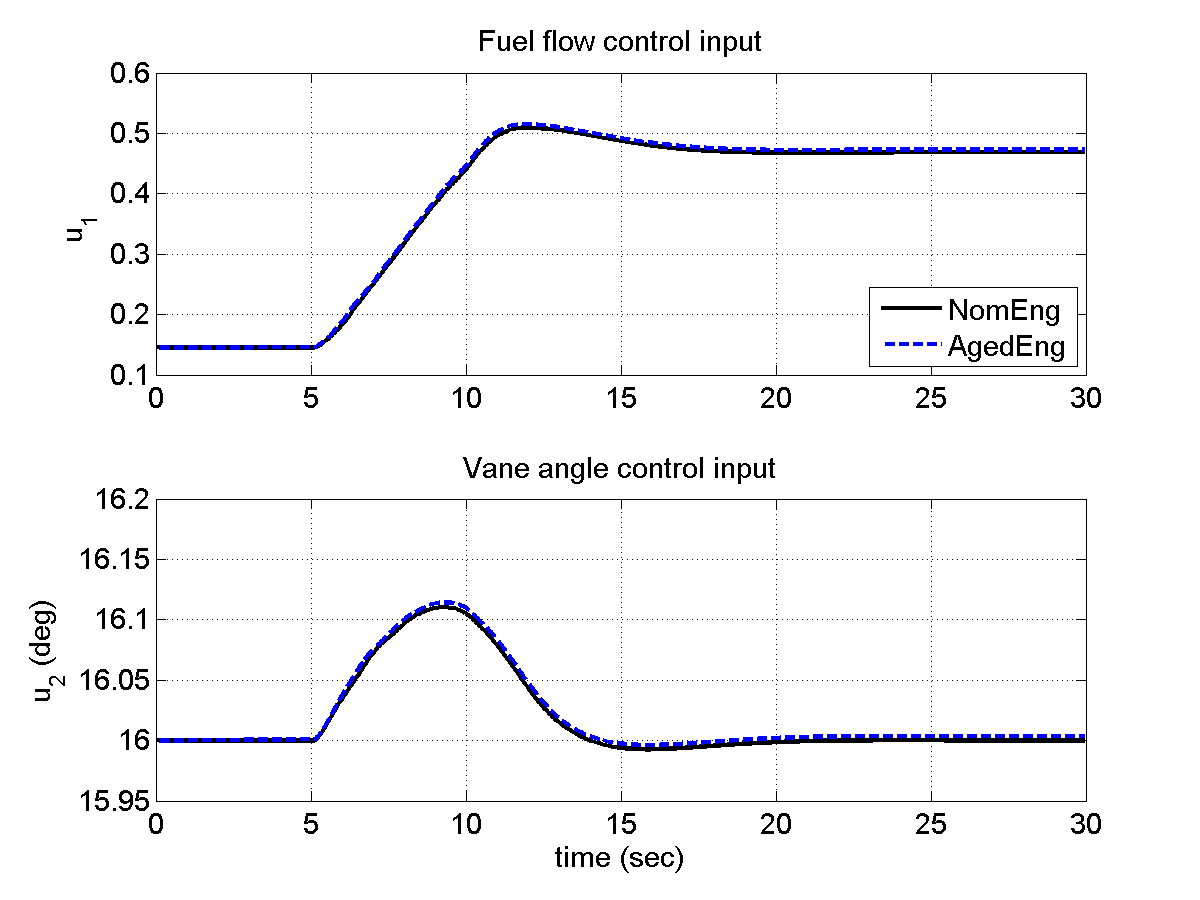}
\caption{Fuel and prop pitch angle control inputs ($u(t)$)}\label{fig1_gsa_14}
\end{minipage}
\end{figure}

$K_i(\alpha)$ have been obtained by interpolation using the predesigned indexed family of fixed-gain controllers, and each controller corresponds to one equilibrium point of the engine. The numerical values of these gains are given in equations (\ref{eqn_agsf3}) to (\ref{eqn_agsf5}), which represent the controller gains for idle and cruise condition and one more equilibrium point in between these two operating points. $\hat{K}_i(t)$ is generated using an adaptive law.
\begin{figure}[!ht]
\centering
\includegraphics[width=0.6\textwidth]{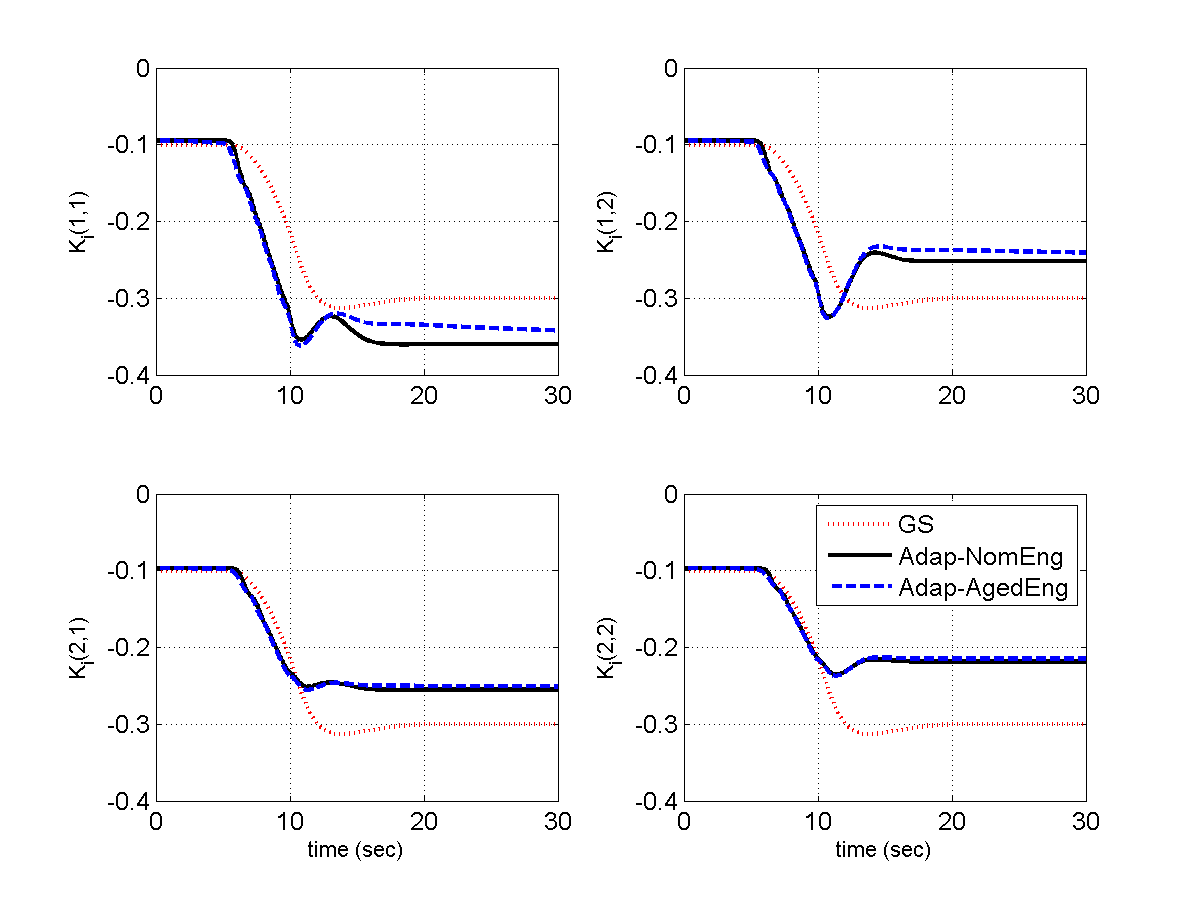}
\caption{Integral gain matrix elements for the gain scheduling controller ($K_i(\alpha(t))$), and for the adaptive controller ($\hat{K}_i(t)$)}\label{fig1_gsa_15}
\end{figure}

For the cases where engine overspeed or turbine temperature limits are of concern, the controller should be modified to consider the inputs with constraints. Investigating the adaptive controller with constrained control inputs for mitigating these issues is out of the scope of this manuscript. 
   
\subsection{Conclusions}
With the aid of convex optimization tools, a single quadratic Lyapunov function was computed, which guarantees the stability of the gain scheduled reference model. Adaptive state feedback control scheme for systems with gain scheduled reference models was developed, and its stability was proven. The resulting closed-loop system was shown to have ultimately bounded solutions with a priori adjustable bounded tracking error. Through the simulation based on a high fidelity physics-based model of a JetCat SPT5 turboshaft engine driving a variable pitch propeller, it was demonstrated the proposed adaptive controller, effectively, tracks the reference model in both nominal and degraded engine models. The developed adaptive control structure is not only limited to control degraded/damaged gas turbine engines, but also can be used for other practical applications.

\section{Adaptive Control of Systems with Gain Scheduled Reference Models - Part II: Constrained Control Inputs}

\subsection{Introduction}
Various adaptive control approaches for systems with input saturation are described in \cite{AdapContConstraint-karason-1994, PhDThesis-jang-2009, AdapSatCont-monopoli-1975, LimAuthAdap-johnson-2003}. The stability proofs in these works are shown for adaptive control systems with LTI reference models. In this section a stable adaptive control structure will be developed with a time-varying reference model presented before. This model is a gain scheduled reference model without switching problem; which, in case of the gas turbine engine example its stability can be shown by finding a single Lyapunov function. The stability analysis presented here, uses some results from \cite{AdapContConstraint-karason-1994, PhDThesis-jang-2009}. The constraints on the control inputs are implemented using a multi-dimensional rectangular saturation function. This controller, then, has been implemented on a high fidelity physics-based JetCat SPT5 turboshaft engine model for large throttle commands with constraints on the control inputs to keep the engine in its safe operating envelope.

The contribution of this section is the development of a stable state feedback model reference adaptive control algorithm for systems with gain scheduled reference models and constrained control inputs in a MIMO setting; the approach is applicable to  systems, such as gas turbine engines, which the stability of their gain scheduled reference model is guaranteed by computing a single Lyapunov function \cite{PhDThesis-pakmehr-2013, GSstability-pakmehr-2013, GainSchedStabConf-pakmehr-2013}. A detailed stability analysis is performed for the proposed controller which can be used towards control software verification and certification \cite{ContSoftware-feron-2010}.

The rest of this section is organized as follows. In subsection II, a model reference adaptive control with a gain scheduled reference model and constrained control inputs is designed with detailed stability proof. In subsection III, simulations are performed for two different engine models including the nominal and degraded engine models. Simulation results show that the developed adaptive controller can be used effectively for the entire flight envelope of the turboshaft engine with guaranteed stability. Subsection IV, concludes this section.

\subsection{Adaptive Control with Constrained Control Inputs}

\subsubsection{Stability Analysis}
In order to avoid the adaptive controller parameters to be adjusted improperly by the saturation error, we use the augmented error method in the adaptive control design developed in \cite{AdapContConstraint-karason-1994, PhDThesis-jang-2009} to provide the stability analysis for a gain scheduled model reference adaptive control system.  The plant (\ref{eqn_ags152}) with saturated control inputs can be written as
\begin{equation}\label{eqn_ags80}
\begin{array}{c}
\dot{x}=A(\alpha(t))x(t) + B R_s(v(t)) + B_r r(t), ~~ \forall \alpha \in \Omega,
\end{array}
\end{equation}
where $v(t)$ is the adaptive control input which introduced in equation (\ref{eqn_ags54}). the ultimate goal is to determine adaptive parameters such that all signals in the plant (\ref{eqn_ags80}) are guaranteed to be bounded, and $y(t)$ tracks $r(t)$. The deficiency of $v(t)$ is defined as $\Delta v(t) = v(t) - R_s(v(t))$. Now, plant (\ref{eqn_ags80}) can be written as
\begin{equation}\label{eqn_ags82}
\begin{array}{c}
\dot{x}=A(\alpha(t))x(t) + B v(t) - B \Delta v(t) + B_r r(t), ~~ \forall \alpha \in \Omega.
\end{array}
\end{equation}
Plant (\ref{eqn_ags82}) with controller (\ref{eqn_ags54}) can be written as
\begin{equation}\label{eqn_ags821}
\begin{array}{c}
\dot{x}=A(\alpha(t))x(t) + B \hat{K}^\mathsf{T}(t)x(t) - B \Delta v(t) + B_r r(t), ~~ \forall \alpha \in \Omega.
\end{array}
\end{equation}
Subtracting the reference model (\ref{eqn_ags52}) and the plant (\ref{eqn_ags821}), a closed-loop error dynamics equation is obtained as
\begin{equation}\label{eqn_ags83}
\begin{array}{l}
\dot{e}(t)= A_m(\alpha(t)) e(t) + B \tilde{K}^\mathsf{T}(t)x(t)- B \Delta v(t), ~~ \forall \alpha \in \Omega.
\end{array}
\end{equation}
In order to eliminate the adverse effect of the disturbance $\Delta v(t)$ we generate a signal $e_{\Delta}(t)$ as
\begin{equation}\label{eqn_ags84}
\begin{array}{l}
\dot{e}_{\Delta}(t)= A_m(\alpha(t)) e_{\Delta}(t) - K_{\Delta}(t) \Delta v(t), ~~ \forall \alpha \in \Omega, \\
e_{\Delta}(t_0)=0
\end{array}
\end{equation}
where $K_{\Delta}(t) \in \Re^{(n+2m) \times m}$. The undesirable effects due to control input saturation can be removed from the error dynamics in equation (\ref{eqn_ags83}) by defining an augmented error $e_v(t)=e(t)-e_{\Delta}(t)$. Its dynamics can be determined as
\begin{equation}\label{eqn_ags85}
\begin{array}{l}
\dot{e}_{v}(t)= A_m(\alpha(t)) e_{v}(t) + B \tilde{K}^\mathsf{T}(t)x(t)- \tilde{K}_{\Delta}(t) \Delta v(t), ~~ \forall \alpha \in \Omega,
\end{array}
\end{equation}
where $\tilde{K}_{\Delta}(t)=B-K_{\Delta}(t)$. Let $K_{\Delta}(t) \in \Theta_{\Delta}$, where $\Theta_{\Delta}$ is a convex compact set. We now choose adaptive laws for adjusting the parameters
\begin{equation}\label{eqn_ags86}
\begin{array}{l}
\dot{\hat{K}}(t)= \text{Proj}_{\Gamma} \left( \hat{K}(t), -x(t)e_v^\mathsf{T}(t) P B \right), \\
\dot{K}^\mathsf{T}_{\Delta}(t)= \text{Proj}_{\Gamma} \left( K^\mathsf{T}_{\Delta}(t), \Delta v(t) e_v^\mathsf{T}(t) P \right),
\end{array}
\end{equation}
where $P=P^\mathsf{T} > 0$ is a solution of LMI (\ref{eqn_gs118}). The gains in adaptive laws $\Gamma \in \Re^{(n+2m) \times (n+2m)}$, $\Gamma_{\Delta} \in \Re^{m \times m}$ are positive definite matrices $\Gamma=\Gamma^\mathsf{T} > 0$ and $\Gamma_{\Delta}=\Gamma^\mathsf{T}_{\Delta} > 0$.
\begin{thm} \label{gsa-thm6}
The error $e_v(t)$ in equation (\ref{eqn_ags85}) is bounded,
\begin{equation}\label{eqn_ags87}
\begin{array}{c}
||e_v(t)||_{L_{\infty}} \leq \sqrt{\frac{k_m ||\Gamma^{-1}||}{\lambda_{\min}(P)}},
\end{array}
\end{equation}
where $k_m :=4 \sum \limits_{j=1}^{m} \max \limits_{k_j \in \Theta_i} ||k^*_j(t)||^2 + 4 \frac{\lambda_{\max}(P)}{\epsilon_q} ||\Gamma^{-1}||\sum \limits_{j=1}^{m} \max \limits_{k^*_j \in \Theta_{k_j}} ||k^*_j(t)|| d_{k_j}$. 
\end{thm}
\begin{proof}
A Lyapunov candidate function is chosen as $V (e_v(t),\tilde{K}(t), \tilde{K}_{\Delta}(t))= e_v^\mathsf{T}(t) P e_v(t)  + \text{trace} \left( \tilde{K}(t)^\mathsf{T} \Gamma^{-1} \tilde{K}(t) \right) + \text{trace} \left( \tilde{K}_{\Delta}(t) \Gamma_{\Delta}^{-1} \tilde{K}^\mathsf{T}_{\Delta}(t) \right)$, where its time-derivative, for all $\alpha \in \Omega$ is given by
\begin{equation}\label{eqn_ags90}
\begin{array}{l}
\dot{V}(.)= e_v^\mathsf{T}(t) \left( A_m^\mathsf{T}(\alpha(t)) P + P A_m(\alpha(t)) \right) e_v(t) \\[3pt]
~~~ + 2e_v^\mathsf{T}(t) P B \tilde{K}^\mathsf{T}(t) x(t) + 2 \text{trace} \left( \tilde{K}^\mathsf{T}(t) \Gamma^{-1} \dot{\tilde{K}}(t) \right) \\[5pt]
~~~ - 2e_v^\mathsf{T}(t) P \tilde{K}_{\Delta}(t) \Delta v(t) + 2 \text{trace} \left( \tilde{K}_{\Delta}(t) \Gamma_{\Delta}^{-1} \dot{\tilde{K}}^\mathsf{T}_{\Delta}(t) \right).
\end{array}
\end{equation}
Using Lemma \ref{gsa_lem1}, applying trace identity, and letting $Y_K(t)=-x(t) e_v^\mathsf{T}(t) P B$, and $Y_{K_{\Delta}}(t)= \Delta v(t) e_v^\mathsf{T}(t) P$, and knowing $\dot{\tilde{K}}(t)=\dot{\hat{K}}(t)-\dot{K}^*(t)$, and $\dot{\tilde{K}}_{\Delta}(t)=-\dot{K}_{\Delta}(t)$  leads to
\begin{equation}\label{eqn_ags91}
\begin{array}{l}
\dot{V}(.) \leq -e_v^\mathsf{T}(t) Q e_v(t)- 2 \text{trace} \left( \tilde{K}^\mathsf{T}(t) \Gamma^{-1} \dot{K}^*(t) \right)\\[3pt]
~~+ 2 \text{trace} \left( \tilde{K}^\mathsf{T}(t) \left[ \Gamma^{-1} \text{Proj}_{\Gamma}(\hat{K}(t), Y_K(t))- Y_K(t) \right] \right)\\[5pt]
~~ + 2 \text{trace} \left( \tilde{K}_{\Delta}(t) \left[ \Gamma_{\Delta}^{-1}\text{Proj}_{\Gamma}(K^\mathsf{T}_{\Delta}(t), Y_{K_{\Delta}}(t))- Y_{K_{\Delta}}(t) \right] \right).
\end{array}
\end{equation}
Using Lemma \ref{gsa-lem11}
\begin{equation}\label{eqn_ags92}
\begin{array}{l}
\dot{V}(.) \leq -e_v^\mathsf{T}(t) Q e_v(t) - 2 \text{trace} \left( \tilde{K}^\mathsf{T}(t) \Gamma^{-1} \dot{K}^*(t) \right).
\end{array}
\end{equation}
In Essence, the projection operator ensures that the columns $\hat{K}_j$ of the adaptive parameter matrix $\hat{K}(t)$ do not exceed their pre-specified bounds $\hat{K}^{max}_j$, hence $\sum \limits_{j=1}^{m} \max \limits_{t \geq 0} \left( \tilde{K}_j^\mathsf{T}(t) \Gamma^{-1} \tilde{K}_j(t) \right) \leq 4 ||\Gamma^{-1}|| \sum \limits_{j=1}^{m} \max \limits_{k^*_j \in \Theta_{k_j}} ||k^*_j(t)||^2 $, for all $t \geq 0$. The rest of the proof is similar to the proof of Lemma 5.1.2 from \cite{L1AdaptiveBook-hovakimyan-2010}.
\end{proof}
\begin{rmk}\label{gsa-rmk3}
The proof of Theorem \ref{gsa-thm6} showed the boundedness of $e_v(t)$, however it can not guarantee the boundedness of the tracking error $e(t)$. To prove the boundedness of $e(t)$, we must prove that $x(t)$ is bounded when the control inputs are constrained under rectangular saturation.
\end{rmk}
We define $\Theta^*_{\max}$ and $\Theta_{\max}$ to be $\Theta^*_{\max}= \mathrm{sup}||K^*(t)||$, $\Theta_{\max}=\max \left[ \mathrm{sup}||\tilde{K}(t)||,~ \mathrm{sup}||\tilde{K}_{\Delta}(t)|| \right]$. Since we assumed the control gains belong to a known compact set, then $\Theta^*_{\max}$ and $\Theta_{\max}$ are positive and finite, hence there exists a smallest $n \in \mathbb{N}$ such that $\Theta^*_{\max} \leq n \Theta_{\max}$. For efficiency of notation we define
\begin{equation}\label{eqn_ags94}
\begin{array}{l}
 \gamma_{\max} := \max \left[ ||\Gamma^{-1}||,~||\Gamma^{-1}_{\Delta}|| \right], \\[10pt]
 v_{\min} := \min \limits_{i} (v_{i, \max}),  ~ v_{\max} := \max \limits_{i} (v_{i, \max}), \\[10pt]
 v_{0} := \sqrt{ \sum \limits_{i=1}^{m} v^2_{i,\max} }, ~ \rho := \sqrt{\frac{\lambda_{\max}(P)}{\lambda_{\min}(P)}},
\end{array}
\end{equation}
where $v_{i,\max}>0$ is the limit of the $i$th element of $v(t)$ and $Z_B \in \Re$ is defined using the induced norm by the vector 2-norm such that the property is described by $||x^\mathsf{T}(t) P[B,~B_r]|| \leq Z_B ||x(t)||$. We also define the following constants for simplicity
\begin{equation}\label{eqn_ags941}
\begin{array}{l}
 x_{\min} := \frac{ Z_B(2v_{0} + 2 r_{\max}) }{\lambda_{\min}(Q)-(3n+2)Z_B \Theta_{\max} }, \\[10pt]
 x_{\max} := \frac{Z_B v_{\min}}{|\lambda_{\min}(Q)-2Z_B \Theta^*_{\max}|}, \\[10pt]
 Z_{\max} := \frac{ \lambda_{\min}(Q)-Z_B \frac{\rho}{x_{\max}}(2v_{0} + 2 r_{\max}) }{ Z_B (3\frac{\rho}{x_{\max}}+3n+2) }.
\end{array}
\end{equation}

\begin{thm} \label{gsa-thm7}
Under Assumptions \ref{gsa-ass3} and \ref{gsa-ass2} for the system (\ref{eqn_ags80}) with the controller (\ref{eqn_ags54}) and the adaptive laws in (\ref{eqn_ags86}), $x(t)$ has a semi-globally bounded trajectory with respect to the level of saturation for all $t>0$ if
\begin{equation}\label{eqn_ags95}
\begin{array}{l}
  \mathrm{(i)}~ ||x(0)|| < \frac{x_{\max}}{\rho},\\
  \mathrm{(ii)}~ \sqrt{V(0)} < \frac{Z_{\max}}{\sqrt{\gamma_{\max}}}.
\end{array}
\end{equation}
Further
\begin{equation}\label{eqn_ags96}
\begin{array}{l}
||x(t)|| <  x_{\max}, ~~ \forall t>0,
\end{array}
\end{equation}
and error $e(t)$ is in the order of $||e(t)|| = \mathcal{O}[\sup \limits_{\tau \leq t}||\Delta v(\tau)||]$.
\end{thm}

\begin{proof}
We choose a positive definite function $W(x(t))$, as
\begin{equation}\label{eqn_ags98}
\begin{array}{l}
W(x(t))=x^{\mathsf{T}}(t) P x(t),
\end{array}
\end{equation}
and define a level set,$\mathcal{N}$, of $W(x(t))$ as
\begin{equation}\label{eqn_ags99}
\begin{array}{l}
 \mathcal{N}= \left\{x(t) |W(x(t))=\lambda_{\min}(P) x^2_{\max}\right\},
\end{array}
\end{equation}
where $x_{\max}$ is defined in (\ref{eqn_ags94}). We know define region of attraction $\mathcal{M}$ as
\begin{equation}\label{eqn_ags100}
\begin{array}{l}
 \mathcal{M}=\left\{x_{\min}<||x(t)||<x_{\max} \right\}.
\end{array}
\end{equation}
The proof proceeds in two steps. First, we show that condition (ii) in this theorem implies that $\mathcal{N} \subset \mathcal{M}$. Then we show that $\dot{W}(x(t)) < 0$ for all $x(t) \in \mathcal{M}$. Condition (i) of theorem implies that $W(x(0)) < W(\mathcal{N})$. Therefore the results of these two steps show that $ W(x(t)) < W(x(0))$, for all $t>0$, and the Theorem \ref{gsa-thm7} follows directly. Here we show that $\mathcal{N} \subset \mathcal{M}$. From condition (ii), it follows that $\Theta_{\max} < Z_{\max}$. Substituting for $Z_{\max}$ yields
\begin{equation}\label{eqn_ags103}
\begin{array}{l}
 \frac{\rho}{x_{\max}} < \frac{\lambda_{\min}(Q)-(3n+2)Z_B \Theta_{\max} }{ Z_B(2v_{0} + 2 r_{\max}+3 \Theta_{\max}) },
\end{array}
\end{equation}
since by definition $x_{\max} >0$ and also $\rho,~v_{0}, ~r_{\max}, ~\Theta_{\max}$, and $Z_{\max}$ are all positive, hence $(\lambda_{\min}(Q)-(2n+3)Z_B \Theta_{\max}) > 0$. Using the definition of $x_{\min}$ from (\ref{eqn_ags94}) we obtain ${x_{\min}} < \frac{ Z_B(2v_{0} + 2 r_{\max}+3 \Theta_{\max}) }{\lambda_{\min}(Q)-(3n+2)Z_B \Theta_{\max} }$. Hence $\rho x_{\min} < x_{\max}$. In (\ref{eqn_ags98}), $W(x)$ can be lower bounded by $\lambda_{\min}(P)||x(t)||^2 \leq W(x)$, which from (\ref{eqn_ags99}) implies $||x(t)|| \leq x_{\max}$, for all $x(t) \in \mathcal{N}$. In a similar process from equation (\ref{eqn_ags100}), $W(x(t))$ can be upper bounded by $W(x(t)) \leq \lambda_{\max}(P)||x(t)||^2$. From (\ref{eqn_ags99}) and $\rho x_{\min} < x_{\max}$ we obtain $x_{\min} < \frac{1}{\rho} x_{\max} < ||x(t)||$, for all $t>0$. From the definition of $\mathcal{N}$ and $\mathcal{M}$, we conclude that $\mathcal{N} \subset \mathcal{M}$. Now we prove that $\dot{W}(.)<0$, for all $x(t) \in \mathcal{M}$. The first case is when there is no saturation in the control inputs and the second case is when the control inputs are limited by rectangular saturation function.

\textbf{Case I:} $\Delta v(t)=0$ - From Assumption \ref{gsa-ass3}, plant (\ref{eqn_ags821}), and $\tilde{K}(t)=\hat{K}(t)-K^*(t)$, we obtain
\begin{equation}\label{eqn_ags107}
\begin{array}{l}
\dot{x}(t)= A_m(\alpha(t)) x(t) + B \tilde{K}^\mathsf{T}(t)x(t)- B_r r(t), ~~~ \forall \alpha \in \Omega,
\end{array}
\end{equation}
which leads to
\begin{equation}\label{eqn_ags108}
\begin{array}{l}
\dot{W}(.)=x^{\mathsf{T}}(t) ( A^{\mathsf{T}}_m(\alpha(t)) P + P A_m(\alpha(t))) x(t)+ x^{\mathsf{T}}(t)(2PB \tilde{K}^\mathsf{T}(t))x(t) + 2x^{\mathsf{T}}(t) P B_r r(t).
\end{array}
\end{equation}
By tacking bounds on the right hand side of (\ref{eqn_ags108}), we obtain
\begin{equation}\label{eqn_ags109}
\begin{array}{l}
 \dot{W}(.) <  \left(2Z_B \Theta_{\max}-\lambda_{\min}(Q) \right)||x(t)||^2 + 2Z_B r_{\max}||x(t)||.
\end{array}
\end{equation}
From condition (ii) and the definition of $\Theta_{\max}$, we obtain
\begin{equation}\label{eqn_ags110}
\begin{array}{l}
 \Theta_{\max} < Z_{\max} < \frac{\lambda_{\min}(Q)}{Z_B(3n+2)}.
\end{array}
\end{equation}
Therefore
\begin{equation}\label{eqn_ags111}
\begin{array}{l}
 \dot{W}(.) < 0, \\[3pt]
  ||x(t)|| > \frac{2Z_B r_{\max}}{\lambda_{\min}(Q)-2Z_B \Theta_{\max}}.
\end{array}
\end{equation}
The choice of $x_{\min}$ in (\ref{eqn_ags94}) leads to $x_{\min} > \frac{2Z_B r_{\max}}{\lambda_{\min}(Q)-2Z_B \Theta_{\max}}$. Hence it is shown that in Case I, $\dot{W}(.)< 0$, for all $x(t) \in \mathcal{M}$.

\textbf{Case II:} $\Delta v(t) \neq 0$ - Suppose $A(\alpha(t))$ is a Hurwitz for all $\alpha \in \Omega$, and consider the following Lyapunov function candidate for the system dynamics $W_A(x(t))=x^{\mathsf{T}}(t) P_A x(t)$, where $P_A=P_A^{\mathsf{T}} >0$, solves the following LMI
\begin{equation}\label{eqn_ags115}
\begin{array}{l}
A^{\mathsf{T}}(\alpha(t))P_A+P_A A(\alpha(t)) \leq -Q_A, ~~ \forall \alpha \in \Omega,
\end{array}
\end{equation}
for some positive definite $Q_A=Q_A^{\mathsf{T}} >0$. Because $\Delta v(t) \neq 0$, then $R_s(v(t))=\bar{v}(t)$ and the system dynamics in equation (\ref{eqn_ags80}) becomes
\begin{equation}\label{eqn_ags116}
\begin{array}{c}
\dot{x}=A(\alpha(t))x(t) + B \bar{v}(t) + B_r r(t), ~~ \forall \alpha \in \Omega.
\end{array}
\end{equation}
From the definition of $v_0$ in (\ref{eqn_ags94}), we know $||\bar{v}(t)|| \leq v_{0}$. Consequently
\begin{equation}\label{eqn_ags117}
\begin{array}{l}
\dot{W}_A(.)=x^{\mathsf{T}} ( A^{\mathsf{T}}(\alpha(t))P_A + P_A A(\alpha(t))) x(t) + 2 x^{\mathsf{T}}(t) P_A B \bar{v}(t) + 2x^{\mathsf{T}}(t) P_A B_r r(t),\\[5pt]
~~~~ \leq  -\lambda_{\min}(Q_A) ||x(t)||^2 + Z_B \left( 2r_{\max}+2v_0 \right) ||x(t)||.
\end{array}
\end{equation}
For open-loop stable systems it immediately implies that
\begin{equation}\label{eqn_ags118}
\begin{array}{l}
 \dot{W}(.) < 0, ~ ||x(t)|| > \frac{Z_B \left( 2r_{\max}+2v_0 \right)}{\lambda_{\min}(Q_A)}.
\end{array}
\end{equation}
Therefore the system states remain bounded. For unstable systems, that is, when $A(\alpha(t))$ is not Hurwitz, we write the dynamics in the following form
\begin{equation}\label{eqn_ags119}
\begin{array}{c}
\dot{x}(t)=A(\alpha(t))x(t)+B \bar{v}(t) + B_r r(t) + B K^{*\mathsf{T}}(t)x(t) -B K^{*\mathsf{T}}(t) x(t), \\[3pt]
~~~~  =A_m(\alpha(t))x(t) -B K^{*\mathsf{T}}(t) x(t) +B \bar{v}(t) + B_r r(t),      ~~ \forall \alpha \in \Omega.
\end{array}
\end{equation}
Then
\begin{equation}\label{eqn_ags120}
\begin{array}{l}
 \dot{W}(.) \leq -x^\mathsf{T}(t) Q x(t) - 2x^\mathsf{T}(t) PBK^{*\mathsf{T}}(t) x(t)  + 2x^\mathsf{T}(t) PB\bar{v}(t) + 2x^\mathsf{T}(t) PB_r r(t).
\end{array}
\end{equation}
In the following, two sub-cases are considered:
\emph{Sub-case II.a:} $2x^\mathsf{T}(t) PB \bar{v}(t) < -v_{\min} Z_B ||x(t)||$  \\
Using for this sub-case and previously defined bounds, we can bound $\dot{W}(.)$ as
\begin{equation}\label{eqn_ags121}
\begin{array}{l}
 \dot{W}(.) < |\lambda_{\min}(Q)-2Z_B \Theta^*_{\max}| ||x(t)||^2  + \left( 2 Z_B r_{\max}-Z_B v_{\min} \right) ||x(t)||.
\end{array}
\end{equation}
This implies that $\dot{W}(.) < 0, ~ ||x(t)|| \leq \frac{Z_B v_{\min}-2 Z_B r_{\max}}{|\lambda_{\min}(Q)-2Z_B \Theta^*_{\max}|}$. From the definition of $x_{\max}$, we obtain $||x(t)|| \leq  \frac{Z_B v_{\min}-2 Z_B r_{\max}}{|\lambda_{\min}(Q) -2Z_B \Theta^*_{\max}|} < x_{\max}$. Hence we can conclude that
\begin{equation}\label{eqn_ags124}
\begin{array}{l}
 \dot{W}(.) < 0, ~~ \forall x(t) \in \mathcal{M} ~\text{for sub-case II.a}.
\end{array}
\end{equation}
\emph{Sub-case II.b:} $2x^\mathsf{T}PB\bar{v}(t) \geq -v_{\min} Z_B ||x(t)||$  \\
Complexities arise in the stability analysis because the rectangular saturation function does not necessarily preserve the direction of the control inputs as they hit their limits. Therefore as defined in equation (\ref{eqn_ags73}), $\bar{v}(t)$ is decomposed into $v_d(t)$ and $\tilde{v}(t)$ as
\begin{equation}\label{eqn_ags125}
\begin{array}{l}
 \bar{v}(t)= v_d(t)+\tilde{v}(t) = \frac{v(t)}{||v(t)||}||v_d(t)|| + \tilde{v}(t),
\end{array}
\end{equation}
and $v_d(t)$ is chosen such that $||v_d(t)|| \geq \max \left[ ||\tilde{v}(t)||,~ v_{\min} \right]$. The decomposition can be constructed without loss of generality. The condition to this sub-case implies that
\begin{equation}\label{eqn_ags127}
\begin{array}{l}
 2x^\mathsf{T}(t) PB \frac{v(t)}{||v(t)||}||v_d(t)|| + v_{\min} Z_B ||x(t)|| + 2x^\mathsf{T}(t)PB \tilde{v}(t) \geq 0.
\end{array}
\end{equation}
Multiplying $\frac{||v(t)||}{||v_d(t)||}$ in (\ref{eqn_ags127}), we obtain
\begin{equation}\label{eqn_ags128}
\begin{array}{l}
 2x^\mathsf{T}(t) PB v(t) + v_{\min} Z_B ||x(t)|| \frac{||v(t)||}{||v_d(t)||} + 2x^\mathsf{T}(t) PB \tilde{v}(t) \frac{||v(t)||}{||v_d(t)||} \geq 0.
\end{array}
\end{equation}
Since $v_d(t)$ in equation (\ref{eqn_ags125}) is chosen such that $\frac{v_{\min}}{||v_d(t)||}<1$ and $\frac{||\tilde{v}(t)||}{||v_d(t)||}<1$ hold, we have
\begin{equation}\label{eqn_ags129}
\begin{array}{l}
 2x^\mathsf{T}(t) PB v(t) + 3Z_B ||x(t)|| ||v(t)|| \geq 0.
\end{array}
\end{equation}
Adding the inequality (\ref{eqn_ags120}) to the inequality (\ref{eqn_ags129}), we obtain
\begin{equation}\label{eqn_ags130}
\begin{array}{l}
 \dot{W}(.) \leq -x^\mathsf{T}(t) Q x(t) + 2x^\mathsf{T}(t) PB \left( \hat{K}^\mathsf{T}(t)-K^{*\mathsf{T}}(t) \right) x(t)\\[3pt]
  ~~ + 2x^\mathsf{T}(t) PB \bar{v}(t) + 2x^\mathsf{T}(t) PB_r r(t) + 3Z_B ||x(t)|| ||v(t)||.
\end{array}
\end{equation}

Note that $||v(t)|| \leq \Theta^*_{\max} ||x(t)|| \leq n \Theta_{\max} ||x(t)||$, and $||\bar{v}(t)|| \leq v_0$, as a result we have $ \dot{W}(.) \leq \left( (3n+2)Z_B \Theta_{\max} -\lambda_{\min}(Q) \right) ||x(t)||^2 + Z_B \left( 2 v_0 + 2 r_{\max}\right) ||x(t)||$. From equation (\ref{eqn_ags110}), we know $(3n+2)Z_B \Theta_{\max}-\lambda_{\min}(Q) < 0$, and then we have $\dot{W}(.) < 0$ for $||x(t)|| > \frac{Z_B \left(2v_0+2r_{\max}\right)}{\lambda_{\min}(Q)-(3n+2)Z_B \Theta_{\max}} := x_{\min}$. From the definition of $x_{\min}$, we conclude that
\begin{equation}\label{eqn_ags133}
\begin{array}{l}
 \dot{W}(.) < 0, ~~ \forall x(t) \in \mathcal{M} ~\text{for sub-case II.b}.
\end{array}
\end{equation}
As a consequence of (\ref{eqn_ags124}) and (\ref{eqn_ags133}), it follows that $\dot{W}(.) < 0$, for all $x(t) \in \mathcal{M}$.
\end{proof}
\begin{rmk}\label{gsa-rmk4}
Theorem \ref{gsa-thm7} implies that if the initial conditions of the state and the parameter error lie within certain bounds, then the adaptive system will have bounded solutions. The local nature of the result for unstable systems is because of the saturation limits on the control input. For open-loop stable systems the results are global.
\end{rmk}

\subsection{Turboshaft Engine Example}
We apply the developed adaptive controller to a high fidelity physics-based model of JetCat SPT5 turboshaft engine driving a variable pitch propeller developed in \cite{fitzgerald-model-2013, pakmehr-decentmodel-2011}. The effect of engine degradation due to aging is modeled in the nonlinear simulation by modifying the efficiencies and flow capacities of key engine components such as: High Pressure Compressor (HPC), High Pressure Turbine (HPT) and Low Pressure Turbine (LPT). The values of these parameters used in this simulation are $\eta_{hpc}=-1.470 \%$, $W_{c,hpc}=-2.455 \%$, $\eta_{hpt}=-1.315 \%$, $W_{c,hpt}=+0.880 \%$, $\eta_{lpt}=-0.269 \%$, and $W_{c,lpt}=+0.1294 \%$, where their nominal values are zero. To show the stability of the closed-loop reference system, 40 different linearizations of the system have been used, to solve inequality (\ref{eqn_gs115}), in Matlab with the aid of YALMIP \cite{YALMIP-lofberg-2004} and SeDuMi \cite{sedumi-Sturm-2001} packages. The numerical value for the common matrix $P$ is
\begin{eqnarray*}
P = \left[
    \begin{array}{cccccc}
        0.491  &  0.079  &  0.102  & -0.004  & -0.072 & -0.039\\
        0.079  &  0.446  &  0.053  &  0.007  & -0.097 & -0.013\\
        0.102  &  0.053  &  0.181  & -0.041  & -0.028 & -0.022\\
       -0.004  &  0.007  & -0.041  &  0.130  &  0.023 &  0.013\\
       -0.072  & -0.097  & -0.028  &  0.023  &  0.321 &  0.045\\
       -0.039  & -0.013  & -0.022  &  0.013  &  0.045 &  0.332
    \end{array}
    \right],
\end{eqnarray*}
where its condition number is $\kappa(P)=6.6303$, and $Q=0.1 \times I_6$. Simulations are conducted for two different cases including the control of the nominal model (NomEng), and the control of the deteriorated engine due to aging (AgedEng). These case studies, simulate the engine acceleration from the idle thrust to the cruise condition and then its deceleration back to the idle condition for a standard day at sea level condition. Simulation results are shown in Figures \ref{comp_map_stab_gsa} to \ref{fig_gsa_16}.

Figure \ref{comp_map_stab_gsa} shows the JetCat SPT5 turboshaft engine compressor map. In this map the approximate stall line and also the operating line for this simulation have been shown. The engine operates in a safe region with a big stall margin during its acceleration from idle to cruise and again during its deceleration back to the idle condition. The 40 points which are used for linearization and stability analysis of the closed-loop system also have been shown in this figure. Thirty of these points correspond to equilibrium linearizations which are situated on the steady-state operating line of the engine, and the other 10 points correspond to non-equilibrium linearizations which are situated near the steady-state operating line of the engine. The engine operating lines for the nominal engine and degraded engine model are shown in this figure; as degradation increases in the engine, the stall margin decreases, the pressure ratio decreases and the turbine temperature increases.
\begin{figure}[!ht]
\centering
\includegraphics[width=0.7\textwidth]{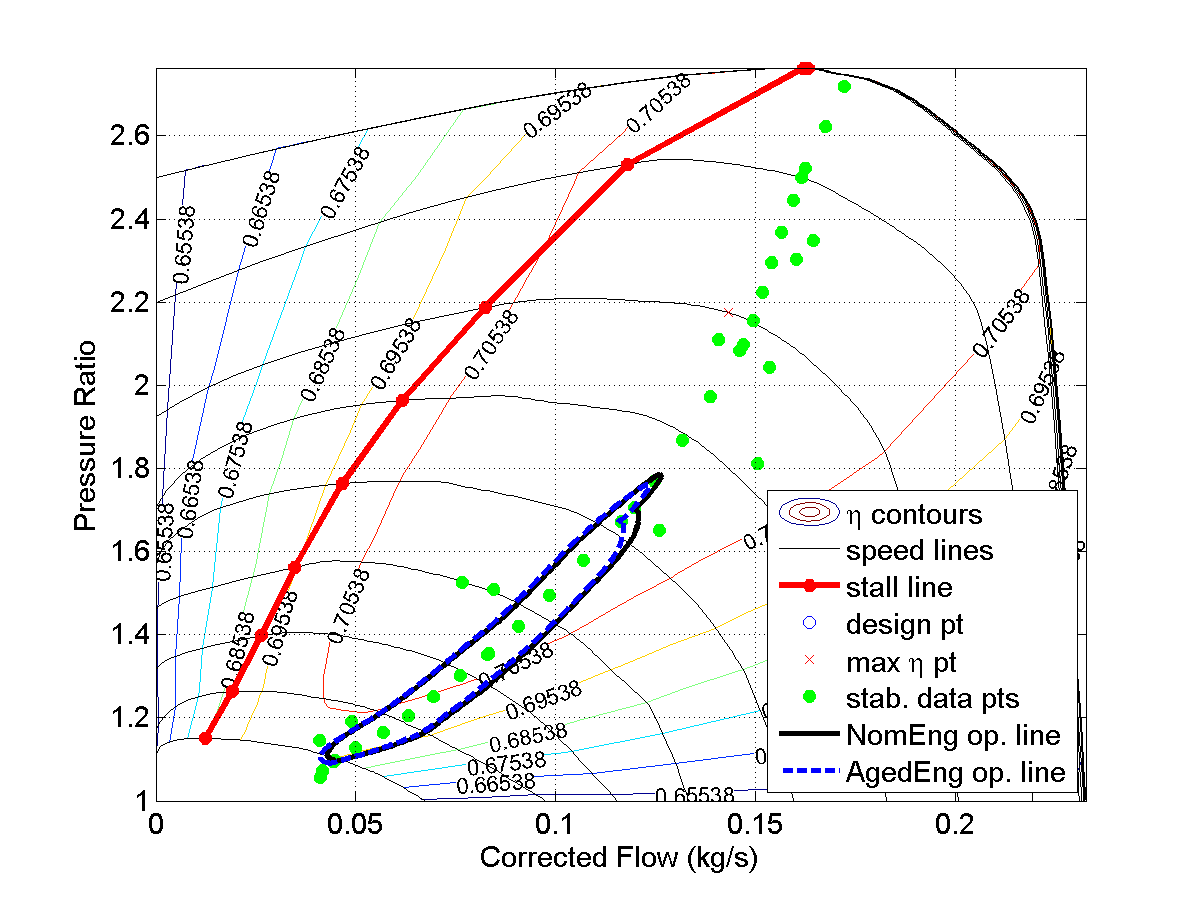}
\caption{JetCat SPT5 engine compressor map with data points used to compute $P$ and operating line for nominal engine, and deteriorated engine.}\label{comp_map_stab_gsa}
\end{figure}

Other controller parameters are $\epsilon_c=1,~\eta_c=3$. The numerical values for the adaptive controller are set as $\Gamma =\text{diag}([50,~50,~50,~50,~50,~50])$, $\Gamma_{\Delta} =\text{diag}([30,~30])$, $v_{1,\max}=0.12, ~v_{2,\max}=0.15$, and the initial conditions and the compact sets are
\begin{equation} \label{eqn2_agsf6}
\begin{array}{c}
     \Theta_{K_i}= \left[
       \begin{array}{cc}
          $[-2,  0]$   &  $[-2,  0]$ \\
          $[-2,  0]$   &  $[-2,  0]$
       \end{array}
    \right], \\[5pt]
    \Theta_{\Delta} = \left[
       \begin{array}{cccccc}
           $\{0\}$ & $\{0\}$ & $[0, 10]$ &  $\{0\}$  &  $\{0\}$ & $\{0\}$ \\
           $\{0\}$ & $\{0\}$ &  $\{0\}$  & $[0, 10]$ &  $\{0\}$ & $\{0\}$
       \end{array}
       \right]^{\mathsf{T}}, \\[5pt]
   \hat{K}_i(0)= \left[
       \begin{array}{cc}
          -0.1950   &  -0.1950 \\
          -0.1970   &  -0.1970
       \end{array}
    \right], \\[5pt]
    K_{\Delta}(0)= \left[
       \begin{array}{cccccc}
           0 & 0 & 2.7 &  0  &  0& 0 \\
           0 & 0 &  0  & 2.7 &  0& 0
       \end{array}
       \right]^{\mathsf{T}}.
\end{array}
\end{equation}
Figures \ref{fig_gsa_07} and \ref{fig_gsa_08} show the high and low pressure spool speeds tracking their reference trajectories closely.
\begin{figure}[!ht]
\centering
\begin{minipage}[b]{0.49\textwidth}
\centering
\includegraphics[width=0.99\textwidth]{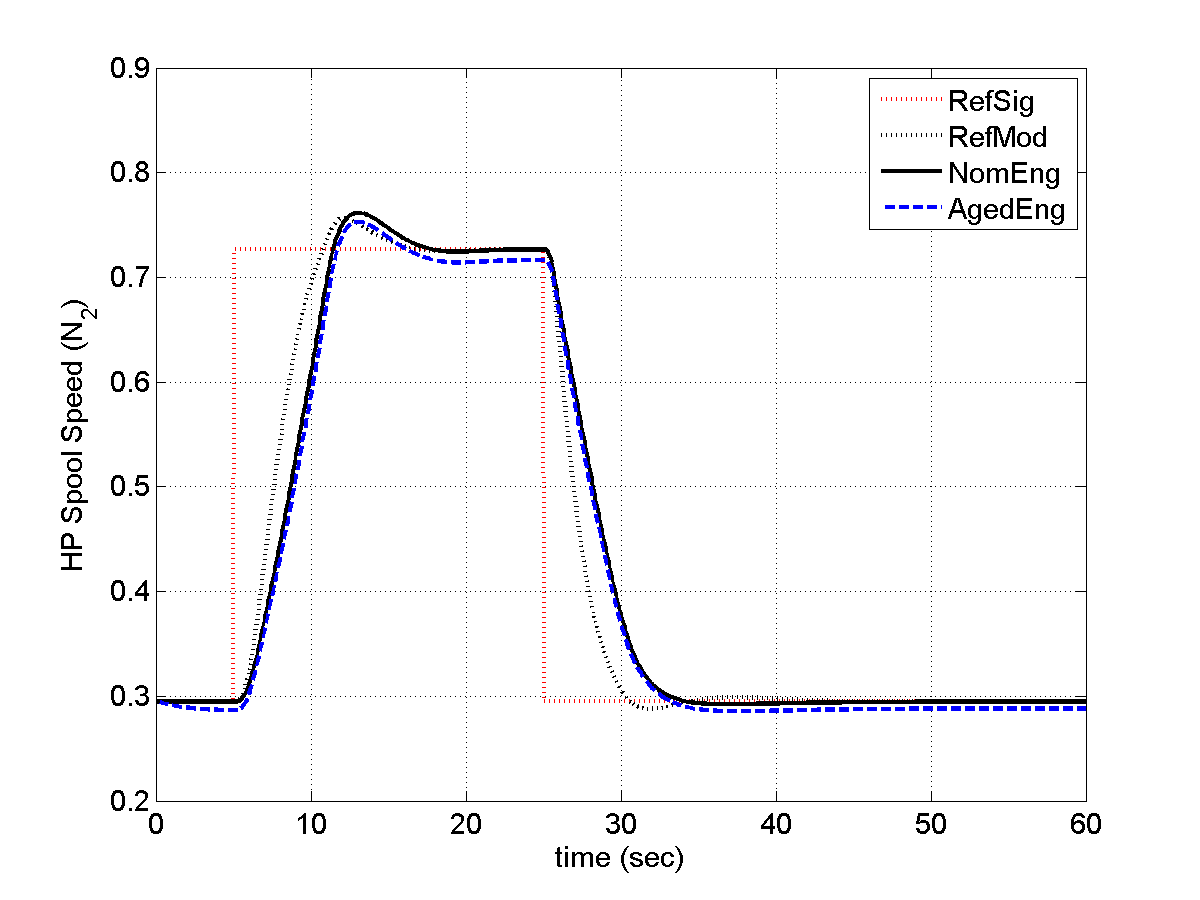}
\caption{High pressure spool speed and its reference signal}\label{fig_gsa_07}
\end{minipage}
\hfill
\begin{minipage}[b]{0.49\textwidth}
\centering
\includegraphics[width=0.99\textwidth]{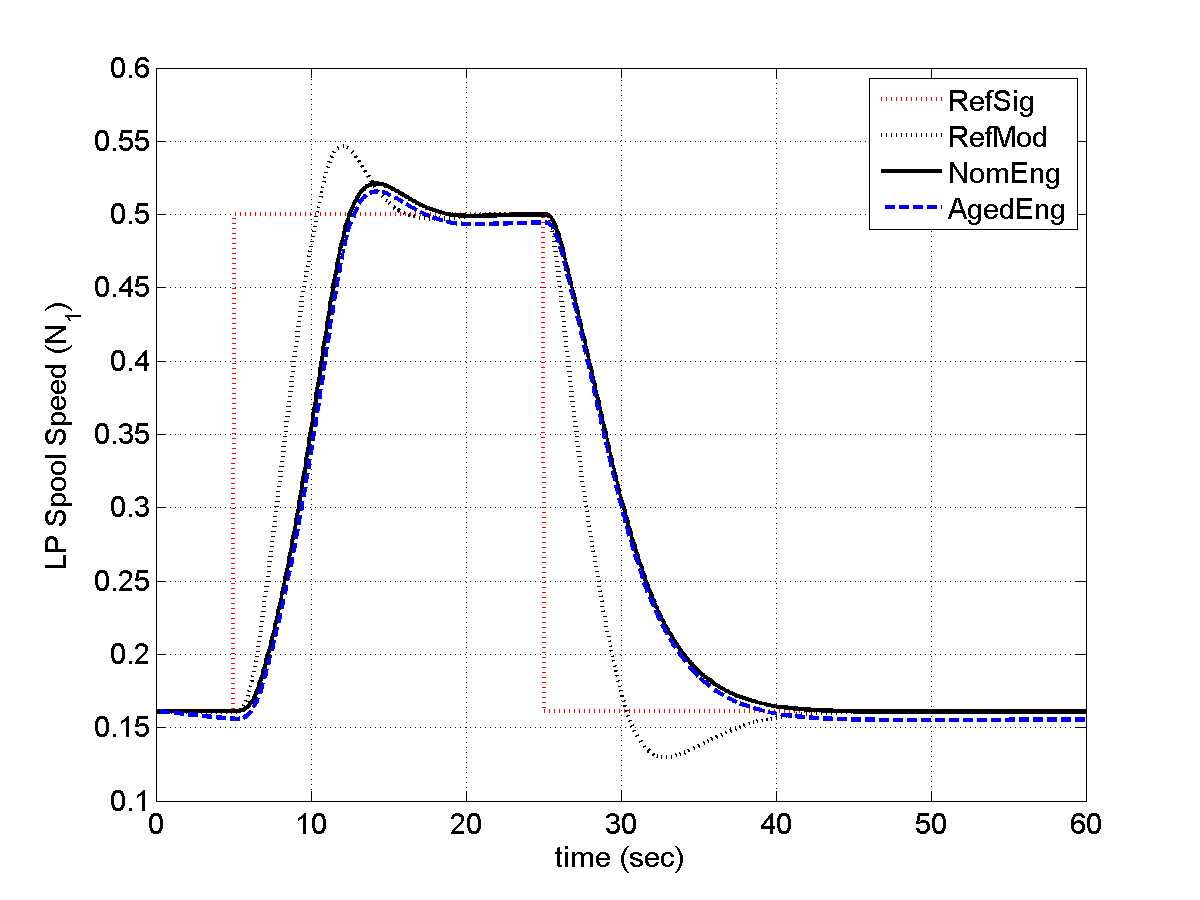}
\caption{Low pressure spool speed and its reference signal.}\label{fig_gsa_08}
\end{minipage}
\end{figure}

\begin{figure}[!ht]
\centering
\includegraphics[width=0.5\textwidth]{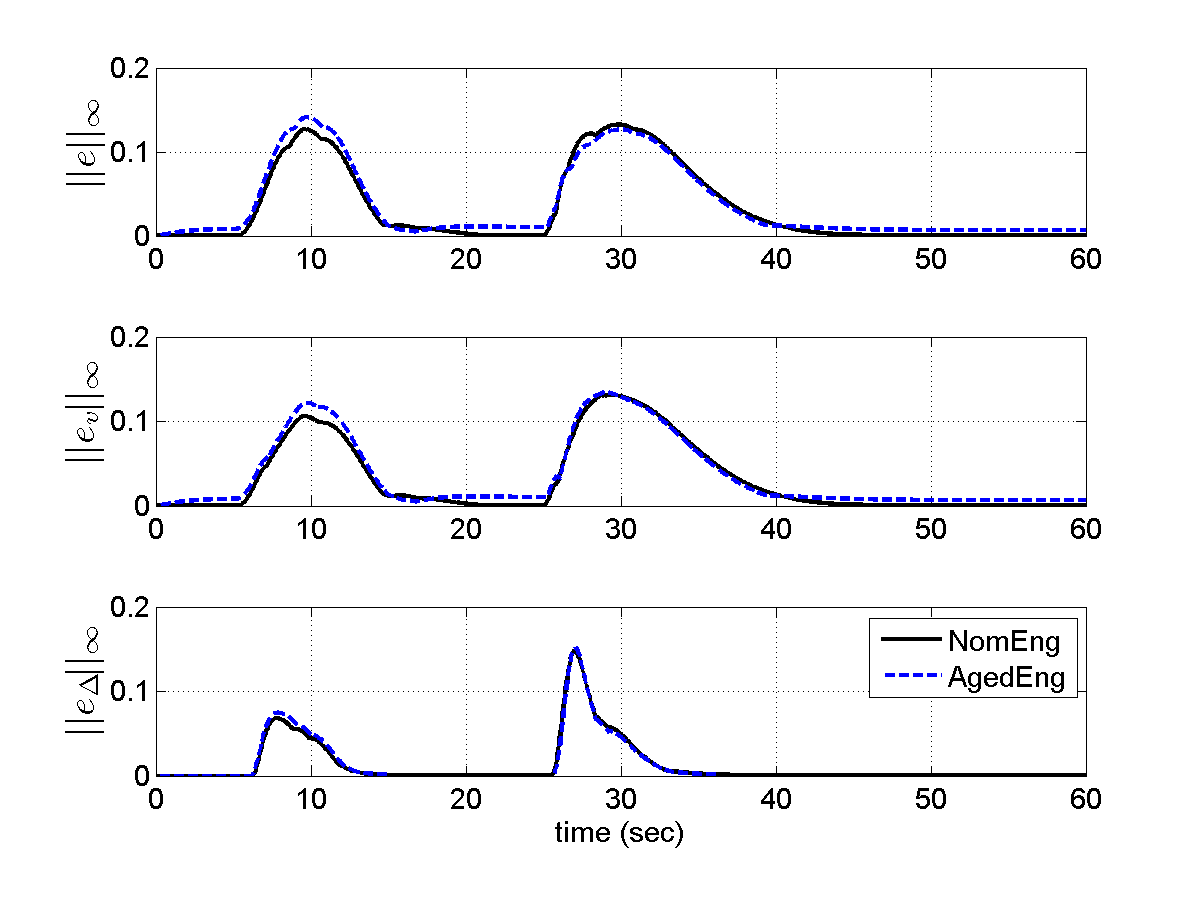}
\caption{Norm of the error signals $||e(t)||_{\infty}$, $||e_v(t)||_{\infty}$, $||e_{\Delta}(t)||_{\infty}$ }\label{fig_gsa_10}
\end{figure}

Figure \ref{fig_gsa_10}, shows the evolution of the infinity norm of the errors $||e(t)||_{\infty}$, $||e_v(t)||_{\infty}$, $||e_{\Delta}(t)||_{\infty}$. The steady-state error in the Aged Engine (AgedEng) simulation case is because of the effect of the aging on the engine health parameters, and this causes a change in the equilibrium manifold of the aged engine in comparison to the nominal engine (NomEng). In other words, since we are using nominal engine equilibrium manifold to design a linear parameter dependant reference model, and the aged engine linear model has a different equilibrium manifold $x^p_{e, nom}(\alpha(t))\neq x^p_{e, aged}(\alpha(t))$, and $u_{e, nom}(\alpha(t))\neq u_{e, aged}(\alpha(t))$, then $\delta x^p_{aged}(t)= x^p_{aged}(t)-x^p_{e, aged}(\alpha(t)) \neq 0$, and $\delta u_{aged}(t)= u_{aged}(t)-u_{e, aged}(\alpha(t)) \neq 0$, and this means $||\delta x_{aged}(t)|| > \delta x_{\min} \neq 0$ for all $t > 0$, hence, there will be a steady-state error value greater than zero.

\begin{figure}[!ht]
\centering
\begin{minipage}[b]{0.49\textwidth}
\centering
\includegraphics[width=0.99\textwidth]{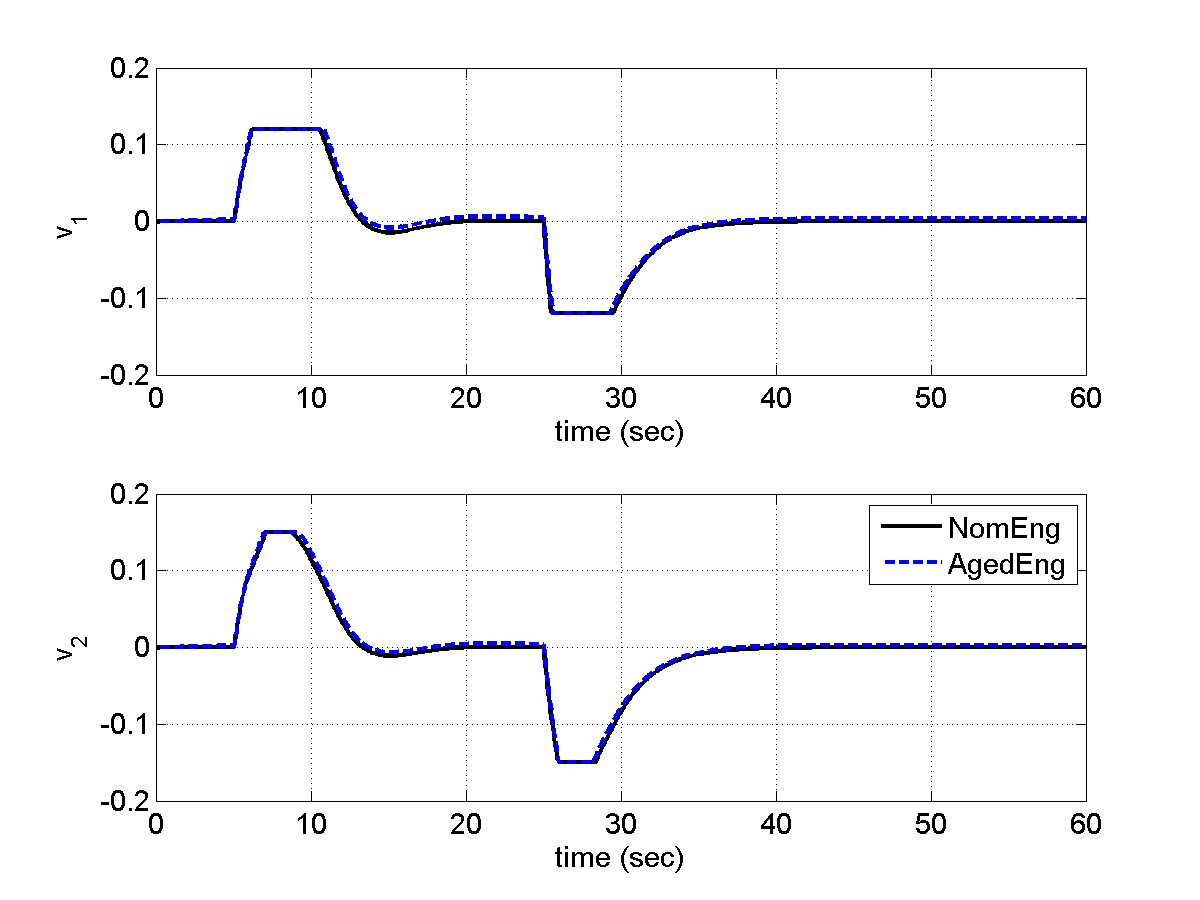}
\caption{Control inputs to the augmented system ($v(t)$)}\label{fig_gsa_11}
\end{minipage}
\hfill
\begin{minipage}[b]{0.49\textwidth}
\centering
\includegraphics[width=0.99\textwidth]{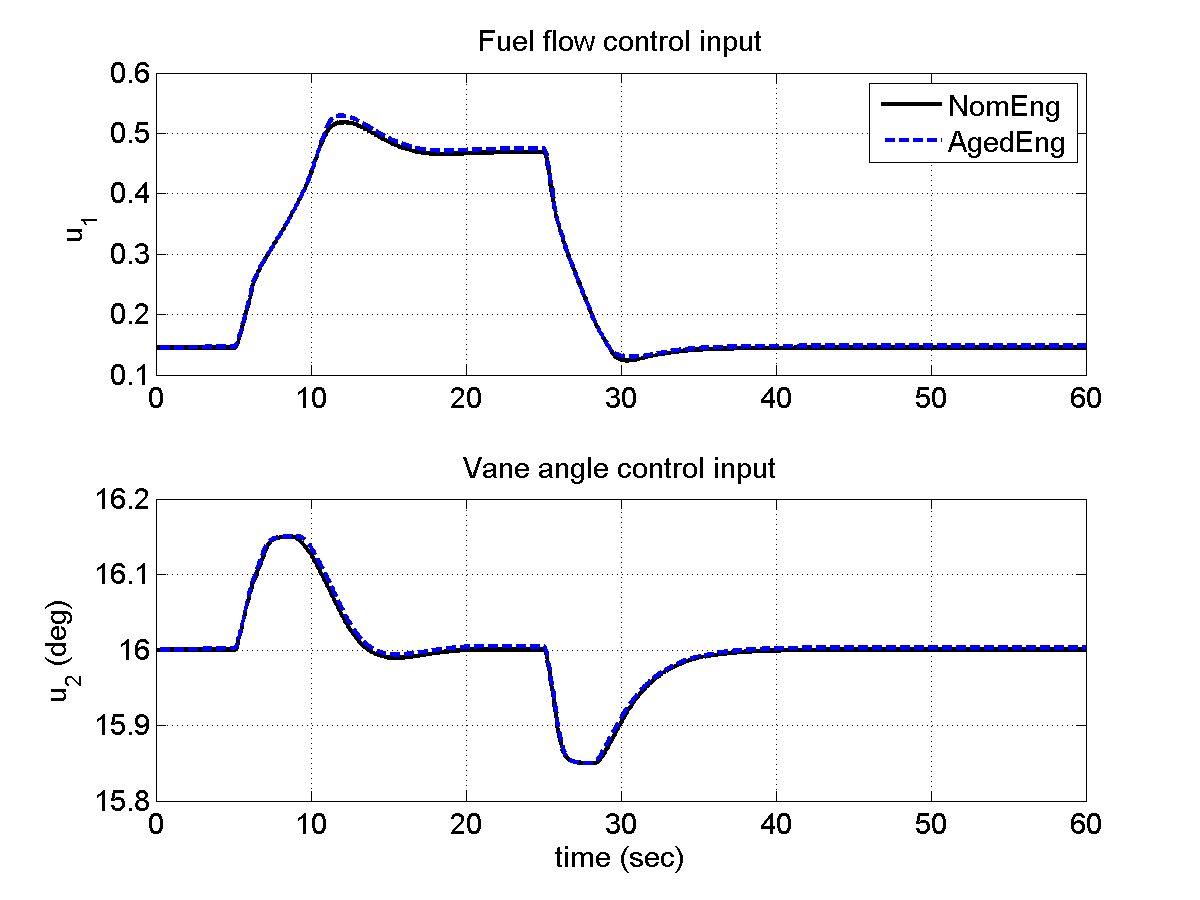}
\caption{Fuel and prop pitch angle control inputs ($u(t)$)}\label{fig_gsa_14}
\end{minipage}
\end{figure}

\begin{figure}[!ht]
\centering
\begin{minipage}[b]{0.49\textwidth}
\centering
\includegraphics[width=0.99\textwidth]{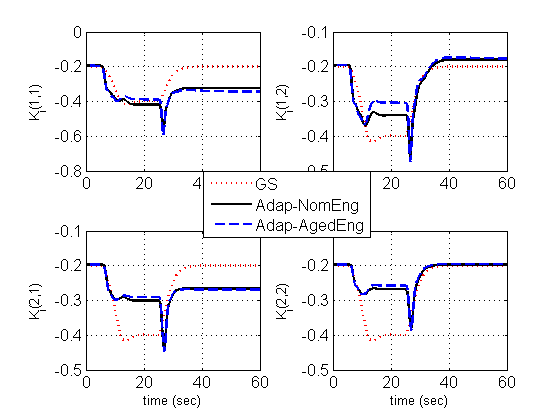}
\caption{Integral gain matrix elements for the gain scheduling controller ($K_i(\alpha)$), and for the adaptive controller ($\hat{K}_i(t)$)}\label{fig_gsa_15}
\end{minipage}
\hfill
\begin{minipage}[b]{0.49\textwidth}
\centering
\includegraphics[width=0.99\textwidth]{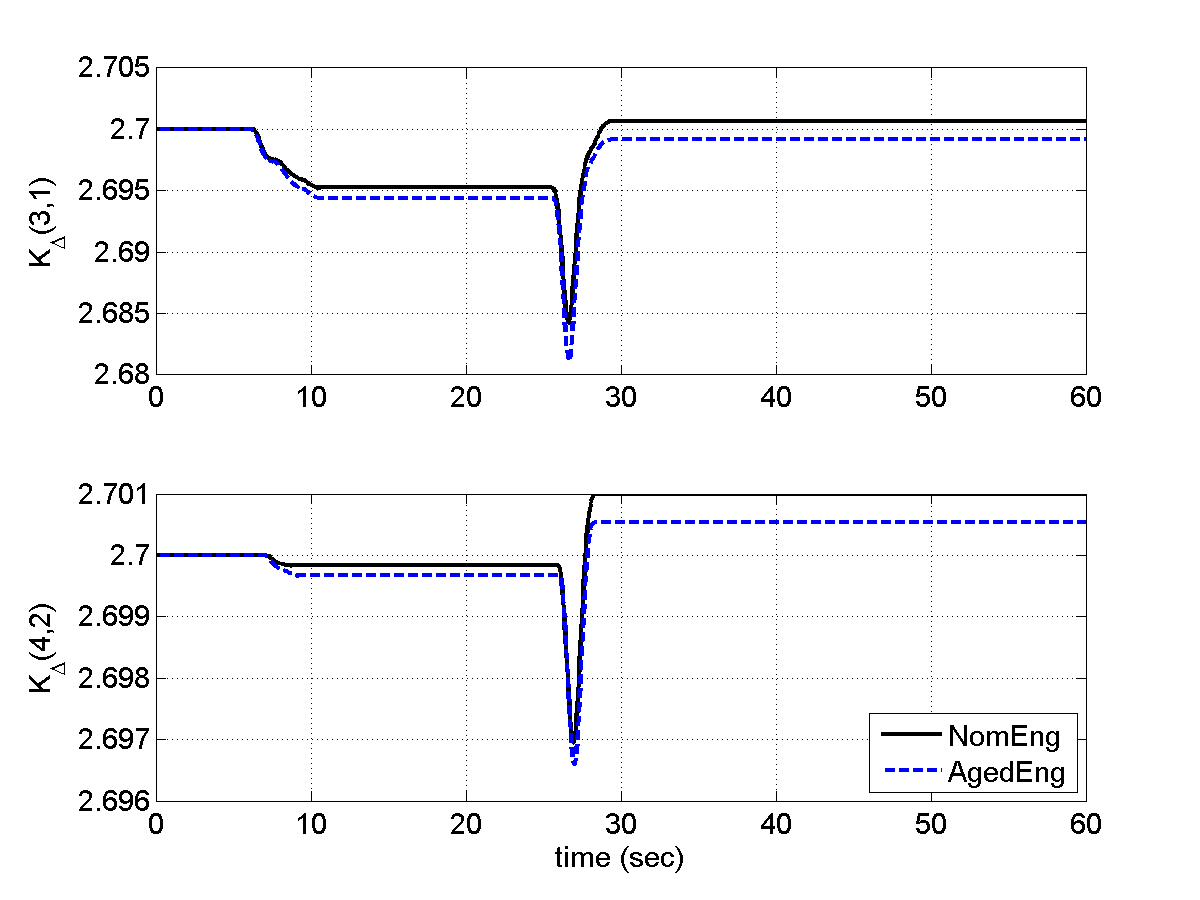}
\caption{Nonzero elements of the augmented adaptive parameter for the saturated system ($K_{\Delta}(t)$)}\label{fig_gsa_16}
\end{minipage}
\end{figure}

Figure \ref{fig_gsa_11} shows the evolution of the control inputs $v(t)=[v_1(t), v_2(t)]^T$, which are inputs to the augmented system, each element is corresponding to one of the control inputs to the original system. For better performance and also to keep the engine in the safe range of operation, hard limits have been defined for both augmented control inputs, $|v_i| \leq v_{i,\max}$ for $i=1,2$. Figure \ref{fig_gsa_14} shows the histories of fuel flow and propeller pitch angle as the control inputs to the plant. Figures \ref{fig_gsa_15} shows the evolution of the gain scheduling controller integral gain matrix ($K_i(\alpha)$) and also adaptive controller gain matrix ($\hat{K}_i(t)$). Figure \ref{fig_gsa_16} shows the evolution of the nonzero elements of the augmented adaptive parameter for the saturated system ($K_{\Delta}(t)$). 

For the cases where engine overspeed, compressor surge/stall, and/or turbine temperature limits are of concern, the adaptive controller with constrained control inputs could potentially be used to mitigate these issues. In order to keep the gas turbine engine to run within its safe operating envelope, proper selection of the constraints on the control inputs, based on the knowledge of the engine dynamics, is needed.  

\newpage
\subsection{Conclusions}
Using convex optimization tools, a single quadratic Lyapunov function was computed, which guarantees the stability of the gain scheduled gas turbine engine reference model. Stability analysis was performed for the developed adaptive control with constrained control inputs architecture by proving the ultimate boundedness of the error signal. Sufficient conditions for ultimate boundedness of the closed-loop system were derived. A semi-global stability result was proved with respect to the level of saturation for open-loop unstable plants while the stability result becomes global for open-loop stable plants. Simulations result for adaptive control of JetCat SPT5 turboshaft engine physics-based model with some degradation due to aging, shows that the proposed adaptive controller tracks the reference model.

\section {Adaptive Control of Systems with Gain Scheduled Reference Models - Part III: Decentralized Approach}

\subsection {Introduction}

During the past decades there has been a growing interest in decentralized adaptive control \cite{decent-ioannou-1986, decent-gavel-1989, decentbook-siljak-1990, decent-shi-1992, solution-ortega-1993, DecentCont-siljak-1996, CoorDecent-Hovakimyan-2005, decent-narendra-2006}. The problem deals with a system composed of $N$ subsystems $S_k$, each of whose inputs is chosen by $N$ controllers $C_k$, where $k = 1, 2, ..., N$. The parameters of the subsystems are assumed to be unknown, and the controllers have to generate their inputs adaptively, using all information available to them, to achieve some desired objectives. 

Control theoretic concepts for gain scheduled model reference adaptive control of gas turbine engines have been developed in \cite{GSstability-pakmehr-2013, PhDThesis-pakmehr-2013}. Since controlling the systems which operate in large operating envelopes, such as gas turbine engines, is not practical close to just one operating point, there is a need to use linear parameter dependent models that cover the entire operating envelope of the system (i.e. control the system for multiple operating points). The contribution of this section is the development of a decentralized adaptive control approach for systems with gain scheduled reference models; the controller can be used to control the dynamical systems with multiple subsystems over large operating envelopes for a continuum of equilibria. The decentralized adaptive control design developed here is based on the results from \cite{GSstability-pakmehr-2013, PhDThesis-pakmehr-2013}. The developed decentralized controller, then is applied to a high fidelity physics-based model of a JetCat SPT5 turboshaft engine with two subsystems. Using this architecture, we can match different engine cores to different props, and the whole propulsion system could work without anymore performance tuning. Simulation results show that the gas turbine engine with two subsystems (i.e., engine core and engine propeller) can be controlled for large throttle commands in a stable manner and with proper tracking performance. 
 
The rest of this section is organized as follows. In subsection II, decentralized linear parameter dependent modeling is presented. In subsection III, the decentralized adaptive control for systems with gain scheduled reference systems is presented. Then, uniform ultimate boundedness of the error signals for all the subsystems of the decentralized system is proven. In subsection IV, simulation results are presented. The simulations studies the efficiency of the developed decentralized adaptive control architecture for controlling the system with a new engine core subsystem along with the nominal engine prop subsystem. Subsection V, concludes this section. 


\subsection{Decentralized Linear Parameter Dependent Modeling}
Here, a decentralized version of plant (\ref{eqn_gs111}) is described. Each one of the subsystems is modeled as a single input, single output (SISO) sub-plant. Each subsystem with its filtered input and its controller can be defined as
\begin{equation} \label{eqn_dgsa1111}
\begin{array}{l}
  \underbrace{\left[
       \begin{array}{c}
           \dot{x}_k^p(t) \\
           \dot{u}_k(t) \\
           \dot{x}_k^c(t)
       \end{array}
    \right]}_{\dot{x}_k} =
     \underbrace{\left[
       \begin{array}{c}
           f_k^p(x^p(t),u(t)) \\
           -\eta_c u_k(t)  \\
           f_k^c(x_k^c(t),g_k^p(x_k^p(t),u_k(t)) ,r_k(t))
       \end{array}
    \right]}_{f_k(x_k(t),x_q(t),r_k(t))} +  \underbrace{\left[
       \begin{array}{c}
           0 \\
           \eta_c \\
           0
       \end{array}
    \right]}_{b_k} v_k(t), \\[5pt]
v_k(t)=\underbrace{g_k^c(x_k^c(t),g_k^p(x_k^p(t),u_k(t)), r_k(t))}_{g_k(x_k(t),r_k(t))},
\end{array}
\end{equation}
and the closed-loop nonlinear subsystem can be written as
\begin{equation}\label{eqn_dgsa1121}
\begin{array}{l}
\dot{x}_k(t)=F_k(x_k(t),x_q(t), r_k(t)),
\end{array}
\end{equation}
where $x_k(t) \in D_{x_k} \subset \Re^{n_k+2}$, and $r_k(t) \in D_{r_k} \subset \Re$, and $x_q(t) \in D_{x_q}$ includes all the states from the other subsystems interconnecting with the $k$th subsystem. Now, similar to controller (\ref{eqn_gs65}), for all $\alpha \in \Omega$, the parameter dependent controller for each subsystem is defined as
\begin{equation} \label{eqn_dgsa651}
\begin{array}{l}
  \left[
       \begin{array}{c}
           \dot{x}_k^c(t) \\
            v_k(t)
       \end{array}
    \right] =
     \left[
       \begin{array}{ccc}
           -\epsilon_c  & 1 &  -1\\
           k_{i,k}(\alpha(t)) & 0 & 0
       \end{array}
    \right] ~
     \left[
       \begin{array}{c}
            x_k^c(t) \\
            \delta y_k(t) \\
            \delta r_k(t)
       \end{array}
    \right].
\end{array}
\end{equation}
 The linear family of systems for the augmented subsystem (\ref{eqn_dgsa1111}) becomes
\begin{equation}\label{eqn_dgsmrac1}
\begin{array}{l}
\delta \dot{x}_k(t) = A_k(\alpha(t)) \delta x_k(t) + b_k v_k(t) + b_{r_k}\delta r_k(t) + \underbrace{\sum_{q=1, q\neq k}^{N} \left [A_{kq}(\alpha(t)) \delta x_q(t)\right]}_{h_k(\delta x_q(t),\alpha(t))}, ~~~\forall \alpha \in \Omega, \\[5pt]
\delta y_k(t)=C_k\delta x_k(t),
\end{array}
\end{equation}
with the state feedback controller
\begin{equation}\label{eqn_dgsmrac1111}
\begin{array}{l}
v_k(t)= K_k^{\mathsf{T}}(\alpha(t)) \delta x_k(t), ~~~\forall \alpha \in \Omega,
\end{array}
\end{equation}
where $\delta x_k(0)=\delta x_{0_k}$, and $\delta x_k(t) \in \Re^{n_k+2}$ is the $k$th subsystem state vector, $v_k(t) \in \Re$ is the $k$th subsystem control input, and $K_k^{\mathsf{T}}(\alpha(t)) \in \Re^{n_k+2}$ is the vector of parameter dependent control gains for subsystem $k$, and $\delta r_k(t) \in \Re$ is the $k$th subsystem reference signal. $h_k(\delta x_q(t),\alpha(t))$ is the interconnection of all other subsystems on the $k$th subsystem. Subscript \textbf{k} represents the $k$th subsystem, where $k \in \{1, ... , N\}$; in turboshaft engine control example $k \in \{Co,Pr\}$. To design a reference model for each subsystem, we ignore the effects of the interconnection terms from other subsystems and for a desired performance, we find out the specific controller $K_k(\alpha(t))=[0,~0,~k_{i,k}(\alpha)]^\mathsf{T}$; as a result we obtain the following closed-loop system
\begin{equation} \label{eqn_dgsa1113}
\begin{array}{l}
  \underbrace{\left[
       \begin{array}{c}
           \delta \dot{x}_k^p(t) \\
           \delta \dot{u}_k(t) \\
            \dot{x}_k^c(t)
       \end{array}
    \right]}_{\delta \dot{x}_{m,k}(t)} =
     \underbrace{\left[
       \begin{array}{ccc}
           A_k^p(\alpha(t)) &~~~ b_k^p(\alpha(t)) &~~~ 0 \\
           0 &~~~ -\eta_c  & ~~~ \eta_c k_{i,k}(\alpha(t)) \\
           1      &~~~ 0  &~~~ -\epsilon_c
       \end{array}
    \right]}_{A_{m,k}(\alpha(t))} 
     \underbrace{\left[
       \begin{array}{c}
           \delta x_k^p(t) \\
           \delta u_k(t) \\
            x_k^c(t)
       \end{array}
    \right]}_{\delta x_{m,k}(t)} +\underbrace{\left[
       \begin{array}{c}
            0 \\
            0 \\
           - 1
       \end{array}
    \right]}_{b_{r_k}} \delta r_k(t), ~~ \forall \alpha \in \Omega.
\end{array}
\end{equation}
The stability of reference model for each subsystem is guaranteed by Lemma \ref{gsa_lem1}.
\begin{rmk}\label{dgsma_rmk_21}
Using pre-designed linear controllers available for important operating points of the system, $K^{\mathsf{T}}_{i,k}(\alpha(t))$ can be obtained based on a stability preserving interpolation approach described in \cite{interp-stilwell-2000} with respect to the scheduling parameter $\alpha$ in a smooth, continuous way. An approach by which the interpolated controller stabilizes the linearized plant for all $\alpha \in \Omega$. Another approach is to compute $K^{\mathsf{T}}_{i,k}(\alpha(t))$ by polynomial approximation as a function of $\alpha$.
\end{rmk}
\subsection{Decentralized Adaptive Control}

\subsubsection{Control Design and Stability Analysis}
Consider a system $S$ consists of $N$ subsystems $S_1, S_2,..., S_N$ that are interconnected. Each of the subsystems is modeled as a single input, single output (SISO) linear parameter dependent model. For convenience, we shall assume that each subsystem $S_k$ has a controller $C_k$ which computes the control input $u_k$ to $S_k$. The subsystems $S_k$ are described by the equations
\begin{equation}\label{eqn_dgsmrac6}
\begin{array}{l}
S_k: \delta \dot{x}_k(t) = A_k(\alpha(t)) \delta x_k(t) + b_k v_k(t) + b_{r_k} \delta r_k(t) + \underbrace{\sum \limits_{q=1, q\neq k}^{N} \left [A_{kq}(\alpha(t)) \delta x_q(t)\right]}_{h_k(\delta x_q(t),\alpha(t))}, ~~~\forall \alpha \in \Omega,\\[5pt]
\delta y_k(t)=C_k \delta x_k(t),
\end{array}
\end{equation}
where $\delta x_k(0)=\delta x_{0_k}$, and $\delta x_k(t) \in \Re^{n_k}$ is the $k$th subsystem state vector, $v_k(t) \in \Re$ is the $k$th subsystem control input, and $\delta r_k(t) \in \Re$ is the $k$th subsystem reference signal. $h_k(\delta x_q(t),\alpha(t))$ is the interconnection of all other subsystems on the $k$th subsystem. Note that $\delta x(t)=[\delta x_1^{\mathsf{T}}(t), ..., \delta x_k^{\mathsf{T}}(t)..., \delta x_N^{\mathsf{T}}(t)]^{\mathsf{T}}$. Subscript \textbf{k} represents the $k$th subsystem, where $k \in \{1, ... , N\}$.
\begin{ass} \label{dgsa_ass2}
For the interconnection term $h_k(\delta x_q(t),\alpha(t))$, there exist positive constants $c_{kq} \in \Re$, for each subsystem $q\neq k$, such that, it is satisfying $\left\|h_k(\delta x_q(t),\alpha(t))\right\| \leq \sum \limits_{q=1, q\neq k}^{N} \left[ c_{kq} ||\delta x_q(t)|| \right]$, for all $\alpha \in \Omega$. 
\end{ass}
\begin{rmk}\label{gsa_rmk23}
This assumption is a result of Assumption \ref{gsa_ass1}, which is about the boundedness of $A_m(\alpha(t))$.
\end{rmk}
\begin{rmk}\label{gsa_rmk22}
The feasibility of this assumption has already been verified in \cite{PhDThesis-pakmehr-2013} by numerical simulation studies, for gas turbine engine applications which we consider as the main application of this work.  For other systems, modeling and numerical studies are also needed for such verification.
\end{rmk}
The linear parameter varying reference model for the $k$th subsystem is expressed as
\begin{equation}\label{eqn_dgsmrac8}
\delta \dot{x}_{m,k}(t) = A_{m,k}(\alpha(t)) \delta x_{m,k}(t) + b_{r_k} \delta r_k(t),  ~~~ \forall \alpha \in \Omega,
\end{equation}
where $\delta r_k(t) \in \Re$ is a bounded continuous reference input signal. The parameter matrix $A_{m,k} \in \Re^{n_k \times n_k}$ is chosen with $A_{m,k}$ being Hurwitz. The boundedness of all the reference trajectories is required in a decentralized tracking control problem, which has been showed in the previous section. Note that $\delta r_k(t) \in \Re$ is the command signal such that $||\delta r_k(t)|| \leq r_{\max, k}$.

The decentralized adaptive control of a linear parameter dependent systems can be stated as follows: Given $N$ subsystems described by (\ref{eqn_dgsmrac6}), and $N$ reference models described by (\ref{eqn_dgsmrac8}), and assuming that controller $C_k$ of $S_k$ can generate an input $v_k(t)$ such that all the signals in the system are bounded and $lim_{t \rightarrow \infty} \left\|\delta x_k(t)-\delta x_{m,k}(t) \right\|=0$. Since the effect of the interactions of subsystems on each other is bounded, we can use the following adaptive state feedback controller for each subsystem
\begin{equation}\label{eqn_dgsmrac10}
C_k:~ v_k(t) = \hat{K}_k^{\mathsf{T}}(t) \delta x_k(t),\\
\end{equation}
with $\hat{K}_k(t) \in \Re^{n_k}$ is the time-varying estimate of the nominal controller parameters $K_k^{*}(t)$.
\medskip
\begin{ass} \label{dgsa_ass3}
For each subsystem $S_k$, there exists an ideal gain matrix $K_k^{*\mathsf{T}}(\alpha(t))=[0,~0,~ k_{i,k}^{*\mathsf{T}}(\alpha(t))]$, that results in perfect matching between the reference model (\ref{eqn_dgsmrac8}) and the plant (\ref{eqn_dgsmrac6}) such that
\begin{equation}\label{eqn_dgsmrac101}
\begin{array}{l}
A_{m,k}(\alpha(t))=A_k(\alpha(t))+b_k K_k^{*\mathsf{T}}(\alpha(t)), ~~ \forall \alpha \in \Omega.
\end{array}
\end{equation}
\end{ass}
\medskip
\begin{rmk}\label{dgsa_rmk_a3}
The feasibility of this assumption has already been verified in \cite{PhDThesis-pakmehr-2013}, for gas turbine engine applications which we consider as the main application of this work.  For other systems, modeling and numerical studies are needed for such verification.
\end{rmk}
\begin{ass} \label{dgsa_ass4}
Let $K_k^*(\alpha(t)) \in \theta_k$ for all $\alpha \in \Omega$, where $\theta_k$ is a known convex compact set. We also assume that $K_k^*(\alpha(t))$ is continuously differentiable, and the derivative is uniformly bounded, $||\dot{K}_k^*(\alpha(t))|| \leq \bar{d}_k < \infty$ for all $\alpha \in \Omega$.
\end{ass}
\medskip
\begin{rmk}\label{dgsa_rmk_a2}
$\alpha(t)$ is defined to be $\alpha(t)=||y(t)||=||x^p(t)||$; since it is a function of endogenous variables (i.e., the plant states), its boundedness is guaranteed by boundedness of the plant states. As a result, its derivative ($\dot{\alpha}(t)=\frac{{x^{p}(t)}^\mathsf{T} \dot{x}^p(t)}{||x^p(t)||}$) is also bounded. More details can be found in \cite{gainsched-shamma-1988, research-rugh-2000, PhDThesis-pakmehr-2013, GSstability-pakmehr-2013, overview-shamma-2012}.
\end{rmk}
\begin{rmk}\label{dgsa_rmk_a21}
Compact set $\theta_k$ can be obtained by extensive numerical simulation studies of the system that the controller is being designed for. Smoothness, continuity, and differentiability of $K_k(\alpha(t))$, and also uniform boundedness of $\dot{K}_k(\alpha(t))$, can be guaranteed, by using proper design and computation process for $K_k(\alpha(t))$ (see Remark \ref{rmk_21}).
\end{rmk}
With adaptive controller (\ref{eqn_dgsmrac10}), the closed-loop form of subsystem $S_k$ becomes
\begin{equation}\label{eqn_dgsmrac12}
\begin{array}{l}
\displaystyle \delta \dot{x}_k(t) = A_{m,k}(\alpha(t)) \delta x_k(t) + b_k \tilde{K}_k^{\mathsf{T}}(t) \delta x_k(t) + b_{r_k} \delta r_k(t)+ h_k(\delta x_q(t),\alpha(t)),
\end{array}
\end{equation}
where $\tilde{K}_k(t)=K_k(t)-K_k^{*}(t)$.
The error equation in terms of state tracking error $e_k(t)=\delta x_k(t)- \delta x_{m,k}(t)$ and controller parameters is
\begin{equation}\label{eqn_dgsmrac13}
\begin{array}{l}
\dot{e}_k(t) = A_{m,k}(\alpha(t)) e_k(t)+b_k \tilde{K}_k^{\mathsf{T}}(t)\delta x_k(t) + h_k(\delta x_q(t),\alpha(t)).
\end{array}
\end{equation}
Based on the error model (\ref{eqn_dgsmrac13}), adaptive laws are presented using the Lyapunov design method. Here we consider the case that for each subsystem a single quadratic Lyapunov function exists for the error model (\ref{eqn_dgsmrac13}).
If for the Hurwitz matrices $A_{m,k}(\alpha(t))$, for each subsystem, for all $\alpha \in \Omega$, there exist a single Lyapunov matrix $P_k = P_k^{\mathsf{T}} > 0$, and a positive definite matrix $Q_k$ such that
\begin{equation}\label{eqn_dgsmrac14}
\displaystyle P_k A_{m,k}(\alpha(t)) + A_{m,k}^{\mathsf{T}}(\alpha(t)) P_k \leq -Q_k,  ~~~ \forall \alpha \in \Omega
\end{equation}
we use the following adaptive law:
\begin{equation}\label{eqn_dgsmrac15}
\displaystyle \dot{\hat{K}}_k(t) = \text{Proj}_{\Gamma} \left( \hat{K}_k(t), - \delta x_k(t) e_k^{\mathsf{T}}(t) P_k b_k\right),
\end{equation}
where $\Gamma_k=\Gamma_k^{\mathsf{T}}$. A visualization of the decentralized gain scheduled model reference adaptive control architecture is given in Figure \ref{schem_cont_subsystem}.
\begin{figure}[!ht]
\centering
\resizebox{4.5in}{!}{\includegraphics{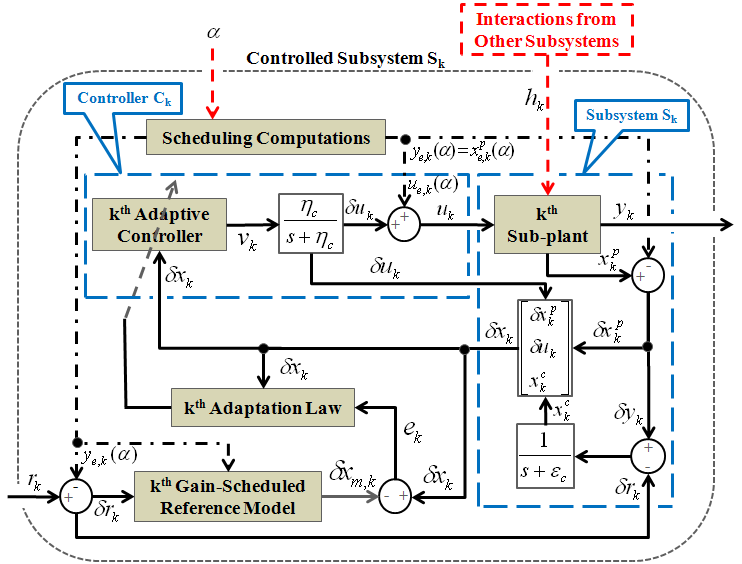}}
\caption{Architecture of decentralized adaptive control illustrated for a subsystem $S_k$}\label{schem_cont_subsystem}
\end{figure}

\begin{thm} \label{dgsa_thm5}
Consider the system $S$ consisting of $N$ interconnected subsystems $S_k$ described by (\ref{eqn_dgsmrac6}) subject to Assumption \ref{dgsa_ass2}. Consider, in addition for subsystems $S_k$, the adaptive control laws $C_k$ defined in (\ref{eqn_dgsmrac10}), with adaptive laws defined in (\ref{eqn_dgsmrac15}) subject to Assumptions \ref{dgsa_ass3} and \ref{dgsa_ass4}. Then the error signals $e_k(t)$ are uniformly ultimately bounded (UUB) for all $k=1,2,...,N$.
\end{thm}

\begin{proof}
For subsystem $S_k$ consider the Lyapunov function candidate as
\begin{equation}\label{eqn_dgsmrac16}
\displaystyle V_k(e_k(t),\tilde{K}_k(t)) = e_k^{\mathsf{T}}(t) P_k e_k(t)+ \tilde{K}_k^{\mathsf{T}}(t) \Gamma_k^{-1} \tilde{K}_{k}(t),
\end{equation}
whose time derivative along (\ref{eqn_dgsmrac13}) and (\ref{eqn_dgsmrac15}) is
\begin{equation}\label{eqn_dgsmrac17}
\begin{array}{l}
\displaystyle \dot{V}_k(.) = e_k^{\mathsf{T}}(t) \left( P_k A_{m,k}(\alpha(t))+A_{m,k}^{\mathsf{T}}(\alpha(t))P_k \right) e_k(t)\\[5pt]
~~~~~~~~~ +2e_k^{\mathsf{T}}(t) P_k b_k \tilde{K}_k^{\mathsf{T}}(t) \delta x_k(t) + 2 \left(\tilde{K}_k^{\mathsf{T}}(t) \Gamma_k^{-1} \dot{\tilde{K}}^{k}(t) \right)\\[5pt]
~~~~~~~~~ +2e_k^{\mathsf{T}}(t) P_k h_k(\delta x_q(t),\alpha(t)).
\end{array}
\end{equation}
Using Lemma \ref{gsa_lem1}, and knowing that for scalars $a^\mathsf{T} b=b a^\mathsf{T}$, and letting $Y_{K,k}(t)=-\delta x_k(t) e_k^\mathsf{T}(t) P_k b_k$, and knowing $\dot{\tilde{K}}_k(t)=\dot{\hat{K}}_k(t)-\dot{K}_k^*(t)$ leads to
\begin{equation}\label{eqn_dgsmrac18}
\begin{array}{l}
\displaystyle \dot{V}_k(.) \leq  -e_k^\mathsf{T}(t) Q_k e_k(t) \\[5pt]
~~~~~~~~~~ + 2 \left( \tilde{K}_k^\mathsf{T}(t) \left[ \Gamma_k^{-1} \text{Proj}_{\Gamma}(\hat{K}_k(t), Y_{K,k}(t))- Y_{K,k}(t) \right] \right) \\[5pt]
~~~~~~~~~~ - 2 \left( \tilde{K}_k(t)^\mathsf{T} \Gamma_k^{-1} \dot{K}_k^*(t) \right) + 2e_k^{\mathsf{T}}(t) P_k h_k(\delta x_q(t),\alpha(t)).
\end{array}
\end{equation}
Using Lemma \ref{gsa-lem10}
\begin{equation}\label{eqn_dgsmrac19}
\begin{array}{l}
\displaystyle \dot{V}_k(.) \leq -e_k^\mathsf{T}(t) Q_k e(t) - 2 \left( \tilde{K}_k^\mathsf{T}(t) \Gamma^{-1} \dot{K}_k^*(t) \right) +2e_k^{\mathsf{T}}(t) P_k h_k(\delta x_q(t),\alpha(t)).
\end{array}
\end{equation}
From Assumption \ref{dgsa_ass2}, knowing $\delta x_q(t)=e_q(t)+ \delta x_{m,q}(t)$, and letting $\bar{x}_{m,k}=sup_t(\sum \limits^{N}_{q=1, q\neq k} c_{kq} ||\delta x_{m,q}||)$, we have
\begin{equation}\label{eqn_dgsmrac20}
\displaystyle ||h_k(\delta x_q(t),\alpha(t))|| \leq \sum \limits_{q=1, q\neq k}^{N} \left [ c_{kq} ||e_q(t)|| \right]+\bar{x}_{m,k}.
\end{equation}
Furthermore, using (\ref{eqn_dgsmrac20}) in the last term of (\ref{eqn_dgsmrac19}) results in
\begin{equation}\label{eqn_dgsmrac21}
\begin{array}{l}
\displaystyle |2e_k^{\mathsf{T}}(t) P_k h_k(\delta x_q(t),\alpha(t))| \leq 2\lambda_{\max}(P_k) ||e_k(t)|| ||h_k(\delta x_q(t),\alpha(t))|| \\[5pt]
~~~~~~~~~~~~~~~~~~~~~~~~~~~~~~~~~~ \leq 2\lambda_{\max}(P_k) ||e_k(t)|| \left(\bar{x}_{m,k}+ \sum \limits_{q=1, q\neq k}^{N} \left [ c_{kq} ||e_q(t)|| \right] \right) .
\end{array}
\end{equation}
Notice that, using Assumption \ref{dgsa_ass4}, we obtain
\begin{equation}\label{eqn_dgsmrac22}
\begin{array}{l}
\displaystyle ||\tilde{K}_k^\mathsf{T}(t) \Gamma_k^{-1} \dot{K}_k^*(t)|| \leq ||\Gamma_k^{-1}|| \max \limits_{K^*_k \in \theta_k} ||K^*_k|| \bar{d}_{k}, ~~ \forall t \geq 0.
\end{array}
\end{equation}
The following upper bound on Lyapunov function derivative for $k$th subsystem, can be found
\begin{equation}\label{eqn_dgsmrac23}
\begin{array}{l}
\displaystyle \dot{V}_k(.) \leq - \bar{\lambda}_k ||e_k(t)||^2 + \bar{\rho}_k ||e_k(t)|| \sum \limits_{q=1, q\neq k}^{N} c_{kq} ||e_q(t)||  + \bar{\xi}_k ||e_k(t)||+ \bar{\psi}_k,
\end{array}
\end{equation}
where $\bar{\lambda}_k:=(\lambda_{min}(Q_k))$, $\bar{\rho}_k:= 2\lambda_{\max}(P_k)$, $\bar{\xi}_k:= 2\lambda_{\max}(P_k) \bar{x}_{m,k}$ and $\bar{\psi}_k:= 2 ||\Gamma_k^{-1}|| \max \limits_{K^*_k \in \theta_k} ||K^*_k|| \bar{d}_{k}$ are all positive constants. Now choosing $V(.)=\sum \limits^{N}_{k=1}V_k(.)$, to show the stability of the whole system $S$, the Lyapunov function derivative for the whole system $S$ is
\begin{equation}\label{eqn_dgsmrac24}
\begin{array}{l}
\displaystyle \dot{V}(.)=\sum \limits^{N}_{k=1}\dot{V}_k(.) \\ [3pt]
~~~~~~~ \leq  \sum \limits^{N}_{k=1} ( - \bar{\lambda}_k ||e_k(t)||^2 + \bar{\rho}_k ||e_k(t)|| \sum \limits_{q=1, q\neq k}^{N} c_{kq} ||e_q(t)|| + \bar{\xi}_k ||e_k(t)||+ \bar{\psi}_k ).
\end{array}
\end{equation}
Letting $\psi:= \sum \limits^{N}_{k=1} \bar{\psi}_k$, and defining the following vectors and matrices
\begin{equation}\label{eqn_dgsmrac25}
\begin{array}{l}
\bar{e}(t):=[||e_1(t)||,...,||e_N(t)||]^{\mathsf{T}}, ~~ \xi:=[\bar{\xi}_1, ..., \bar{\xi}_N]^{\mathsf{T}},\\ [3pt]
\Phi:= \left[
       \begin{array}{ccccc}
           0                   &  \bar{\rho}_1 c_{12}   &  .  & \bar{\rho}_1 c_{1q} &  \bar{\rho}_1 c_{1N}  \\
           \bar{\rho}_2 c_{21} &  0     &  . &   . &   \bar{\rho}_2 c_{2N}   \\
           .                   &  .     & 0    & .  &  .   \\
           \bar{\rho}_k c_{k1} &  .     &  \bar{\rho}_k c_{kq} &   . &   \bar{\rho}_k c_{kN}   \\
           \bar{\rho}_N c_{N1} &  .     & \bar{\rho}_N c_{Nq}    & .      & 0   \\
       \end{array}
    \right], \\ [5pt]
\Lambda:=diag([\bar{\lambda}_1,...,\bar{\lambda}_N]), ~~ \Pi:= \Lambda - \Phi,
\end{array}
\end{equation}
the upper bound on Lyapunov function derivative becomes
\begin{equation}\label{eqn_dgsmrac26}
\begin{array}{l}
\displaystyle \dot{V}(.) \leq - \bar{e}(t)^{\mathsf{T}} \Pi \bar{e}(t) + \xi^{\mathsf{T}} \bar{e}(t) + \psi \\ [3pt]
~~~~~~~ \leq -\lambda_{\min}(\Pi) ||\bar{e}(t)||^2 + ||\xi|| ||\bar{e}(t)|| + \psi.
\end{array}
\end{equation}
By proper selection of $Q_k$ for all $k=1,...,N$, we can make sure that $\lambda_{\min}(\Pi) >0$. Having
\begin{equation}\label{eqn_dgsmrac27}
\begin{array}{l}
\displaystyle ||\bar{e}(t)|| >  \frac{||\xi||+\sqrt{||\xi||^2+4\lambda_{\min}(\Pi)\psi}}{2 \lambda_{\min}(\Pi)},
\end{array}
\end{equation}
renders $\dot{V}(.) < 0$. Hence $e_k(t)$ is UUB for all $k=1,...,N$.
\end{proof}

\subsection{Turboshaft Engine Example}
We apply the developed decentralized controller to a high fidelity physics-based model of JetCat SPT5 turboshaft engine driving a variable pitch propeller developed in \cite{fitzgerald-model-2013, pakmehr-decentmodel-2011}. To show the stability of the closed-loop reference model for each subsystem, 40 different (30 equilibrium, and 10 non-equilibrium) linearizations are used, to solve inequality (\ref{eqn_gs115}); the inequality is solved in Matlab with using YALMIP \cite{YALMIP-lofberg-2004} and SeDuMi \cite{sedumi-Sturm-2001} packages. The numerical values for $Q_{Co}$, $Q_{Pr}$, and the matrices $P_{Co}$ and $P_{Pr}$ for the subsystems are
\begin{eqnarray} \label{eqn_dgsmrac_sim6}
P_{Co} = \left[
    \begin{array}{ccc}
            4.9034  &  0.9895 &  -0.6234 \\
            0.9895  &  1.7716 &  -0.1078 \\
           -0.6234  & -0.1078 &   3.4583
    \end{array}
    \right], \\[3pt]
    P_{Pr} = \left[
    \begin{array}{ccc}
         1.9015  &  0.0513  &  0.1912 \\
         0.0513  &  0.3882  & -0.0553 \\
         0.1912  & -0.0553  &  1.0811
    \end{array}
    \right],
\end{eqnarray}
where the condition numbers are $\kappa(P_{Co})=3.6384$ and $\kappa(P_{Pr})=5.1066$. $Q_{Co}=0.1 \times I_3$ and $Q_{Pr}=0.1 \times I_3$. These simulations include the control of the nominal model (NomEng), and also control of the engine with a new core (NewCore). These decentralized adaptive control case studies, simulate the engine acceleration from the idle thrust to the cruise condition and then its deceleration back to the idle condition in a stable manner, with proper tracking performance. The initial conditions for each subsystems, and the numerical values for the corresponding adaptive controllers are $x_{Co}(0)=x_{m,Co}(0) = \left[0.295 , ~0.145 , ~0 \right]^\mathsf{T}$, $x_{Pr}(0)=x_{m,Pr}(0) = \left[0.161 , ~16 , ~0 \right]^\mathsf{T}$, $\hat{K}_{Co}(0)= \left[ 0 , ~0 , ~-0.49 \right]^\mathsf{T}$, $\hat{K}_{Pr}(0)= \left[ 0 , ~0 , ~-0.49 \right]^\mathsf{T}$, $   \Gamma_{Co} =\text{diag}([40,~40,~40])$, $\Gamma_{Pr} =\text{diag}([30,~30,~30])$, $K^*_{Co} \in \theta_{k_{Co}}= \left[ [-2, ~ 0] , [-2, ~ 0] , [-2, ~ 0] \right]^\mathsf{T}$, and $K^*_{Pr} \in \theta_{k_{Pr}}= \left[ [-2, ~ 0] , [-2, ~ 0] , [-2, ~ 0] \right]^\mathsf{T}$. To simulate a new engine core, we assumed the high pressure spool inertia is $I_{hps, new}=0.8I_{hps, nom}$, where $I_{hps, nom}=4 \times 10^{-5}~(kg.m^2)$. Simulation results for this scenario are shown in figures \ref{fig_dgsmrac_02} to \ref{fig_dgsmrac_31}.
\begin{figure}[!ht]
\centering
\begin{minipage}[b]{0.49\textwidth}
\centering
\includegraphics[width=0.99\textwidth]{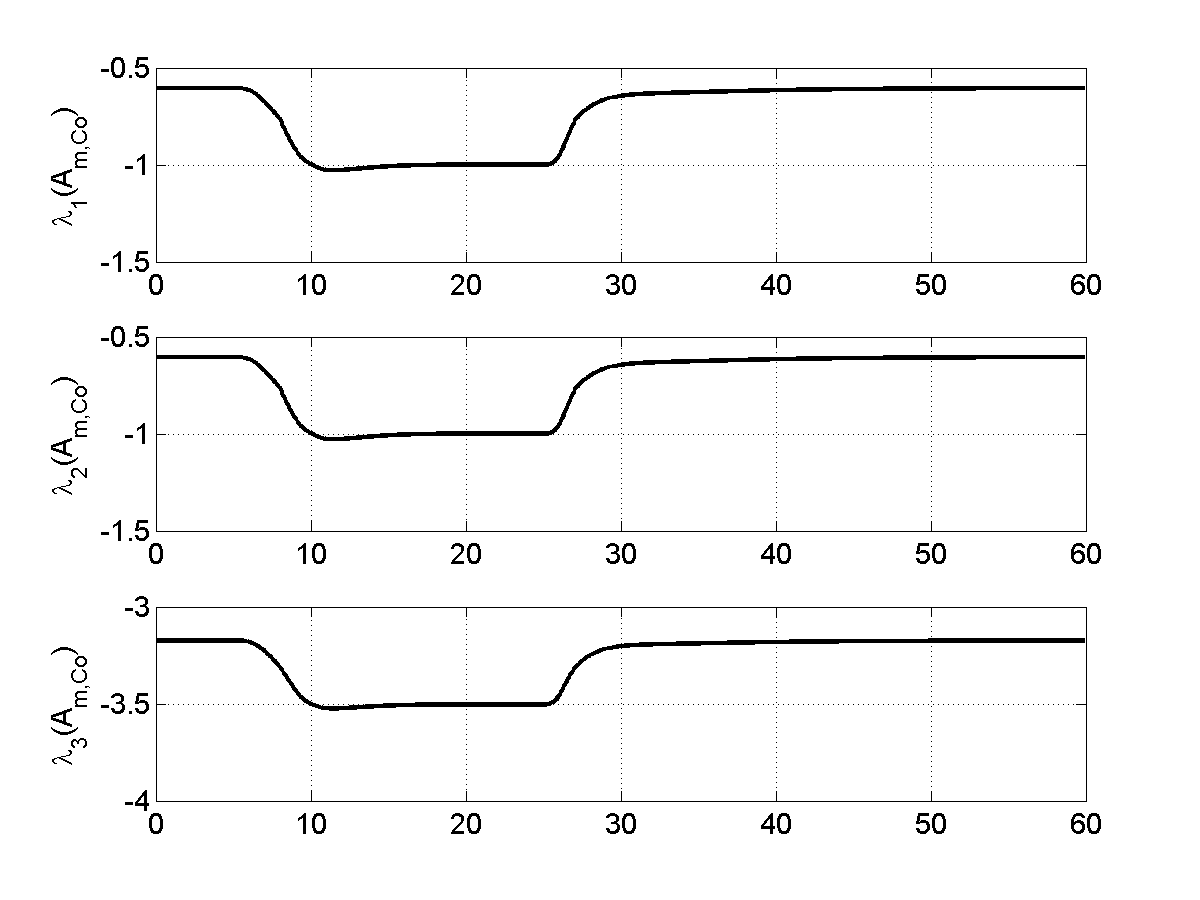}
\caption{Engine core subsys ref. model eigenvalues ($\lambda[A_{m,Co}(\alpha(t))]$)} \label{fig_dgsmrac_02}
\end{minipage}
\hfill
\begin{minipage}[b]{0.49\textwidth}
\centering
\includegraphics[width=0.99\textwidth]{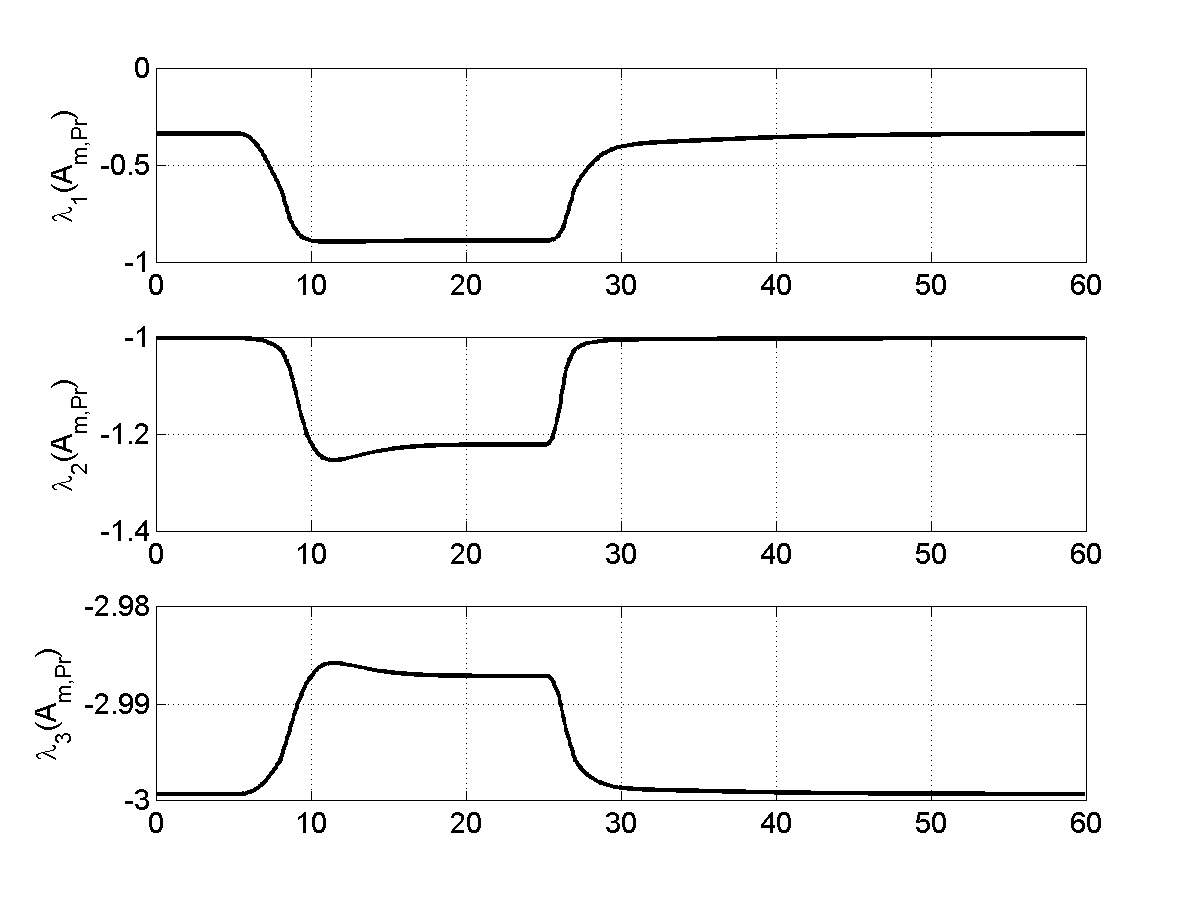}
\caption{Engine prop subsys ref. model eigenvalues ($\lambda[A_{m,Pr}(\alpha(t))]$)} \label{fig_dgsmrac_04}
\end{minipage}
\end{figure}

Figures \ref{fig_dgsmrac_02} and \ref{fig_dgsmrac_04}, show the history of the desired reference system matrix eigenvalues for the core $\lambda[A_{m,Co}(\alpha(t))]$, and prop $\lambda[A_{m,Pr}(\alpha(t))]$ subsystems. As it is apparent, all the eigenvalues remain negative with the time change of the scheduling parameter $\alpha$.

\begin{figure}[!ht]
\centering
\begin{minipage}[b]{0.49\textwidth}
\centering
\includegraphics[width=0.99\textwidth]{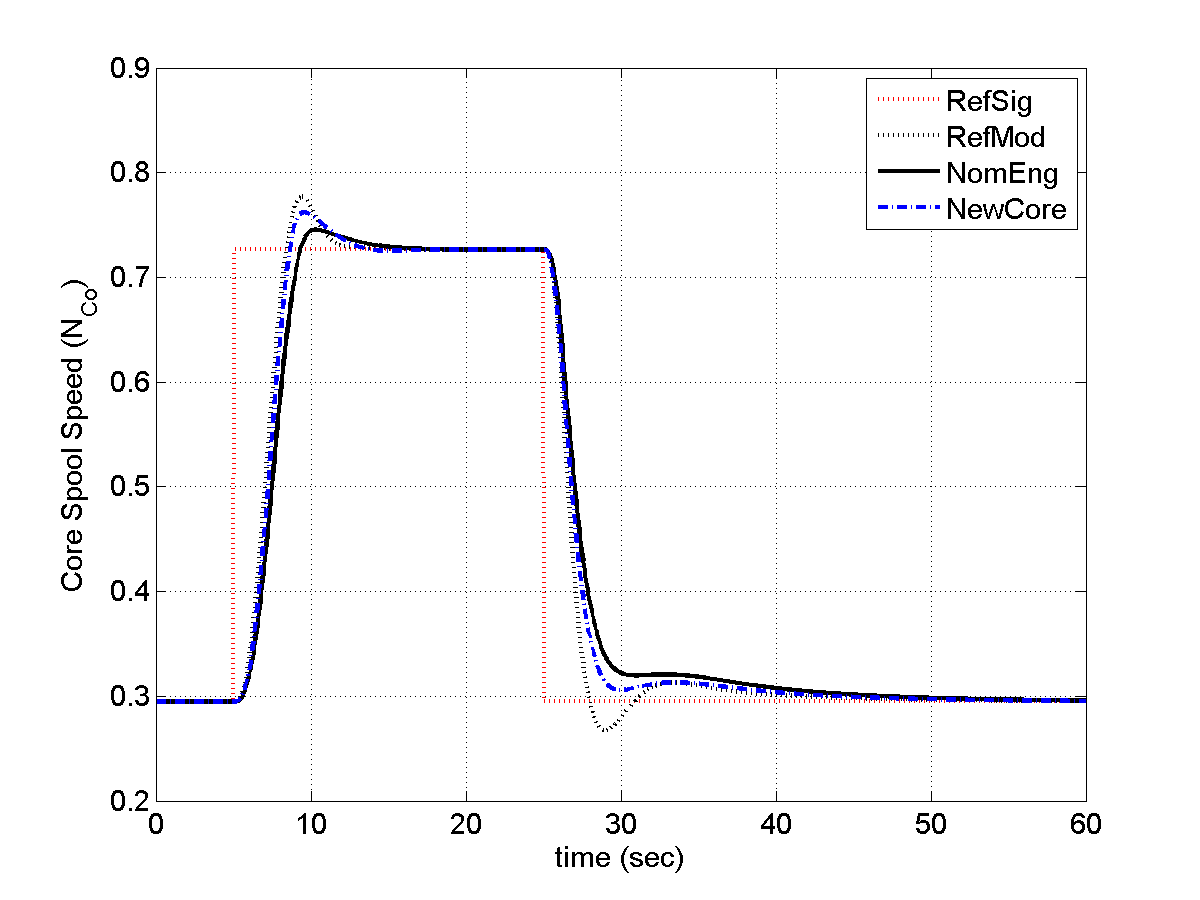}
\caption{Core spool speed and its ref. signal}\label{fig_dgsmrac_21}
\end{minipage}
\hfill
\begin{minipage}[b]{0.49\textwidth}
\centering
\includegraphics[width=0.99\textwidth]{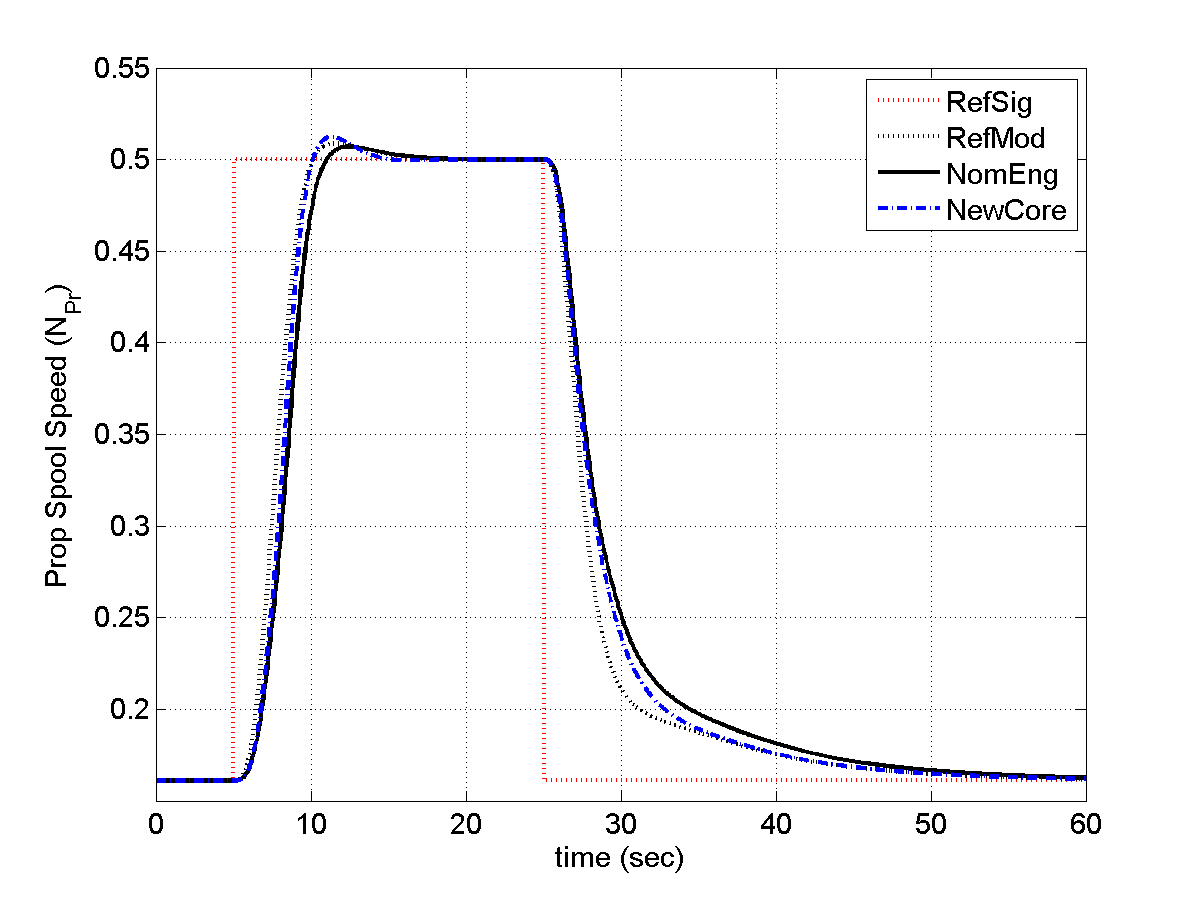}
\caption{Prop spool speed and its ref. signal}\label{fig_dgsmrac_22}
\end{minipage}
\end{figure}

Figures \ref{fig_dgsmrac_21} and \ref{fig_dgsmrac_22}, show the output of the core subsystem ($x^p_{Co}(t)$) and prop subsystem ($x^p_{Pr}(t)$) tracking their reference signals. Figure \ref{fig_dgsmrac_26}, shows the evolution of the control inputs to the augmented engine core ($v_{Co}(t)$), and prop ($v_{Pr}(t)$) subsystems, each element is corresponding to one of the control inputs to the original subsystem.

\begin{figure}[!ht]
\centering
\begin{minipage}[b]{0.49\textwidth}
\centering
\includegraphics[width=0.99\textwidth]{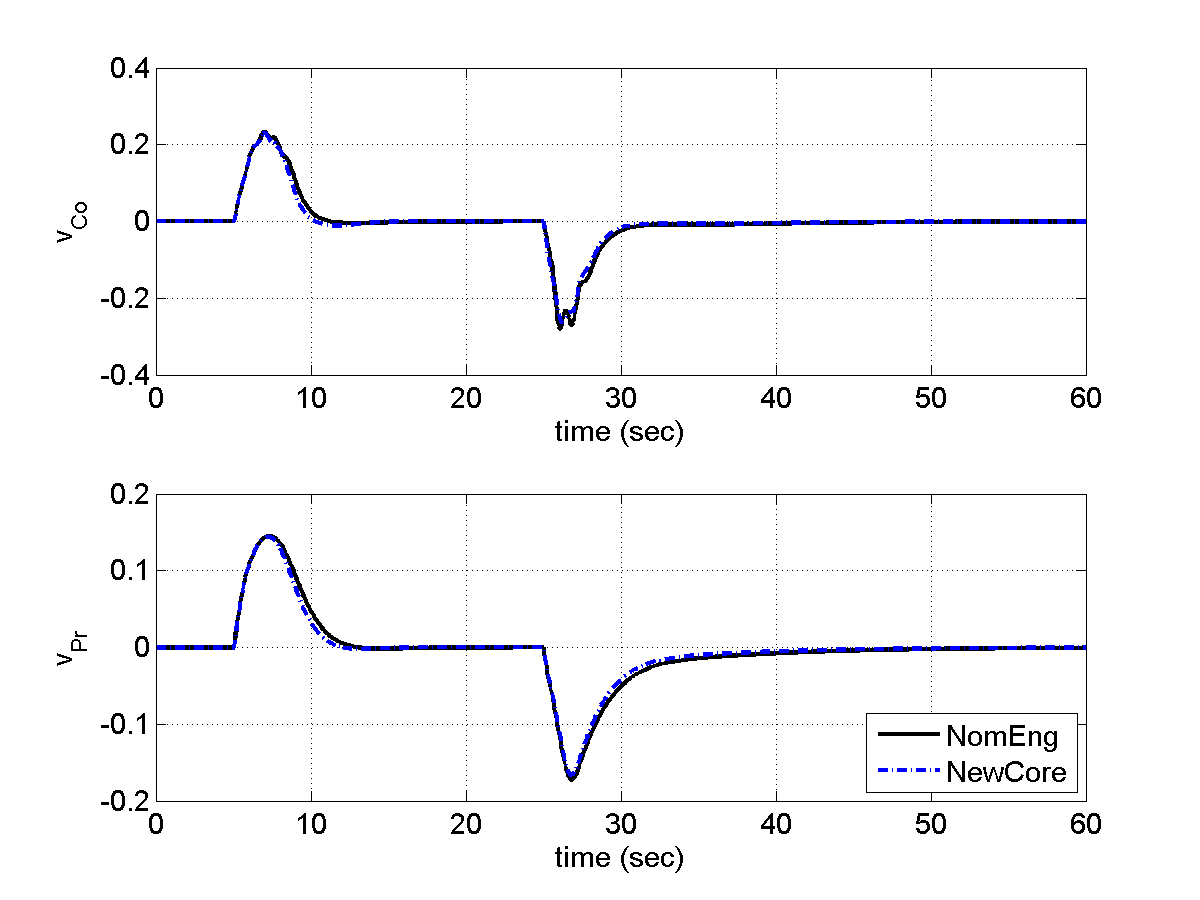}
\caption{Control inputs to the augmented engine core ($v_{Co}(t)$), and prop ($v_{Pr}(t)$) subsystems}\label{fig_dgsmrac_26}
\end{minipage}
\hfill
\begin{minipage}[b]{0.49\textwidth}
\centering
\includegraphics[width=0.99\textwidth]{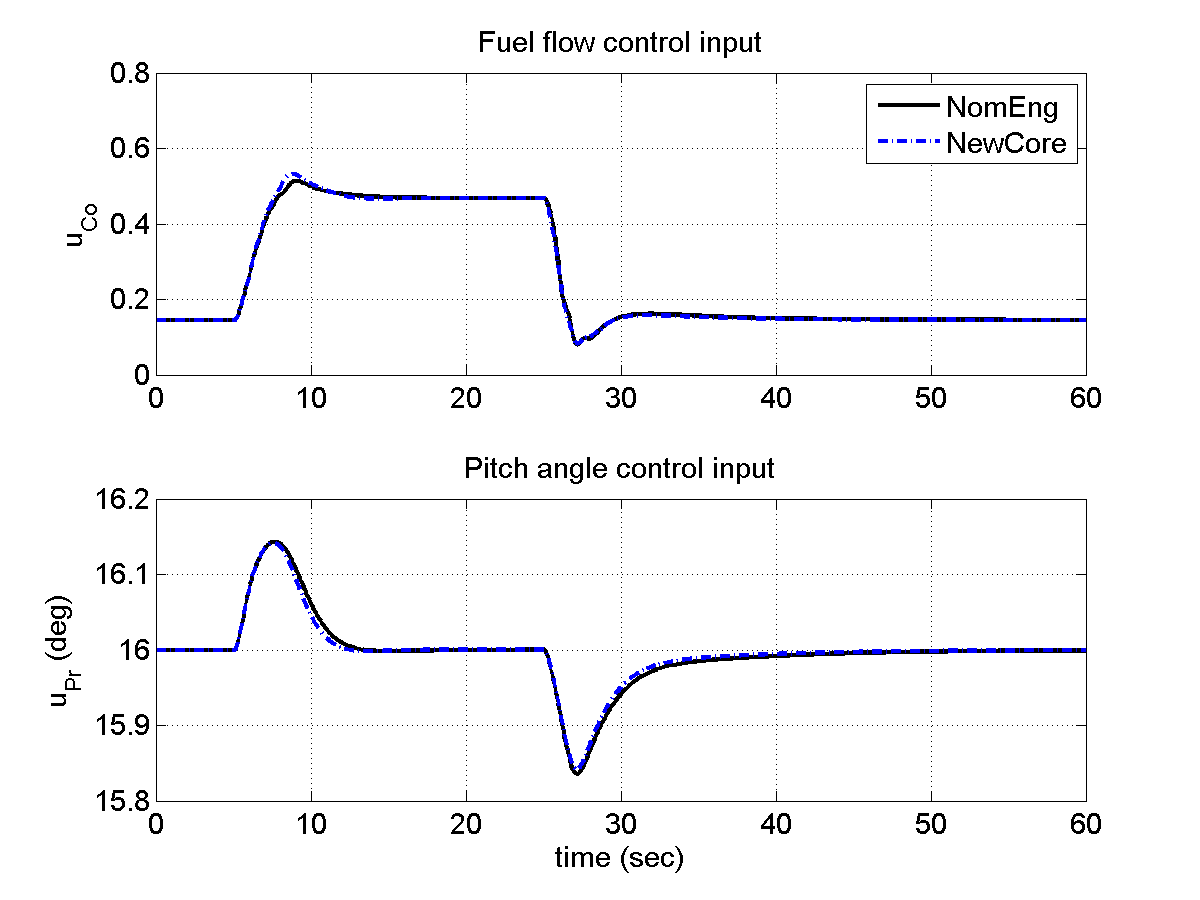}
\caption{Fuel ($u_{Co}(t)$) and prop pitch angle ($u_{Pr}(t)$) control inputs }\label{fig_dgsmrac_28}
\end{minipage}
\end{figure}

\begin{figure}[!ht]
\centering
\includegraphics[width=0.55\textwidth]{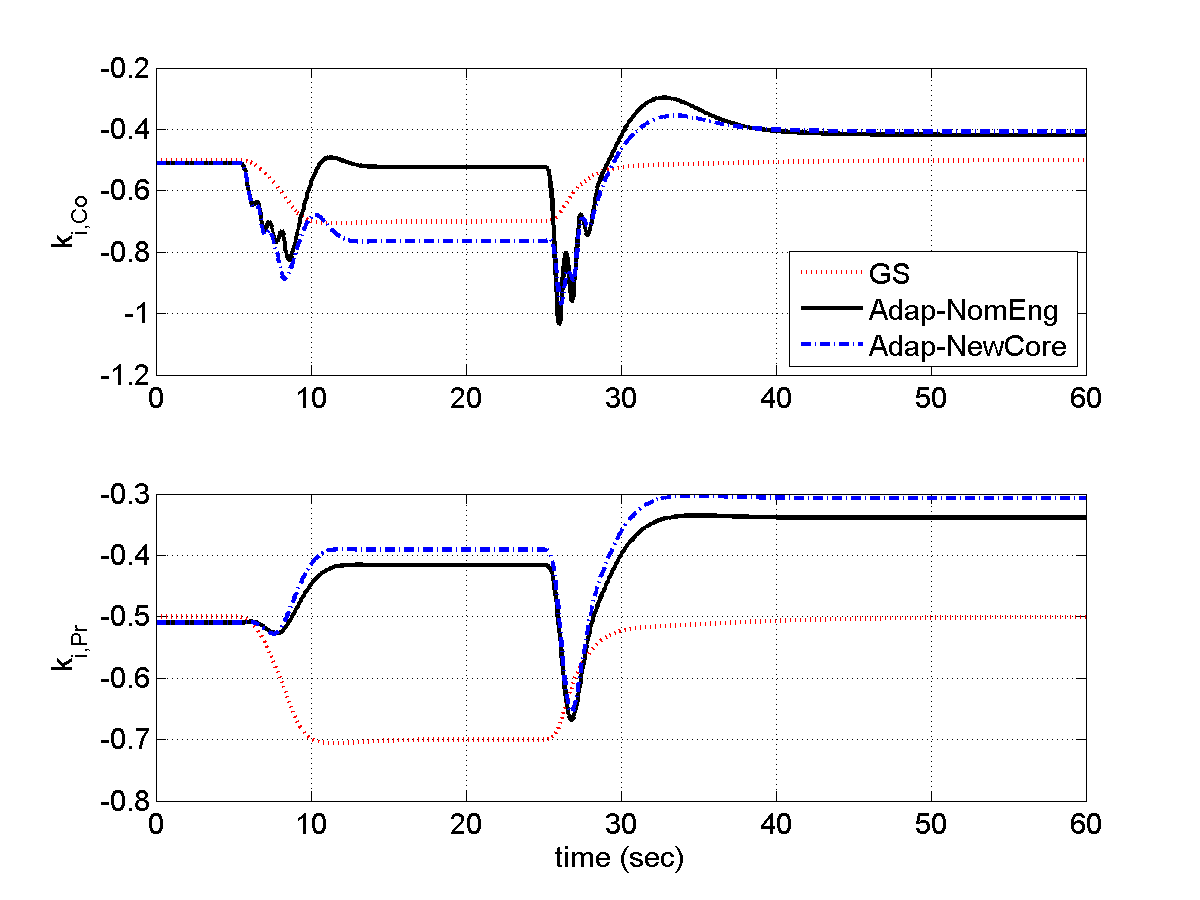}
\caption{Gain scheduled and adaptive integral gain for the engine core ($k_{i,Co}(\alpha(t))$, $\hat{k}_{i,Co}(t)$), and prop ($k_{i,Pr}(\alpha(t))$, $\hat{k}_{i,Pr}(t)$) subsystems}\label{fig_dgsmrac_29}
\end{figure}

Figure \ref{fig_dgsmrac_28}, shows the histories of fuel flow ($u_{Co}(t)$) and propeller pitch angle ($u_{Pr}(t)$) as the control inputs to each subsystem. Figures \ref{fig_dgsmrac_29}, shows gain scheduled and adaptive integral gains for the engine core ($k_{i,Co}(\alpha(t))$, $\hat{k}_{i,Co}(t)$), and prop ($k_{i,Pr}(\alpha(t))$, $\hat{k}_{i,Pr}(t)$) subsystems. The gain scheduled control gains have been obtained by interpolation using the predesigned indexed family of fixed-gain controllers, and each controller corresponds to one equilibrium point of the engine. $\hat{k}_{i,Co}(t)$ and $\hat{k}_{i,Pr}(t)$ are generated using adaptive laws designed for each subsystem.
\begin{figure}[!ht]
\centering
\begin{minipage}[b]{0.49\textwidth}
\centering
\includegraphics[width=0.99\textwidth]{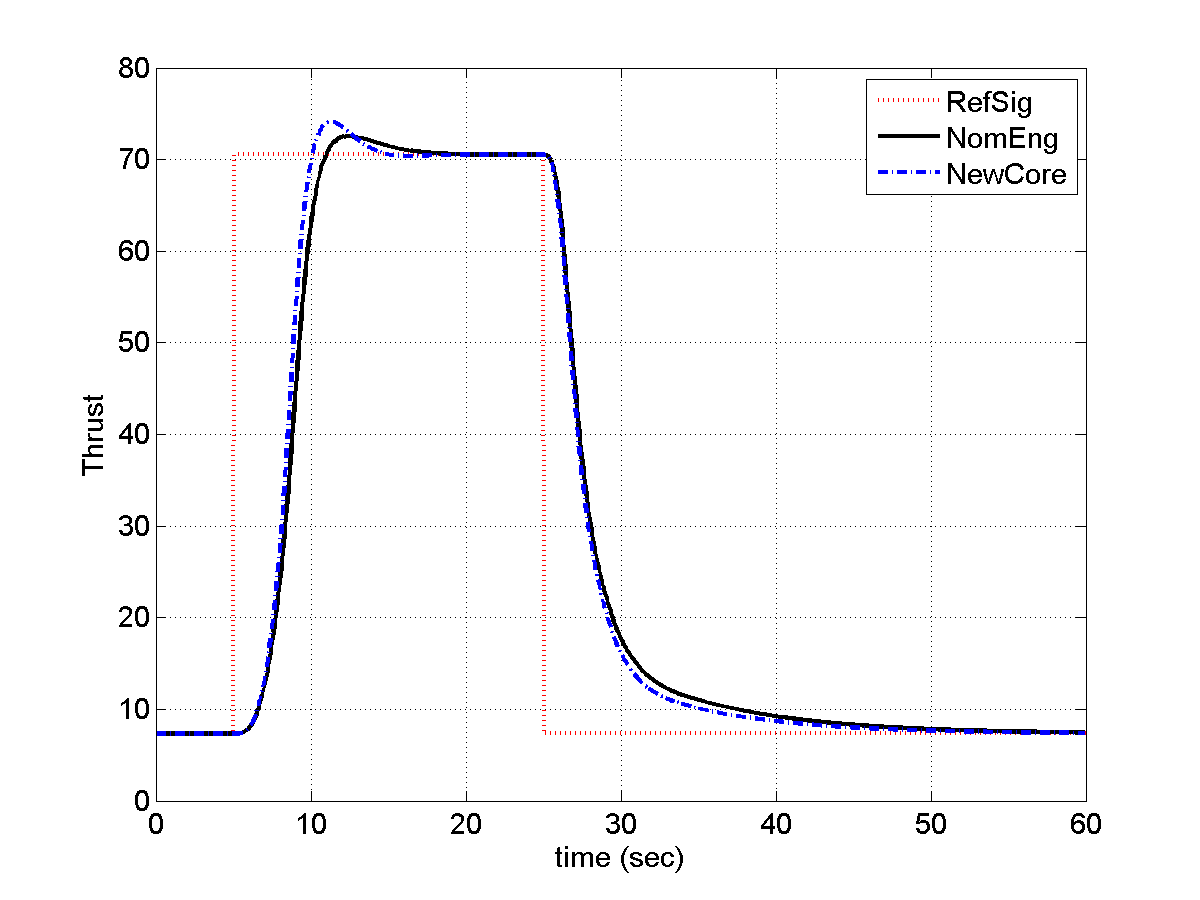}
\caption{Thrust and its ref. signal}\label{fig_dgsmrac_30}
\end{minipage}
\hfill
\begin{minipage}[b]{0.49\textwidth}
\centering
\includegraphics[width=0.99\textwidth]{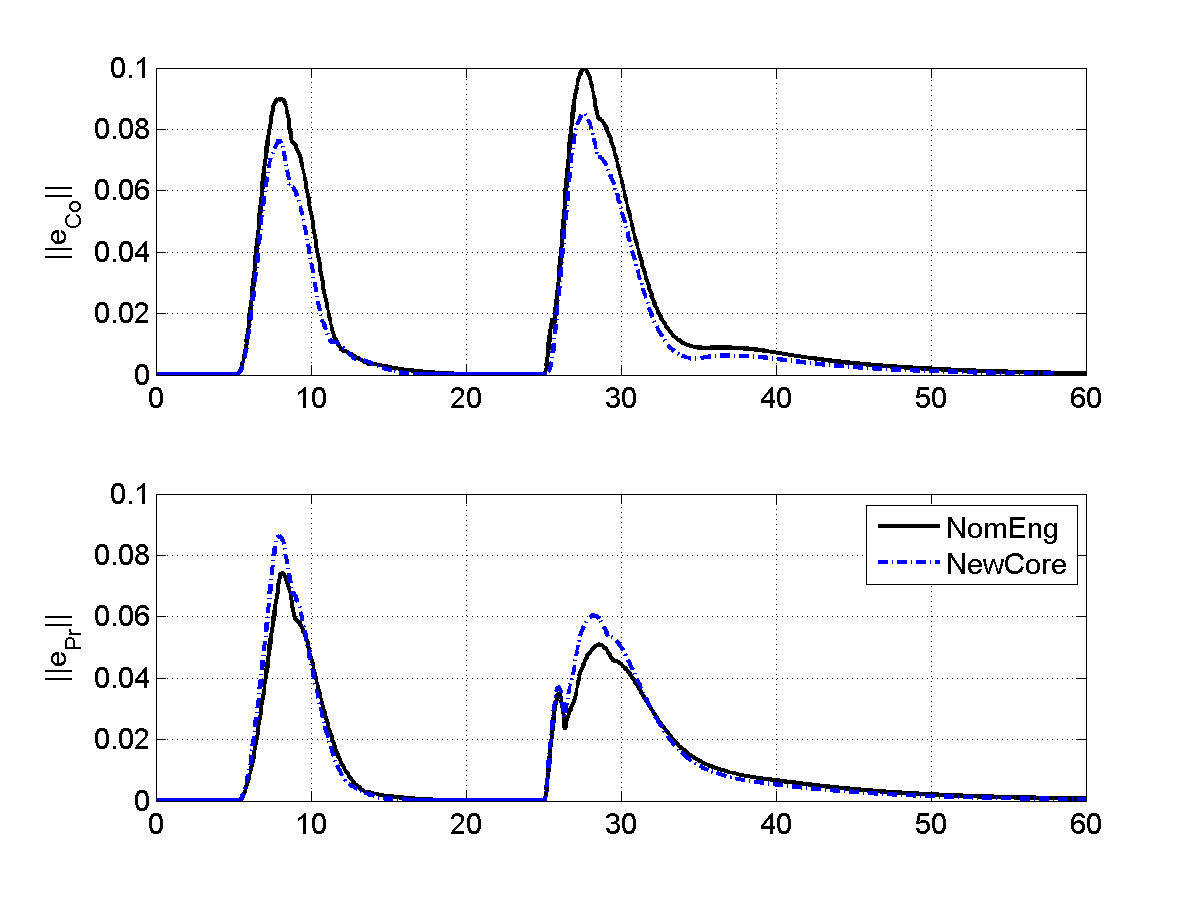}
\caption{Norm of the error signals for the engine core $||e_{Co}(t)||$, and prop $||e_{Pr}(t)||$ subsystems}\label{fig_dgsmrac_31}
\end{minipage}
\end{figure}

Figure \ref{fig_dgsmrac_30}, shows the history of thrust and it is following its reference command from idle to cruise condition and then back to the idle for standard day, sea level condition. Figure \ref{fig_dgsmrac_31}, shows the evolution of the infinity norm of the errors $||e_{Co}(t)||$ and $||e_{Pr}(t)||$. The smallness of the errors suggest that the  subsystems closely track the desired reference trajectories. It also verifies the Assumption \ref{dgsa_ass2}, which is on the boundedness of the coupling effects of the subsystems on each other, for the gas turbine engine control example.

\newpage
\subsection{Conclusions}
Gain scheduled reference models were developed for each subsystem of the decentralized architecture. Using convex optimization tools, a single quadratic Lyapunov function was computed for each subsystem, which guaranteed the stability of the gain scheduled gas turbine engine core and prop reference models. Rigorous stability analysis was done by proving the uniform ultimate boundedness of the error signals for all the subsystems. Sufficient conditions for uniform ultimate boundedness of the entire system were derived. Through the simulation based on a physics-based nonlinear model of a JetCat SPT5 turboshaft engine with a new core, it was demonstrated the proposed decentralized adaptive controllers track their reference models in each subsystem for the entire flight envelope of the engine.

\section*{Acknowledgments}
This material is based upon the work supported by the Air Force Research Laboratory (AFRL) and also the National Science Foundation (NSF).


\bibliographystyle{plain}



\end{document}